\documentclass{amsart}
\usepackage{amssymb
,amsthm
,amsmath
,amscd
,mathtools
}
\usepackage[all]{xy}
\usepackage[top=30truemm,bottom=30truemm,left=25truemm,right=25truemm]{geometry}

\newcommand{\A}{\mathbb{A}}
\newcommand{\Z}{\mathbb{Z}}
\newcommand{\R}{\mathbb{R}}
\newcommand{\C}{\mathbb{C}}
\newcommand{\F}{\mathbb{F}}
\renewcommand{\H}{\mathbb{H}}

\renewcommand{\AA}{\mathcal{A}}
\newcommand{\LL}{\mathcal{L}}
\newcommand{\HH}{\mathcal{H}}
\newcommand{\VV}{\mathcal{V}}
\newcommand{\TT}{\mathcal{T}}
\newcommand{\Sc}{\mathcal{S}}

\newcommand{\oo}{\mathfrak{o}}
\newcommand{\z}{\mathfrak{z}}

\newcommand{\cent}{\mathrm{Cent}}
\newcommand{\disc}{\mathrm{disc}}
\newcommand{\cusp}{\mathrm{cusp}}
\newcommand{\ab}{\mathrm{ab}}
\newcommand{\id}{\mathrm{id}}
\newcommand{\im}{\mathrm{Im}}
\newcommand{\re}{\mathrm{Re}}
\newcommand{\Irr}{\mathrm{Irr}}
\newcommand{\unit}{\mathrm{unit}}
\newcommand{\temp}{\mathrm{temp}}
\newcommand{\gen}{\mathrm{gen}}
\newcommand{\Ind}{\mathrm{Ind}}
\newcommand{\Res}{\mathrm{Res}}
\newcommand{\Int}{\mathrm{Int}}
\newcommand{\tr}{\mathrm{tr}}
\newcommand{\fin}{\mathrm{fin}}
\newcommand{\Hom}{\mathrm{Hom}}
\newcommand{\Gal}{\mathrm{Gal}}
\newcommand{\Lie}{\mathrm{Lie}}
\newcommand{\Frob}{\mathrm{Frob}}
\newcommand{\WD}{{\it WD}}
\newcommand{\Ad}{\mathrm{Ad}}
\newcommand{\vol}{\mathrm{vol}}
\newcommand{\even}{{\rm even}}
\newcommand{\odd}{{\rm odd}}

\newcommand{\Sp}{\mathrm{Sp}}
\newcommand{\SL}{\mathrm{SL}}
\newcommand{\GL}{\mathrm{GL}}
\newcommand{\Mp}{\mathrm{Mp}}
\newcommand{\Oo}{\mathrm{O}}
\newcommand{\SO}{\mathrm{SO}}
\newcommand{\Sym}{\mathrm{Sym}}

\newcommand{\iif}{&\quad&\text{if }}
\newcommand{\other}{&\quad&\text{otherwise}}
\newcommand{\resp}{resp.~}
\newcommand{\bs}{\backslash}
\newcommand{\ep}{\varepsilon}
\newcommand{\lam}{\lambda}
\newcommand{\Lam}{\Lambda}
\renewcommand{\1}{{\bf 1}}
\newcommand{\bphi}{\underline{\phi}}
\newcommand{\bpsi}{\underline{\psi}}

\newcommand{\pair}[1]{\langle #1 \rangle}
\newcommand{\half}[1]{\frac{#1}{2}}
\newcommand{\norm}[1]{\| #1 \|}
\newcommand{\cl}[1]{\widetilde{#1}}
\newcommand{\ch}[1]{\check{#1}}

\newtheorem{thm}{Theorem}[section]
\newtheorem{lem}[thm]{Lemma}
\newtheorem{prop}[thm]{Proposition}
\newtheorem{cor}[thm]{Corollary}
\newtheorem{rem}[thm]{Remark}
\newtheorem{conj}[thm]{Conjecture}
\newtheorem{hyp}[thm]{Hypothesis}
\newtheorem{des}[thm]{Desideratum}

\makeatletter
\def\iddots{\mathinner{\mkern1mu\raise\p@
	\hbox{.}\mkern2mu\raise4\p@\hbox{.}\mkern2mu
	\raise7\p@\vbox{\kern7\p@\hbox{.}}\mkern1mu}}
\def\adots{\mathinner{\mkern2mu\raise\p@\hbox{.}
 \mkern2mu\raise4\p@\hbox{.}\mkern1mu
 \raise7\p@\vbox{\kern7\p@\hbox{.}}\mkern1mu}}
\makeatother

\allowdisplaybreaks

\title[Even orthogonal groups]{On the local Langland correspondence \\and Arthur conjecture
\\ for Even Orthogonal Groups}
\author{Hiraku Atobe \and Wee Teck Gan}
\date{}
\address{Department of mathematics, Kyoto University, Kitashirakawa-Oiwake-cho, Sakyo-ku, Kyoto, 606-8502, Japan}
\email{atobe@math.kyoto-u.ac.jp}
\address{Department of Mathematics, National University of Singapore, 10 Lower Kent Ridge Road, Singapore 119076}
\email{matgwt@nus.edu.sg}

\begin{document}
\maketitle

\begin{abstract}
In this paper, we highlight and state precisely the local Langlands correspondence 
for quasi-split $\Oo_{2n}$ established by Arthur.
We give two applications: Prasad's conjecture and Gross--Prasad conjecture for $\Oo_n$.
Also, we discuss the Arthur conjecture for $\Oo_{2n}$, and
establish the Arthur multiplicity formula for $\Oo_{2n}$.
\end{abstract}

\tableofcontents
\section{Introduction}
In his long-awaited book \cite{Ar}, 
Arthur obtained a classification of irreducible representations of 
quasi-split symplectic and special orthogonal groups over local fields of characteristic $0$  
(the local Langlands correspondence LLC) 
as well as a description of the automorphic discrete spectrum of these groups 
over number fields (the Arthur conjecture). 
He proved these results by establishing the twisted endoscopic transfer of 
automorphic representations from these classical groups to $\GL_N$ 
by exploiting the stabilization of the twisted trace formula of $\GL_N$ 
(which has now been completed by Waldspurger and M{\oe}glin). 
However, for the quasi-split special even orthogonal groups $\SO_{2n}$, 
the result is not as definitive as one hopes.  
More precisely, for a $p$-adic field $F$, 
Arthur only gave a classification of the irreducible representations of $\SO_{2n}(F)$ 
up to conjugacy by $\Oo_{2n}(F)$, instead of by $\SO_{2n}(F)$.  
Likewise, over a number field $\F$, 
he does not distinguish between a square-integrable automorphic representation $\pi$ 
and its twist by the outer automorphism corresponding to an element of 
$\Oo_{2n}(\F) \smallsetminus \SO_{2n}(\F)$.
\par

The reason for this less-than-optimal result for $\SO_{2n}$ is that, 
for the purpose of the twisted endoscopic transfer, 
it is more natural to work with the orthogonal groups $\Oo_{2n}$ instead of $\SO_{2n}$.  
In fact, Arthur has obtained in \cite[Theorems 2.2.1, 2.2.4]{Ar} 
a full classification of the irreducible representations of $\Oo_{2n}(F)$. 
It is from this that he deduced the weak LLC for 
$\SO_{2n}(F)$ alluded to above. 
Indeed, the weak LLC for $\SO_{2n}(F)$ is equivalent to 
the classification of irreducible representations of $\Oo_{2n}(F)$ 
modulo twisting by the determinant character $\det$. 
\par

Unfortunately, this rather complete result for $\Oo_{2n}(F)$ was not highlighted in \cite{Ar}. 
One possible reason is that $\Oo_{2n}$ is a disconnected linear algebraic group 
and so does not fit in the framework of the classical Langlands program; 
for example, one does not have a systematic definition of the $L$-group of 
a disconnected reductive group 
and so one does not have the notion of Langlands parameters. 
In choosing to stick to the context of the Langlands program, 
Arthur has highlighted the results for $\SO_{2n}$ instead. 
However, it has been observed that a suitable $L$-group for $\Oo_{2n}$ is the group $\Oo_{2n}(\C)$
and an $L$-parameter for $\Oo_{2n}(F)$ should be an orthogonal representation of 
the Weil--Deligne group $\WD_F$. 
A precise statement to this effect seems to be first formulated in the paper \cite{P} of D. Prasad. 
\par

The main goal of this paper is to highlight and state precisely the local Langlands correspondence 
for quasi-split $\Oo_{2n}(F)$ established by Arthur in \cite{Ar} using this notion of $L$-parameters, 
and to establish various desiderata of this LLC. 
We also formulate the natural extension of the LLC to the pure inner forms 
(using Vogan $L$-packets \cite{V}).  
The statements can be found in Desiderata \ref{desO} and \ref{propO}.  
We especially note the key role played by the local intertwining relation in Hypothesis \ref{Hypo}.
This local intertwining relation was established in \cite{Ar} for quasi-split groups 
but is conjectural for pure inner forms.  
\par

Our main motivation for formulating a precise LLC for $\Oo_{2n}$ is that 
the representations of $\Oo_{2n}(F)$ arise naturally in various context, 
such as in the theory of theta correspondence. 
If one wants to describe the local theta correspondence for the dual pair 
$\Oo_{2n}(F)  \times \Sp_{2m}(F)$, 
one would need a classification of irreducible representations of $\Oo_{2n}(F)$. 
Thus, this paper lays the groundwork needed for our paper \cite{AG} 
in which we determine the local theta lifting of tempered representations 
in terms of the local Langlands correspondence.
Having described the LLC for $\Oo_{2n}$, we give two applications:
\par

\begin{itemize}
\item (Prasad's conjecture) 
We complete the results in the first author's PhD thesis \cite{At}, 
in which the local theta correspondences for the almost equal rank dual pairs  
$\Oo_{2n}(F) \times \Sp_{2n}(F)$ and $\Oo_{2n}(F) \times \Sp_{2n-2}(F)$ were 
determined in terms of the weak LLC for $\Oo_{2n}(F)$.   
In particular, we describe these theta correspondences completely 
in terms of the LLC for $\Oo_{2n}(F)$, 
thus completing the proof of Prasad's conjecture (Conjectures \ref{P O} and \ref{P Sp}). 
The result is contained in Theorem \ref{main1}.

\item (Gross--Prasad conjecture for $\Oo_n$)  
In \cite{At}, these theta correspondences were used to prove 
the Fourier--Jacobi case of the local Gross--Prasad (GP) conjecture for 
symplectic-metaplectic groups, 
by relating the Fourier--Jacobi case with the Bessel case of (GP) for $\SO_n$ 
(which has been established by Waldspurger). 
For this purpose, the weaker version of Prasad's conjecture 
(based on the weak LLC for $\Oo_{2n}$) is sufficient. 
Now that we have the full Prasad's conjecture, 
we use the Fourier--Jacobi case of (GP) to prove a version of (GP) for $\Oo_n$.  
In other words, 
we shall extend and establish the Gross-Prasad conjecture to the context of orthogonal groups
(Conjecture \ref{GPO}).
The result is contained in Theorem \ref{main2}.
\end{itemize}
\par

These results are related in a complicated manner.
The following diagram is a summary of the situation:
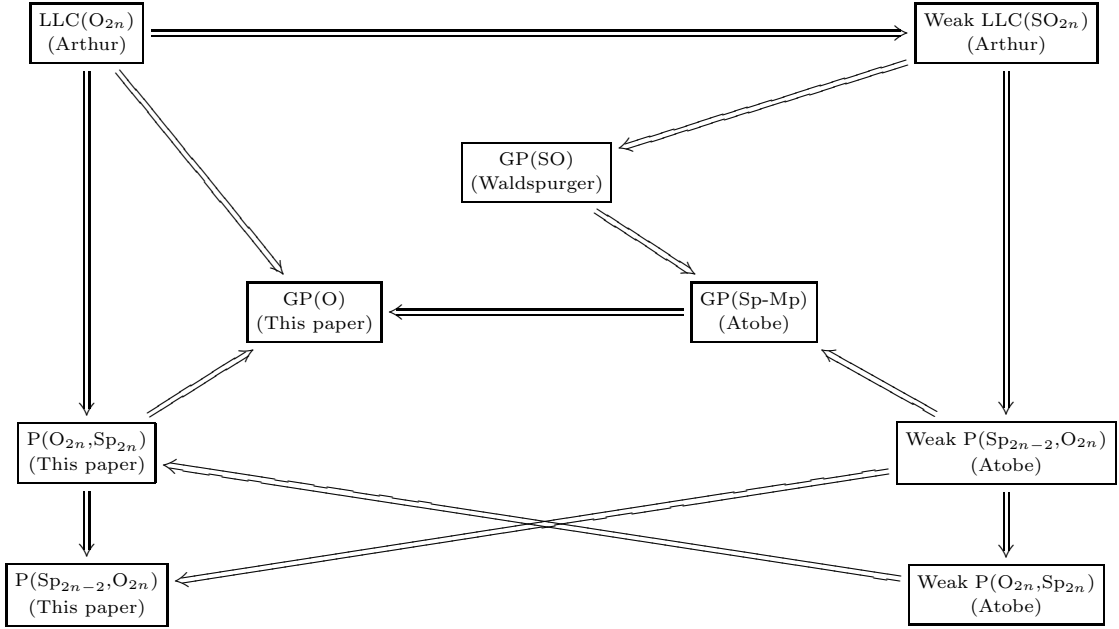
\[
\xymatrix{
\text{\fbox{$\substack{\text{LLC}(\Oo_{2n})\\[1pt] \text{(Arthur)}}$}} 
 \ar@{=>}[rrrr] \ar@{=>}[ddd] \ar@{=>}[ddr]&&&&
\text{\fbox{$\substack{\text{Weak LLC}(\SO_{2n})\\[1pt] \text{(Arthur)}}$}} 
\ar@{=>}[lld] \ar@{=>}[ddd]\\
&&
\text{\fbox{$\substack{\text{GP}(\SO)\\[1pt]\text{(Waldspurger)}}$}} 
\ar@{=>}[rd]&&\\
&\text{\fbox{$\substack{\text{GP}(\Oo)\\[1pt]\text{(This paper)}}$}} 
&&\text{\fbox{$\substack{\text{GP}(\Sp \text{-} \Mp)\\[1pt]\text{(Atobe)}}$}} 
\ar@{=>}[ll]&\\
\text{\fbox{$\substack{\text{P}(\Oo_{2n}, \Sp_{2n})\\[1pt]\text{(This paper)}}$}} 
\ar@{=>}[d] \ar@{=>}[ur]&&&&
\text{\fbox{$\substack{\text{Weak P}(\Sp_{2n-2}, \Oo_{2n})\\[1pt]\text{(Atobe)}}$}} 
\ar@{=>}[ul] \ar@{=>}[d] \ar@{=>}[lllld]\\
\text{\fbox{$\substack{\text{P}(\Sp_{2n-2}, \Oo_{2n})\\[1pt]\text{(This paper)}}$}} 
&&&&
\text{\fbox{$\substack{\text{Weak P}(\Oo_{2n}, \Sp_{2n})\\[1pt]\text{(Atobe)}}$}} 
\ar@{=>}[llllu]
}
\]
Here, 
\begin{itemize}
\item
LLC($\Oo_{2n}$) means the LLC for $\Oo_{2n}$, which has been established by Arthur 
\cite[Theorems 2.2.1, 2.2.4]{Ar};
\item
Weak LLC($\SO_{2n}$) means the weak LLC for $\SO_{2n}$ or $\Oo_{2n}$;
\item
Weak P$(\Sp_{2n-2}, \Oo_{2n})$ means the weaker version of Prasad's conjecture for 
$(\Sp_{2n-2}, \Oo_{2n})$, which was proven in \cite[\S 7]{At};
\item
Weak P($\Oo_{2n}, \Sp_{2n}$) means the weaker version of Prasad's conjecture for 
$(\Oo_{2n}, \Sp_{2n})$, which follows from Weak P$(\Sp_{2n-2}, \Oo_{2n})$ (see \cite[\S 5.5]{At});
\item
P($\Oo_{2n}, \Sp_{2n}$) means Prasad's conjecture for $(\Oo_{2n}, \Sp_{2n})$
(Conjecture \ref{P O}), which follows from Weak P($\Oo_{2n}, \Sp_{2n}$) and Hypothesis \ref{Hypo}
(see Theorem \ref{main1});
\item
P($\Sp_{2n-2}, \Oo_{2n}$) means Prasad's conjecture for $(\Sp_{2n-2}, \Oo_{2n})$
(Conjecture \ref{P Sp}), which follows from 
Weak P($\Oo_{2n}, \Sp_{2n}$) and P($\Oo_{2n}, \Sp_{2n}$)
(see Theorem \ref{thmPSp});
\item
GP($\SO$) means the local Gross--Prasad (GP) conjecture for special orthogonal groups, 
which was established by Waldspurger \cite{W2}, \cite{W3}, \cite{W4}, \cite{W5};
\item
GP($\Sp$-$\Mp$) means the local Gross--Prasad (GP) conjecture for 
symplectic-metaplectic groups, which was proven in \cite[Theorem 1.3]{At}
by using GP($\SO$) and Weak P($\Sp_{2n-2}, \Oo_{2n}$);
\item
GP($\Oo$) means the local Gross--Prasad (GP) conjecture for orthogonal groups
(Conjecture \ref{GPO}), which follows from 
GP($\Sp$-$\Mp$) and P($\Oo_{2n}, \Sp_{2n}$) (see Theorem \ref{main2}).
\end{itemize}
\par

Finally, we discuss how the Arthur conjecture for the automorphic discrete spectrum of 
$\SO_{2n}$ over a number field $\F$ can be extended to an analogous statement for $\Oo_{2n}$. 
This is implicitly in \cite{Ar}, but not precisely formulated. 
In particular, we describe the automorphic discrete spectrum of $\Oo_{2n}$ 
in terms of local and global Arthur packets, 
and establish the Arthur multiplicity formula (Theorem \ref{main3}). 
Also we show that the tempered part of the automorphic discrete spectrum of $\Oo_{2n}$ 
has multiplicity $1$ (Proposition \ref{mult1}, see also Remark \ref{temp}). 
This should lay the groundwork for a more precise study of 
the global theta correspondence for symplectic-orthogonal dual pairs. 

\subsection*{Acknowledgments}
The authors are grateful to Professor Atsushi Ichino for his helpful comments
and in particular for ringing \cite[Theorem 2.2.4]{Ar} to their attention.
We also thank Dipendra Prasad for telling us his results (\cite{P3}, \cite{P4}).
The first author is supported by JSPS KAKENHI Grant Number 26-1322.
The second author is partially supported by 
a Singapore government MOE Tier 2 grant R-146-000-175-112.
The paper was completed when the second author visited Kyoto University in December 2015
as Distinguished Visiting Project Professor 
(Japan Gateway: Kyoto University Top Global Program) 
of Center for the Promotion of Interdisciplinary Education and Research. 
The second author would like to thank Atsushi Ichino 
for his kind invitation and Kyoto University for its generous support.

\subsection*{Notations}
Let $F$ be a non-archimedean local field with characteristic zero, 
$\oo$ be the ring of integers of $F$,
$\varpi$ be a uniformizer, 
$q$ be the number of elements in the residue class field $\oo/\varpi\oo$ 
and $|\cdot|_F$ be the normalized absolute value on $F$ so that $|\varpi|_F=q^{-1}$.
We denote by $W_F$ and $\WD_F=W_F\times\SL_2(\C)$ 
the Weil and Weil--Deligne groups of $F$, respectively.
Fix a non-trivial additive character $\psi$ of $F$. 
For $c\in F^\times$, 
we define an additive character $\psi_c$ or $c\psi$ of $F$ by
\[
\psi_c(x)=c\psi(x)=\psi(cx).
\]
We set $\chi_c=(\cdot,c)$ to be the quadratic character of $F^\times$ 
associated to $c\in F^\times/F^{\times2}$.
Here, $(\cdot,\cdot)$ is the quadratic Hilbert symbol of $F$.
For a totally disconnected locally compact group $G$, 
we denote the set of equivalence classes of irreducible smooth
representations of $G$ by $\Irr(G)$.
If $G$ is the group of $F$-points of a linear algebraic group over $F$,
we denote by $\Irr_\temp(G)$ (\resp $\Irr_\disc(G)$) the subset of $\Irr(G)$
of classes of irreducible tempered representations
(\resp irreducible discrete series representations).
For a topological group $H$,
we define the component group of $H$ by $\pi_0(H)=H/H^\circ$, 
where $H^\circ$ is the identity component of $H$.
The Pontryagin dual (i.e., the character group)
of a finite abelian group $A$ is denoted by $A^D$ or $\widehat{A}$.
\par

%\section{Orthogonal spaces}
%\section{Orthogonal spaces}
\section{Quasi-split orthogonal groups}
In this section, we summarize facts about quasi-split orthogonal groups and their representations.

%\subsection{Orthogonal spaces}\label{Quadratic}
\subsection{Orthogonal spaces}\label{Quadratic}
Let $V=V_{2n}$ be a vector space of dimension $2n$ over $F$ and
\[
\pair{\cdot,\cdot}_V \colon V\times V \rightarrow F
\]
be a non-degenerate symmetric bilinear form.
We take a basis $\{e_1, \dots, e_{2n}\}$ of $V$, and define the discriminant of $V$ by
\[
\disc(V) = (-1)^{n} \det((\pair{e_i,e_j}_V)_{i,j}) \bmod F^{\times2} \in F^\times/ F^{\times2}.
\]
Let $\chi_V = (\cdot, \disc(V))$ be the character of $F^\times$ associated with $F(\sqrt{\disc(V)})$.
We call $\chi_V$ the discriminant character of $V$.
The orthogonal group $\Oo(V)$ associated to $V$ is defined by
\[
\Oo(V) = \{g \in \GL(V)\ |\ \text{$\pair{gv,gv'}_V = \pair{v,v'}_V$ for any $v,v' \in V$}\}.
\]
\par

Fix $c, d \in F^\times$.
Let 
\[
V_{(d,c)} = F[X]/(X^2-d)
\]
be a $2$-dimensional vector space equipped with a bilinear form
\[
(\alpha,\beta)\mapsto \pair{\alpha,\beta}_{V_{(d,c)}}\coloneqq c\cdot\tr(\alpha\overline{\beta}),
\]
where $\beta\mapsto\overline{\beta}$ is the involution on $F[X]/(X^2-d)$
induced by $a+bX\mapsto a-bX$.
This involution is regarded as an element $\epsilon \in \Oo(V_{(d,c)})$.
The images of $1, X \in F[X]$ in $V_{(d,c)}$ are denoted by $e, e'$, respectively.
\par

For $n>1$, we say that $V_{2n}$ is associated to $(d,c)$ if
\[
V_{2n} \cong V_{(d,c)} \oplus \H^{n-1}
\]
as orthogonal spaces, 
where $\H$ is the hyperbolic plane, i.e., 
$\H=Fv_i+Fv_i^*$ with $\pair{v_i,v_i}=\pair{v_i^*,v_i^*}=0$ and $\pair{v_i,v_i^*}=1$.
Note that $\disc(V_{2n}) = d \bmod F^{\times2}$.
The orthogonal group $\Oo(V_{2n})$ is quasi-split, and
any quasi-split orthogonal group can be obtained in this way.
Note that $V_{2n} = V_{(d,c)} \oplus \H^{n-1} \cong V_{2n}' = V_{(d',c')} \oplus \H^{n-1}$ 
as orthogonal spaces if and only if
$d \equiv d' \bmod F^{\times2}$ so that $E \coloneqq F(\sqrt{d})=F(\sqrt{d'})$ and
$c \equiv c' \bmod N_{E/F}(E^\times)$.

%\subsection{Generic representations}\label{sec gen}
\subsection{Generic representations}\label{sec gen}
Suppose that $n>1$ in this subsection.
Let $V=V_{2n}$ be an orthogonal space associated to $(d,c)$.
We set
\[
X_k=Fv_1+\dots+Fv_k
\quad\text{and}\quad
X_k^*=Fv_1^*+\dots+Fv_k^*
\]
for $1\leq k\leq n-1$.
We denote by $B_0=TU_0$ the $F$-rational Borel subgroup of $\SO(V_{2n})$ stabilizing the complete flag
\[
0\subset \pair{v_1}\subset \pair{v_1,v_2}\subset\dots\subset\pair{v_1,\dots,v_{n-1}}=X_{n-1},
\]
where $T$ is the $F$-rational torus stabilizing the lines $Fv_i$ for $i=1,\dots,n-1$.
We identify $\Oo(V_{(d,c)})$ as a subgroup of $\Oo(V_{2n})$ which fixes $\H^{n-1}$.
Via the canonical embedding $\Oo(V_{(d,c)}) \hookrightarrow \Oo(V_{2n})$, 
we regard $\epsilon$ as an element in $\Oo(V_{2n})$.
Note that $\epsilon$ depends on $(d,c)$.
We define a generic character $\mu_{c}$ of $U_0$ by
\[
\mu_{c}(u)=\psi(\pair{uv_2,v_1^*}_V+\dots+\pair{uv_n,v_{n-2}^*}_V+\pair{ue,v_{n-1}^*}_V).
\]
Note that $\epsilon$ normalizes $U_0$ and fixes $\mu_c$.
\par

Put $E=F(\sqrt{d})$.
If $c' \in cN_{E/F}(E^\times)$, then 
we have an isomorphism 
$V_{2n} = V_{(d,c)} \oplus \H^{n-1} \rightarrow V'_{2n} = V_{(d,c')} \oplus \H^{n-1}$, 
and so that we obtain an isomorphism
\[
f \colon \Oo(V_{2n}) \rightarrow \Oo(V'_{2n}).
\]
Moreover, we can take an isomorphism $f \colon \Oo(V_{2n}) \rightarrow \Oo(V'_{2n})$ such that
$f(B_0) = B'_0$, $f(T) = T'$ and $f(\Oo(V_{(d,c)})) = \Oo(V_{(d,c')})$, 
where $B'_0 = T'U'_0$ and $T'$ are
the Borel subgroup and maximal torus of $\SO(V'_{2n})$, respectively, defined as above.
Let $\mathcal{F}$ be the set of such isomorphisms.
Then the group $T' \rtimes \pair{\epsilon'} \cong \Oo(V_{(d,c')}) \times (F^\times)^{n-1}$ acts on 
$\mathcal{F}$ by
\[
(t' \cdot f) (g) = t'f(g)t'^{-1}
\]
for $t' \in T' \rtimes \pair{\epsilon'}$ and $g \in \Oo(V_{2n})$.
Here, $\epsilon' \in \Oo(V_{(d,c')})$ is an analogue to $\epsilon \in \Oo(V_{(d,c)})$.
Since $n>1$, this action of $T' \rtimes \pair{\epsilon'}$ is transitive.
\par

Choosing $f \in \mathcal{F}$, 
we regard $\mu_{c'}$ as a generic character of $U_0$ by
\[
U_0 \xrightarrow{f} U'_0 \xrightarrow{\mu_{c'}} \C^\times.
\]
Note that the $T$-orbit of $\mu_{c'}$ is independent of the choice of $f$
since $\epsilon'$ fixes $\mu_{c'}$.
\par

We consider a $4$-tuple $(V, B_0, T, \mu)$, where
\begin{itemize}
\item
$V = V_{2n}$ is an orthogonal space associated to some $(d,c)$;
\item
$B_0$ is an $F$-rational Borel subgroup of $\SO(V)$;
\item
$T$ is a maximal $F$-torus contained in $B_0$;
\item
$\mu$ is a generic character of $U_0$, 
where $U_0$ is the unipotent radical of $B_0$.
\end{itemize}
We say that two tuples $(V,B_0,T, \mu)$ and $(V',B'_0,T', \mu')$ is equivalent 
if the following conditions hold:
\begin{enumerate}
\item
there exists an isomorphism $V \rightarrow V'$ as orthogonal spaces, 
which induces a group isomorphism $f \colon \Oo(V) \rightarrow \Oo(V')$;
\item
$f(B_0) = B'_0$ and $f(T) = T'$ (so that $f(U_0) = U'_0$);
\item
there exists $t \in T$ such that $\mu' \circ f = \mu \circ \Int(t)$.
\end{enumerate}
\begin{prop}
Fix $d \in F^{\times}/F^{\times2}$.
For $c \in F^\times$, we associate the $4$-tuple
$(V, B_0, T, \mu)$, where
\begin{itemize}
\item
$V = V_{2n}$ is an orthogonal space associated to $(d,c)$;
\item
$B_0$ and $T$ are as above;
\item
$\mu = \mu_c$.
\end{itemize}
Then the map $c \mapsto (V,B_0,T, \mu)$ gives a canonical bijection (not depending on $\psi$)
\[
F^\times/F^{\times2} \rightarrow 
\{\text{equivalence classes of tuples $(V, B_0, T, \mu)$ with $\disc(V) = d$}\}.
\]
\end{prop}
\begin{proof}
Let $V = V_{2n}$ be an orthogonal space associated to $(d,c)$.
By \cite[\S12]{GGP}, the map $c'\mapsto \mu_{c'} \circ f$ for $f \in \mathcal{F}$ 
gives a well-defined bijection 
\[
cN_{E/F}(E^\times)/F^{\times2}\rightarrow \text{\{$T$-orbits of generic characters of $U_0$\}}, 
\]
where $E=F(\sqrt{d})$.
For $c' \in cN_{E/F}(E^\times)$, let $V' = V'_{2n}$ be an orthogonal space associated to $(d,c')$.
Then two tuples $(V,B_0,T,\mu_{c'} \circ f)$ and $(V',B'_0,T',\mu_{c'})$
are equivalent each other.
This implies that the map
\[
F^\times \rightarrow 
\{\text{equivalence classes of tuples $(V, B_0, T, \mu)$ with $\disc(V) = d$}\}.
\]
is surjective.
Also, we note that for a generic character $\mu$, 
two tuples $(V,B_0,T, \mu_c)$ and $(V,B_0,T, \mu)$ are equivalent 
if and only if $\mu = \mu_c \circ \Int(t)$ for some $t \in T$.
This implies that the above map induces the bijection
\[
F^\times/F^{\times2} \rightarrow 
\{\text{equivalence classes of tuples $(V, B_0, T, \mu)$ with $\disc(V) = d$}\},
\]
as desired.
\end{proof}

\begin{rem}
Let $(V, \pair{\cdot, \cdot}_V)$ be an orthogonal space associated to $(d,c)$.
Fix $a \in F^\times$.
We define a new orthogonal $(V', \pair{\cdot, \cdot}_{V'})$ by 
$V'=V$ as vector spaces and by
\[
\pair{x,y}_{V'} = a\cdot \pair{x,y}_V.
\]
Then $(V', \pair{\cdot, \cdot}_{V'})$ is associated to $(d,ac)$.
As subgroup of $\GL(V) = \GL(V')$, we have identifications
\[
\Oo(V) = \Oo(V')
\quad\text{and}\quad
\SO(V) = \SO(V').
\]
These identifications preserve $F$-rational Borel subgroups and maximal $F$-tori.
Moreover, the generic character $\mu_c$ of a maximal unipotent subgroup of $\SO(V)$
transfers $\mu_{ac}$.
More precisely, see \cite[Appendix A.5]{At}.
\end{rem}

Since $\epsilon'$ stabilizes $\mu_{c'}$, 
we can extend $\mu_{c'}$ to $U'=U'_0 \rtimes \pair{\epsilon'}$.
There are exactly two such extensions $\mu_{c'}^\pm \colon U' \rightarrow \C^\times$ 
which are determined by
\[
\mu_{c'}^\pm(\epsilon')=\pm1.
\]
We say that an irreducible smooth representation $\sigma$ of $\Oo(V_{2n})$
is $\mu_{c'}^\pm$-generic if
\[
\Hom_{f^{-1}(U')}(\sigma,\mu_{c'}^\pm)\not=0
\]
for some $f \in \mathcal{F}$.
For $\sigma_0 \in \Irr(\SO(V_{2n}))$, 
the $\mu_{c'}$-genericity is defined similarly.
The $\mu_{c'}^{\pm}$-genericity and the $\mu_{c'}$-genericity are independent of the choice of $f$.
Note that $f^{-1}(U'_0) = U_0$ for $f \in \mathcal{F}$, and
\[
\dim_\C(\Hom_{U_0}(\sigma_0,\mu_{c'})) \leq 1
\]
for $\sigma_0 \in \Irr(\SO(V_{2n}))$.

\begin{lem}\label{gen}
Let $\sigma_0 \in \Irr(\SO(V_{2n}))$.
\begin{enumerate}
\item
Assume that $\sigma_0$ can be extended to $\Oo(V_{2n})$.
Then there are exactly two such extensions.
Moreover, the following are equivalent:
\begin{enumerate}
\item[(A)]$\sigma_0$ is $\mu_{c'}$-generic;
\item[(B)]
exactly one of two extensions is $\mu_{c'}^+$-generic but not $\mu_{c'}^-$-generic,
and the other is $\mu_{c'}^-$-generic but not $\mu_{c'}^+$-generic.
\end{enumerate}

\item
Assume that $\sigma_0$ can not be extended to $\Oo(V_{2n})$.
Then $\sigma = \Ind_{\SO(V_{2n})}^{\Oo(V_{2n})}(\sigma_0)$ is irreducible.
Moreover, the following are equivalent:
\begin{enumerate}
\item[(A)]$\sigma_0$ is $\mu_{c'}$-generic;
\item[(B)]
$\sigma$ is both $\mu_{c'}^+$-generic and $\mu_{c'}^-$-generic.
\end{enumerate}

\end{enumerate}
\end{lem}

\begin{proof}
The first assertions of (1) and (2) follow from the Clifford theory.
It is easy that (B) implies (A) in both (1) and (2).
\par

We show (A) implies (B).
Let $\sigma_0 \in \Irr(\SO(V_{2n}))$  be a $\mu_{c'}$-generic representation, i.e.,
\[
\dim_\C(\Hom_{U_0}(\sigma_0,\mu_{c'} \circ f)) = 1.
\]
for some $f \in \mathcal{F}$.
By the Frobenius reciprocity, we have
\begin{align*}
\Hom_{f^{-1}(U')}(\Ind_{\SO(V_{2n})}^{\Oo(V_{2n})}(\sigma_0), \mu_{c'}^\pm \circ f)
&\cong 
\Hom_{\Oo(V_{2n})}(\Ind_{\SO(V_{2n})}^{\Oo(V_{2n})}(\sigma_0), 
\Ind_{f^{-1}(U')}^{\Oo(V_{2n})}(\mu_{c'}^\pm \circ f))\\
&\cong 
\Hom_{\SO(V_{2n})}(\sigma_0, \Ind_{f^{-1}(U')}^{\Oo(V_{2n})}(\mu_{c'}^\pm \circ f)|\SO(V_{2n}))\\
&\cong 
\Hom_{\SO(V_{2n})}(\sigma_0, \Ind_{U_0}^{\SO(V_{2n})}(\mu_{c'} \circ f))
\cong 
\Hom_{U_0}(\sigma_0, \mu_{c'} \circ f).
\end{align*}
In particular, if $\sigma = \Ind_{\SO(V_{2n})}^{\Oo(V_{2n})}(\sigma_0)$ is irreducible, 
then $\sigma$ is both $\mu_{c'}^+$-generic and $\mu_{c'}^-$-generic.
This shows that (A) implies (B) in (2).
If $\Ind_{\SO(V_{2n})}^{\Oo(V_{2n})}(\sigma_0) \cong \sigma_1 \oplus \sigma_2$, then
\[
\Hom_{U_0}(\Ind_{\SO(V_{2n})}^{\Oo(V_{2n})}(\sigma_0)|\SO(V_{2n}),\mu_{c'} \circ f)
\cong 
\Hom_{U_0}(\sigma_1|\SO(V_{2n}),\mu_{c'} \circ f) \oplus 
\Hom_{U_0}(\sigma_2|\SO(V_{2n}),\mu_{c'} \circ f),
\]
and $\Hom_{U_0}(\sigma_i|\SO(V_{2n}),\mu_{c'} \circ f)$ is $f^{-1}(\epsilon')$-stable for $i=1,2$.
Hence this subspace is an eigenspace of $f^{-1}(\epsilon)$.
Since both $\pm1$ are  eigenvalues of $f^{-1}(\epsilon)$ in 
$\Hom_{U_0}(\Ind_{\SO(V_{2n})}^{\Oo(V_{2n})}(\sigma_0),\mu_{c'} \circ f)$, 
exactly one of $\sigma_1$ and $\sigma_2$ is $\mu_{c'}^+$-generic, and
the other is $\mu_c^-$-generic.
This shows that (A) implies (B) in (1).
\end{proof}

%\subsection{Unramified representations}\label{sec.unram}
\subsection{Unramified representations}\label{sec.unram}
Let $V=V_{2n}$ be an orthogonal space associated to $(d,c)$.
We say that $\Oo(V_{2n})$ (or $\SO(V_{2n})$) is unramified if $c, d \in \oo^\times$.
In this subsection, we assume this condition.
Recall that $V_{2n}$ has a decomposition
\[
V_{2n} = Fv_1 + \dots Fv_{n-1} + V_{(d,c)} + Fv_1^* + \dots + Fv_{n-1}^*
\]
with $V_{(d,c)} = Fe + Fe'$.
We set
\[
\left\{
\begin{aligned}
&v_0 = \half{e+u^{-1}e'},\ v_0^* = \frac{e-u^{-1}}{2c}	\iif d=u^2,\\
&\text{$\oo_E$ to be the ring of integers of $E= F(\sqrt{d})\cong V_{(d,c)}$}	
\iif d \not\in \oo^{\times2}.
\end{aligned}
\right.
\]
Note that $\pair{v_0,v_0}_{V}=\pair{v_0^*,v_0^*}_{V}=0$ and $\pair{v_0,v_0^*}_{V}=1$.
Let $L_{2n}$ be the $\oo$-lattice of $V_{2n}$ defined by
\[
L_{2n}=\left\{
\begin{aligned}
&\oo v_1 + \dots + \oo v_{n-1} + \oo v_0 + \oo v_0^* + \oo v_1^* + \dots + \oo v_{n-1}^*
\iif d \in \oo^{\times2},\\
&\oo v_1 + \dots + \oo v_{n-1} + \oo_E + \oo v_1^* + \dots + \oo v_{n-1}^*
\iif d \not\in \oo^{\times2}.
\end{aligned}
\right.
\]
Let $K$ be the maximal compact subgroup of $\Oo(V_{2n})$ which preserves the lattice $L_{2n}$.
Note that $K$ contains $\epsilon$ and satisfies
\[
K = K_0 \rtimes \pair{\epsilon}, 
\]
where $K_0 = K \cap \SO(V_{2n})$ is a maximal compact subgroup of $\SO(V_{2n})$.
\par

Let $\sigma \in \Irr(\Oo(V_{2n}))$ and $\sigma_0 \in \Irr(\SO(V_{2n}))$.
We say that $\sigma$ (\resp $\sigma_0$) is unramified (with respect to $K$ (\resp $K_0$))
if $\sigma$ (\resp $\sigma_0$) has a nonzero $K$-fixed (\resp $K_0$-fixed) vector.
In this case, it is known that $\dim(\sigma^K)=\dim(\sigma_0^{K_0})=1$.

\begin{lem}\label{unram}
Let $\sigma_0 \in \Irr(\SO(V_{2n}))$ be an unramified representation.
Then $\Ind_{\SO(V_{2n})}^{\Oo(V_{2n})}(\sigma_0)$ has a unique irreducible unramified constituent.
\end{lem}
\begin{proof}
Note that the map
\[
\Ind_{\SO(V_{2n})}^{\Oo(V_{2n})}(\sigma_0)^K \rightarrow \sigma_0^{K_0},\ f \mapsto f(1)
\]
is a $\C$-linear isomorphism.
Hence the assertion holds if $\Ind_{\SO(V_{2n})}^{\Oo(V_{2n})}(\sigma_0)$ is irreducible.
Now suppose that $\Ind_{\SO(V_{2n})}^{\Oo(V_{2n})}(\sigma_0)$ is reducible.
Then it decomposes into direct sum
\[
\Ind_{\SO(V_{2n})}^{\Oo(V_{2n})}(\sigma_0) \cong \sigma_1 \oplus \sigma_2.
\]
We may assume that $\sigma_1$ and $\sigma_2$ are realized on 
the same space $\VV$ as $\sigma_0$.
Since $\sigma_i(\epsilon)$ preserve the one dimension subspace $\VV^{K_0}$, 
we have $\sigma_i(\epsilon) = \pm \id$ on $\VV^{K_0}$.
Since $\sigma_1(\epsilon) = -\sigma_2(\epsilon)$, 
exactly one $i \in\{1,2\}$ satisfies that $\sigma_i(\epsilon) = +1$.
Then $\sigma_i$ is the unique irreducible unramified constituent of 
$\Ind_{\SO(V_{2n})}^{\Oo(V_{2n})}(\sigma_0)$.
\end{proof}

%\section{Local Langlands correspondence for $\SO(V_{2n})$ and $\Oo(V_{2n})$}
%\section{Local Langlands correspondence for $\SO(V_{2n})$ and $\Oo(V_{2n})$}
\section{Local Langlands correspondence for $\SO(V_{2n})$ and $\Oo(V_{2n})$}
In this we explain the LLC for $\SO(V_{2n})$ and $\Oo(V_{2n})$. 

%\subsection{Orthogonal representations of $\WD_F$ and its component groups}
\subsection{Orthogonal representations of $\WD_F$ and its component groups}
Let $M$ be a finite dimensional vector space over $\C$.
We say that a homomorphism $\phi \colon \WD_F \rightarrow \GL(M)$
is a representation of $\WD_F = W_F \times \SL_2(\C)$ if
\begin{itemize}
\item
$\phi(\Frob_F)$ is semi-simple, where $\Frob_{F}$ is a geometric Frobenius element in $W_{F}$;
\item
the restriction of $\phi$ to $W_F$ is smooth;
\item
the restriction of $\phi$ to $\SL_2(\C)$ is algebraic.
\end{itemize}
We call $\phi$ tempered if the image of $W_F$ is bounded.
\par

We say that $\phi$ is orthogonal
if there exists a non-degenerate bilinear form 
$B\colon M \times M \rightarrow \C$ such that
\[
\left\{
\begin{aligned}
&B(\phi(w)x,\phi(w)y) = B(x,y),\\
&B(y,x) = B(x,y)
\end{aligned}
\right.
\]
for $x,y\in M$ and $w\in \WD_F$.
In this case, $\phi$ is equivalent to its contragredient $\phi^\vee$.
More precisely, see \cite[\S 3]{GGP}.
\par

For an irreducible representation $\phi_0$ of $\WD_F$, 
we denote the multiplicity of $\phi_0$ in $\phi$ by $m_\phi(\phi_0)$.
We can decompose 
\[
\phi=m_1\phi_1+\dots+m_s\phi_s+\phi'+\phi'^\vee,
\]
where $\phi_1,\dots, \phi_s$ are distinct irreducible orthogonal representations
of $\WD_F$, $m_i=m_\phi(\phi_i)$, 
and $\phi'$ is a sum of irreducible representations of $\WD_F$
which are not orthogonal.
We say that a parameter $\phi$ is discrete if
$m_i=1$ for any $i=1,\dots, s$ and $\phi'=0$, i.e.,
$\phi$ is a multiplicity-free sum of irreducible orthogonal representations of $\WD_F$. 
\par

For a representation $\phi$ of $\WD_F$, 
the $L$-factor and the $\ep$-factor associated to $\phi$, which are defined in \cite{T},
are denoted by $L(s,\phi)$ and $\ep(s,\phi,\psi)$, respectively.
If $(\phi,M)$ is an orthogonal representation
with $\WD_F$-invariant symmetric bilinear form $B$, 
then we define the adjoint $L$-function $L(s,\phi,\Ad)$ associated to $\phi$ to be the $L$-function associated to 
\[
\Ad \circ \phi \colon \WD_F \rightarrow \GL(\Lie(\mathrm{Aut}(M,B))).
\]
We say that $\phi$ is generic if $L(s,\phi,\Ad)$ is regular at $s=1$.
Note that $\Ad \circ \phi \cong \wedge^2 \phi$ since $B$ is symmetric.
Hence the adjoint $L$-function $L(s,\phi,\Ad)$ is equal to
the exterior square $L$-function $L(s, \phi, \wedge^2) = L(s, \wedge^2 \phi)$.
\par

Let $\phi$ be a representation of $\WD_F = W_F \times \SL_2(\C)$.
We denote the inertia subgroup of $W_F$ by $I_F$.
We say that $\phi$ is unramified if $\phi$ is trivial on $I_F \times \SL_2(\C)$.
In this case, $\phi$ is a direct sum of unramified characters of $W_F^{\ab} \cong F^\times$.
\par

Let $(\phi,M)$ be an orthogonal representation of $\WD_F$
with invariant symmetric bilinear form $B$.
Let 
\[
C_\phi=\{g \in \GL(M)\ |\ 
\text{$B(gx,gy)=B(x,y)$ for any $x, y \in M$,
and 
$g \phi(w) = \phi(w) g$ for any $w \in \WD_F$}
\}
\]
be the centralizer of $\im(\phi)$ in $\mathrm{Aut}(M,B) \cong \Oo(\dim(M),\C)$.
Also we put
\[
C_\phi^+ = C_\phi \cap \SL(M).
\]
Finally, we define the large component group $A_\phi$ by
\[
A_\phi = \pi_0(C_\phi).
\]
The image of $C_\phi^+$ under the canonical map $C_\phi \rightarrow A_\phi$
is denoted by $A_\phi^+$, and called the component group of $\phi$.
By \cite[\S4]{GGP}, $A_\phi$ and $A_\phi^+$ are described explicitely as follows:
\par

Let $\phi=m_1\phi_1+\dots+m_s\phi_s+\phi'+\phi'^\vee$ be an orthogonal representation as above.
Then we have
\[
A_\phi=\bigoplus_{i=1}^{s}(\Z/2\Z) a_i \cong (\Z/2\Z)^s.
\]
Namely, $A_\phi$ is a free $\Z/2\Z$-module of rank $s$ and $\{a_1, \dots, a_s\}$
is a basis of $A_\phi$ with $a_i$ associated to $\phi_i$.
For $a=a_{i_1}+\dots+a_{i_k} \in A_\phi$ with $1\leq i_1 < \dots < i_k \leq s$, 
we put
\[
\phi^{a}=\phi_{i_1} \oplus \dots \oplus \phi_{i_k}.
\]
Also, we put 
\[
z_\phi \coloneqq \sum_{i=1}^s m_\phi(\phi_i)\cdot a_i 
=\sum_{i=1}^s m_i \cdot a_i \in A_\phi.
\]
This is the image of $-\1_M \in C_\phi$.
We call $z_\phi$ the central element in $A_\phi$.
The determinant map $\det \colon \GL(M) \rightarrow \C^\times$ gives a homomorphism
\begin{align*}
\det \colon A_\phi \rightarrow \Z/2\Z, \quad
\sum_{i=1}^s \ep_i a_i \mapsto \sum_{i=1}^s \ep_i \cdot \dim(\phi_i) \pmod2,
\end{align*}
where $\ep_i\in\{0,1\} = \Z/2\Z$.
Then we have $A_\phi^+=\ker(\det)$.
\par

By \cite[\S4]{GGP},
for each $c \in F^\times$, we can define a character $\eta_{\phi,c}$ of $A_\phi$ by
\[
\eta_{\phi,c}(a)=\det(\phi^a)(c).
\]
Note that $\eta_{\phi,c}(z_\phi)=1$ if and only if $c \in \ker(\det(\phi))$.

%\subsection{$L$-group and $L$-parameters of $\SO(V_{2n})$}\label{Lpara}
\subsection{$L$-group and $L$-parameters of $\SO(V_{2n})$}\label{Lpara}
Let $V_{2n}$ be an orthogonal space associated to $(d,c)$ for some $c,d \in F^\times$.
We put $E=F(\sqrt{d})$.
Then the Langlands dual group of $\SO(V_{2n})$ is the complex Lie group $\SO(2n,\C)$.
We use 
\[
J= \begin{pmatrix}
&&1\\
&\iddots&\\
1&&
\end{pmatrix}
\]
to define $\Oo(2n,\C)$, i.e., 
$\Oo(2n,\C) = \{g \in \GL_{2n}(\C)\ |\ {}^t g J g = J\}$.
We denote the $L$-group of $\SO(V_{2n})$ by ${}^L(\SO(V_{2n})) = \SO(2n,\C) \rtimes W_F$.
The action of $W_F$ on the dual group $\SO(2n,\C)$ factors through $W_F/W_E \cong \Gal(E/F)$.
If $E \not=F$, i.e., $\SO(V_{2n})$ is not split, then
the generator $\gamma \in \Gal(E/F)$ acts on $\SO(2n,\C)$ by
the inner automorphism of 
\[
\epsilon = 
\begin{pmatrix}
\1_{n-1}&&&\\
&0&1&\\
&1&0&\\
&&&\1_{n-1}
\end{pmatrix}
\in \Oo(2n,\C).
\]
Hence by $\gamma \mapsto \epsilon$, we have the homomorphism
\[
{}^L (\SO(V_{2n})) = \SO(2n,\C) \rtimes W_F \twoheadrightarrow 
\SO(2n, \C) \rtimes \Gal(E/F) \xrightarrow{\cong} \Oo(2n,\C).
\]
On the other hand, if $E=F$, i.e., $\SO(V_{2n})$ is split, then
$W_F$ acts on $\SO(2n,\C)$ trivially so that we have the homomorphism
\[
{}^L (\SO(V_{2n})) = \SO(2n,\C) \rtimes W_F \twoheadrightarrow 
\SO(2n, \C)  \hookrightarrow \Oo(2n,\C).
\]
\par

An $L$-parameter of $\SO(V_{2n})$ is an admissible homomorphism
\[
\bphi \colon \WD_F \rightarrow {}^L(\SO(V_{2n})) = \SO(2n,\C) \rtimes W_F.
\]
We put
\[
\Phi(\SO(V_{2n})) = \{\text{$\SO(2n,\C)$-conjugacy classes of $L$-parameters of $\SO(V_{2n})$}\}.
\]
For an $L$-parameter $\bphi \colon \WD_F \rightarrow {}^L(\SO(V_{2n}))$, 
by composing with the above map ${}^L(\SO(V_{2n})) \rightarrow \Oo(2n,\C)$, 
we obtain a homomorphism
\[
\phi \colon \WD_F \rightarrow \Oo(2n,\C).
\]
We may regard $\phi$ as an orthogonal representation of $\WD_F$.
Note that $\det(\phi) = \chi_V$ is the discriminant character of $V_{2n}$.
The map $\bphi \mapsto \phi$ gives an identification
\[
\Phi(\SO(V_{2n})) =
\{
\phi \colon \WD_F \rightarrow \Oo(2n,\C) \ |\ \det(\phi) = \chi_V\}/(\text{$\SO(2n,\C)$-conjugacy}).
\]
Namely, we may regard $\Phi(\SO(V_{2n}))$ as the set of $\SO(M)$-conjugacy classes of
orthogonal representations $(\phi,M)$ of $\WD_F$ with $\dim(M)=2n$ and $\det(\phi) = \chi_V$.
\par

We denote the subset of $\Phi(\SO(V_{2n}))$ consisting of 
$\SO(M)$-conjugacy classes of tempered (\resp discrete, generic) representations $(\phi,M)$
by $\Phi_\temp(\SO(V_{2n}))$ (\resp $\Phi_\disc(\SO(V_{2n}))$, $\Phi_\gen(\SO(V_{2n}))$).
Then we have a sequence
\[
\Phi_\disc(\SO(V_{2n})) \subset \Phi_\temp(\SO(V_{2n})) \subset \Phi_\gen(\SO(V_{2n})).
\]
We define $\Phi^\epsilon(\SO(V_{2n}))$ to be the subset of $\Phi(\SO(V_{2n}))$ consisting of $\phi$ 
which contains an irreducible orthogonal representation of $\WD_F$ with odd dimension.
We put $\Phi_*^\epsilon(\SO(V_{2n}))=\Phi^\epsilon(\SO(V_{2n})) \cap \Phi_*(\SO(V_{2n}))$
for $* \in \{\disc, \temp,\gen\}$.

%\subsection{Local Langlands correspondence for $\SO(V_{2n})$}
\subsection{Local Langlands correspondence for $\SO(V_{2n})$}\label{LLC SO}
Let $V_{2n}$ be an orthogonal space associated to $(d,c)$ for some $c,d \in F^\times$.
The discriminant character is denoted by $\chi_V \coloneqq \chi_d$.
We set $V_{2n}'$ to be the orthogonal space such that
\[
\dim(V_{2n}')=2n
\quad\text{and}\quad
\disc(V_{2n}')=\disc(V_{2n})
\]
but $V_{2n}' \not\cong V_{2n}$ as orthogonal spaces.
Such $V_{2n}'$ exists uniquely up to isomorphisms
unless $n=1$ and $d \in F^{\times2}$
By a companion space of $V_{2n}$, we mean $V_{2n}$ or $V_{2n}'$.
\par

Now we describe the desiderata for the local Langlands correspondence for $\SO(V_{2n})$. 
\begin{des}[LLC for $\SO(V_{2n})$]\label{des}
Let $V_{2n}$ be an orthogonal space associated to $(d,c)$, 
and $\chi_V = (\cdot, d)$ be the discriminant character of $V_{2n}$.
\begin{enumerate}
\item
There exists a canonical surjection
\[
\bigsqcup_{V_{2n}^\bullet}
\Irr(\SO(V_{2n}^\bullet)) \rightarrow \Phi(\SO(V_{2n})).
\]
where $V_{2n}^\bullet$ runs over all companion spaces of $V_{2n}$.
For $\phi \in \Phi(\SO(V_{2n}))$,
we denote by $\Pi_\phi^0$ the inverse image of $\phi$ under this map, 
and call $\Pi_\phi^0$ the $L$-packet of $\phi$.

\item
We have
\[
\bigsqcup_{V_{2n}^\bullet}
\Irr_*(\SO(V_{2n}^\bullet)) = \bigsqcup_{\phi \in \Phi_*(\SO(V_{2n}))}\Pi_\phi^0
\]
for $* \in \{\temp, \disc\}$.

\item
For each $c' \in F^\times$, there exists a suitable bijection
\[
\iota_{c'} \colon \Pi_\phi^0 \rightarrow \widehat{A_{\phi}^+}.
\]

\item
For $\sigma_0 \in \Pi_\phi^0$ and $c' \in F^\times$, the following are equivalent:
\begin{itemize}
\item
$\sigma_0 \in \Irr(\SO(V_{2n}))$;
\item
$\iota_{c'}(\sigma_0)(z_\phi) = \chi_V(c'/c)$.
\end{itemize}

\end{enumerate}
\end{des}
\par

Note that $(\phi,M) \in \Phi(\SO(V_{2n}))$ is not an equivalence class
but an $\SO(M)$-conjugacy class, 
Because of this difference, Desideratum \ref{des} has not been established.
Arthur \cite{Ar} has established LLC for $\Oo(V_{2n})$, and 
deduced a weaker version of Desideratum \ref{des} as follows.
\par

We introduce an equivalence relation $\sim_\epsilon$ on $\Irr(\SO(V_{2n}^\bullet))$.
Choose an element $\epsilon$ in $\Oo(V_{2n}^\bullet)$ such that $\det(\epsilon)=-1$.
For $\sigma_0 \in \Irr(\SO(V_{2n}^\bullet))$, 
we define its conjugate $\sigma_0^\epsilon$ by 
$\sigma_0^\epsilon(g)=\sigma_0(\epsilon^{-1} g \epsilon)$.
Then the equivalence relation $\sim_\epsilon$ on $\Irr(\SO(V_{2n}^\bullet))$ is defined by
\[
\sigma_0 \sim_\epsilon \sigma_0^\epsilon.
\]
The canonical map $\Irr(\SO(V_{2n}^\bullet)) \rightarrow \Irr(\SO(V_{2n}^\bullet))/\sim_\epsilon$
is denoted by $\sigma_0 \mapsto [\sigma_0]$.
We say that $[\sigma_0] \in \Irr(\SO(V_{2n}^\bullet))/\sim_\epsilon$
is tempered (\resp discrete, $\mu_{c'}$-generic, unramified) 
if so is some (and hence any) representative $\sigma_0$.
\par

Also, we introduce an equivalence relation $\sim_\epsilon$ on $\Phi(\SO(V_{2n}))$.
For $\phi, \phi' \in \Phi(\SO(V_{2n}))$, we write $\phi \sim_\epsilon \phi'$
if $\phi$ is $\Oo(2n,\C)$-conjugate to $\phi'$, i.e., $\phi$ is equivalent to $\phi'$
as representations of $\WD_F$.
The equivalence class of $\phi$ is also denoted by $\phi$.
\par

The desiderata for the weaker version of the local Langlands correspondence for $\SO(V_{2n})$
is described as follows: 
\begin{des}[Weak LLC for $\SO(V_{2n})$]\label{desSO}
Let $V_{2n}$ be an orthogonal space associated to $(d,c)$,
and $\chi_V = (\cdot, d)$ be the discriminant character of $V_{2n}$.
\begin{enumerate}
\item
There exists a canonical surjection
\[
\bigsqcup_{V_{2n}^\bullet}
\Irr(\SO(V_{2n}^\bullet))/\sim_\epsilon \rightarrow \Phi(\SO(V_{2n}))/\sim_\epsilon.
\]
where $V_{2n}^\bullet$ runs over all companion spaces of $V_{2n}$.
For $\phi \in \Phi(\SO(V_{2n}))/\sim_\epsilon$,
we denote by $\Pi_\phi^0$ the inverse image of $\phi$ under this map, 
and call $\Pi_\phi^0$ the $L$-packet of $\phi$.

\item
We have
\[
\bigsqcup_{V_{2n}^\bullet}
\Irr_*(\SO(V_{2n}^\bullet))/\sim_\epsilon = 
\bigsqcup_{\phi \in \Phi_*(\SO(V_{2n}))/\sim_\epsilon}\Pi_\phi^0
\]
for $* \in \{\temp, \disc\}$.

\item
The following are equivalent:
\begin{itemize}
\item
$\phi \in \Phi^\epsilon(\SO(V_{2n}))/\sim_\epsilon$;
\item
some $[\sigma_0] \in \Pi_\phi^0$ satisfies $\sigma_0^\epsilon \cong \sigma_0$;
\item
all $[\sigma_0] \in \Pi_\phi^0$ satisfy $\sigma_0^\epsilon \cong \sigma_0$.
\end{itemize}
Here, $\Phi^\epsilon(\SO(V_{2n}))/\sim_\epsilon$ is the subset of $\Phi(\SO(V_{2n}))/ \sim_\epsilon$
consisting of $\phi$ which contains an irreducible orthogonal representation of $\WD_F$ with odd dimension.

\item
For each $c' \in F^\times$, there exists a bijection (not depending on $\psi$)
\[
\iota_{c'} \colon \Pi_\phi^0 \rightarrow \widehat{A_{\phi}^+},
\]
which satisfies the endoscopic and twisted endoscopic character identities.

\item
For $[\sigma_0] \in \Pi_\phi^0$ and $c' \in F^\times$, the following are equivalent:
\begin{itemize}
\item
$\sigma_0 \in \Irr(\SO(V_{2n}))$;
\item
$\iota_{c'}([\sigma_0])(z_\phi) = \chi_V(c'/c)$.
\end{itemize}

\item
Assume that $\phi=\phi_\tau+\phi_0+\phi_\tau^\vee$, where 
$\phi_0$ is an element in $\Phi_\temp(\SO(V_{2n_0}))/\sim_\epsilon$ and 
$\phi_\tau$ is an irreducible tempered representation of $\WD_F$
which corresponds to $\tau \in \Irr_\temp(\GL_k(F))$ with $n=n_0+k$.
Then the induced representation 
\[
\Ind_P^{\SO(V_{2n})}(\tau \otimes \sigma_0)
\]
is a multiplicity-free direct sum of tempered representations of $\SO(V_{2n})$,
where $P$ is a parabolic subgroup of $\SO(V_{2n})$ with
Levi subgroup $M_P=\GL_k(F) \times \SO(V_{2n_0})$ and
$\sigma_0$ is a representative of an element in $\Pi_{\phi_0}^0$.
The $L$-packet $\Pi_\phi^0$ is given by
\[
\Pi_\phi^0 = \{[\sigma]\ |\ 
\text{$\sigma \subset \Ind_P^{\SO(V_{2n})}(\tau \otimes \sigma_0)$ 
for some $[\sigma_0] \in \Pi_{\phi_0}^0$}\}.
\]
Moreover if $\sigma \subset \Ind_P^{\SO(V_{2n})}(\tau \otimes \sigma_0)$, 
then $\iota_{c'}([\sigma])|A_{\phi_0}^+ = \iota_{c'}([\sigma_0])$ for $c' \in F^\times$.

\item
Assume that 
\[
\phi=\phi_{\tau_1}|\cdot|_F^{s_1} + \dots +\phi_{\tau_r}|\cdot|_F^{s_r} + \phi_0
+(\phi_{\tau_1}|\cdot|_F^{s_1} + \dots +\phi_{\tau_r}|\cdot|_F^{s_r})^\vee,
\]
where $\phi_0$ is an element in $\Phi_\temp(\SO(V_{2n_0}))/\sim_\epsilon$, 
$\phi_{\tau_i}$ is an irreducible tempered representation of $\WD_F$
which corresponds to $\tau_i \in \Irr_\temp(\GL_{k_i}(F))$ with $n=n_0+k_1+\dots+k_r$
and $s_i$ is a real number with $s_1 \geq \dots \geq s_r>0$.
Then the $L$-packet $\Pi_\phi^0$ consists of the equivalence classes of
the unique irreducible quotients $\sigma$ of 
the standard modules
\[
\Ind_P^{\SO(V_{2n})}(\tau_1|\det|_F^{s_1} \otimes \dots \otimes \tau_r|\det|_F^{s_r} 
\otimes \sigma_0),
\]
where $\sigma_0$ runs over representatives of elements of $\Pi_{\phi_0}^0$
and $P$ is a parabolic subgroup of $\SO(V_{2n})$ with
Levi subgroup $M_P=\GL_{k_1}(F) \times \dots \times \GL_{k_r}(F) \times \SO(V_{2n_0})$.
Moreover if $\sigma$ is the unique irreducible quotient of 
$\Ind_P^{\SO(V_{2n})}(\tau_1|\det|_F^{s_1} \otimes \dots \otimes \tau_r|\det|_F^{s_r} 
\otimes \sigma_0)$,
then $\iota_{c'}([\sigma])|A_{\phi_0}^+ = \iota_{c'}([\sigma_0])$ for $c' \in F^\times$.
\end{enumerate}
\end{des}

In this paper, we take the position that 
the stabilization of the twisted trace formula used in \cite{Ar} is complete.
See the series of papers
\cite{Stab1}, \cite{Stab2}, \cite{Stab3}, \cite{Stab4}, \cite{Stab5}, 
\cite{Stab6}, \cite{Stab7}, \cite{Stab8}, \cite{Stab9} and \cite{Stab10}
of Waldspurger and M{\oe}glin--Waldspurger, 
and papers of Chaudouard--Laumon \cite{CL1} and \cite{CL2}.
Then the following theorem holds.

\begin{thm}[\cite{Ar}]\label{LLC-SO}
Let $V_{2n}$ be an orthogonal space associated to $(d,c)$.
Put $E=F(\sqrt{d})$.
Then there exist a surjective map
\[
\Irr_{\temp}(\SO(V_{2n}))/\sim_\epsilon \rightarrow \Phi_\temp(\SO(V_{2n}))/\sim_\epsilon
\]
with the inverse image $\Pi_\phi^{0}$ of $\phi \in \Phi_\temp(\SO(V_{2n}))/\sim_\epsilon$, 
and a bijection 
\[
\iota_{c'} \colon \Pi_\phi^{0} \rightarrow (A_\phi^+/\pair{z_\phi})\widehat{\ }
\]
for $c' \in cN_{E/F}(E^\times)$ 
satisfying Desideratum \ref{desSO} (2), (3), (4), and (6).
Moreover, using the Langlands classification, 
we can extend the map $[\sigma] \mapsto \phi$ to a surjective map
\[
\Irr(\SO(V_{2n}))/\sim_\epsilon \rightarrow \Phi(\SO(V_{2n}))/\sim_\epsilon
\]
which satisfies Desideratum \ref{desSO} (7). 
\end{thm}

\begin{rem}\label{remSO}
\begin{enumerate}
\item
M{\oe}glin's work in \cite[\S 1.4, Theorem 1.4.1]{M2} may have extended 
Theorem \ref{LLC-SO} to the pure inner forms as well, 
in which case Desideratum \ref{desSO} would be known in general.
However, we are not sure how her work fits with the general theory of 
T. Kaletha on the normalization of transfer factors for inner forms \cite{Ka2}. 
In particular, we are not sure if the local character relation
of Arthur \cite[Theorems 2.2.1, 2.2.4]{Ar} (the analog of Hypothesis \ref{Hypo} below) 
holds in her work. 
\item
If $d \not \in F^{\times2}$, then $\SO(V_{2n}^\bullet)$ is quasi-split for any 
companion space of $V_{2n}$, so that we may define $L$-packets $\Pi_\phi^0$
and bijections
\begin{align*}
\iota_{c_1} \colon \Pi_\phi^{0} \cap \Irr(\SO(V_{2n}))/\sim_\epsilon 
\rightarrow (A_\phi^+/\pair{z_\phi})\widehat{\ },\\
\iota_{c_2} \colon \Pi_\phi^{0} \cap \Irr(\SO(V'_{2n}))/\sim_\epsilon 
\rightarrow (A_\phi^+/\pair{z_\phi})\widehat{\ }
\end{align*}
for $c_1, c_2 \in F^\times$ with $c_1\in cN_{E/F}(E^\times)$ and $c_2 \not\in cN_{E/F}(E^\times)$.
We define $\iota_{c_2}([\sigma])$ for $[\sigma] \in \Irr(\SO(V_{2n}))/\sim_\epsilon$ by
\[
\iota_{c_2}([\sigma]) \coloneqq \iota_{c_1}([\sigma]) \otimes \eta_{\phi, c_1/c_2}, 
\]
and define $\iota_{c_1}([\sigma'])$ for $[\sigma'] \in \Irr(\SO(V'_{2n}))/\sim_\epsilon$ similarly.
%However, it is not known at this point how to define the bijection
%$\iota_{c'} \colon \Pi_\phi^0 \rightarrow \widehat{A_{\phi}^+}$.
%In order for Proposition \ref{propSO} (1) below to be true unconditionally, 
%there should be a very simple-minded way of 
%so that 
Then the character relations and local intertwining relations would continue to hold
after modifying the transfer factor and the normalization of intertwining operators.
See also \cite{KMSW}, \cite[\S 5.4]{Ka2} and \cite[Appendix A]{At}.
\end{enumerate}
\end{rem}

There are some properties of $\Pi_\phi^0$.
\begin{prop}\label{propSO}
Assume Weak LLC for $\SO(V_{2n})$ (Desideratum \ref{desSO}).
\begin{enumerate}
\item
For $c_1, c_2 \in F^\times$, we have
\[
\iota_{c_2}([\sigma_0]) = \iota_{c_1}([\sigma_0]) \otimes \eta_{\phi,c_2/c_1}
\]
as a character of $A_\phi^+$.
\item
$\phi$ is generic, i.e., $L(s,\phi,\Ad)$ is regular at $s=1$
if and only if
$\Pi_\phi^0$ contains a $\mu_{c'}$-generic class $[\sigma_0]$ 
for each $c' \in F^\times$.
Note that if $c' \not \in cN_{E/F}(E^\times)$, then $\sigma_0 \in \Irr(\SO(V'_{2n}))$.
\item
If $\phi$ is generic, then for each $c' \in F^\times$, 
$[\sigma_0] \in \Pi_\phi^0$ is $\mu_{c'}$-generic 
if and only if $\iota_{c'}([\sigma_0])$ is 
the trivial representation of $A_\phi^+$.
\item
If both $\SO(V_{2n})$ and $\phi$ are unramified, 
then $\Pi_\phi^0$ contains a unique unramified class $[\sigma_0]$, and
it corresponds to the trivial representation of $A_\phi^+$ under $\iota_{c}$. 
\end{enumerate}
In particular, these properties hold for the $L$-packets 
$\Pi_\phi^{0} \cap \Irr(\SO(V_{2n}))/\sim_\epsilon$ of quasi-split $\SO(V_{2n})$
and $c_1,c_2,c' \in cN_{E/F}(E^\times)$ unconditionally.
\end{prop}
\begin{proof}
(1) is given in \cite[Theorem 3.3]{Ka}.
(2) is a conjecture of Gross--Prasad and Rallis (\cite[Conjecture 2.6]{GP}), and
has been proven by Gan--Ichino (\cite[Appendix B]{GI2}).
For (3), it is shown in \cite[Proposition 8.3.2 (a)]{Ar} supplemented by some results of many others
that the class $[\sigma_0]$ corresponding to 
the trivial representation of $A_\phi^+$ under $\iota_{c'}$
is $\mu_{c'}$-generic.
A simple proof of the other direction is given by the first author \cite{At2}.
Also (3) is a special case of Gross--Prasad conjecture \cite[Conjecture 17.1]{GGP}, 
which is proven by Waldspurger \cite{W2}, \cite{W3}, \cite{W4} and \cite{W5}.
Finally, (4) is proven by M\oe glin \cite{M}.
\end{proof}

%\subsection{Local Langlands correspondence for $\Oo(V_{2n})$}
\subsection{Local Langlands correspondence for $\Oo(V_{2n})$}\label{LLC O}
Let $V_{2n}$ be an orthogonal space associated to $(d,c)$, and
$\epsilon \in \Oo(V_{2n})$ be as in \S \ref{Quadratic}.
Put $\theta=\Int(\epsilon)$.
It is an element in $\mathrm{Aut}(\SO(V_{2n}))$.
In \cite{Ar}, Arthur has established the local Langlands correspondence for
not $\Oo(V_{2n})$ but for
\[
\SO(V_{2n}) \rtimes \pair{\theta}.
\]
As topological groups, $\Oo(V_{2n})$ and $\SO(V_{2n}) \rtimes \pair{\theta}$ are isomorphic.
However, it is not canonical.
There are exactly two isomorphism $\Oo(V_{2n}) \cong \SO(V_{2n}) \rtimes \pair{\theta}$
which are identity on $\SO(V_{2n})$, and they are determined by 
$\pm \epsilon \leftrightarrow \theta$.
We use the isomorphism such that $\epsilon \leftrightarrow \theta$.
Via this isomorphism, we translate LLC for $\SO(V_{2n}) \rtimes \pair{\theta}$
into LLC for $\Oo(V_{2n})$.
Note that the changing of the choice of the isomorphism corresponds to 
the automorphism
\[
\Oo(V_{2n}) \rightarrow \Oo(V_{2n}),\ 
g \mapsto \det(g) \cdot g
= \left\{
\begin{aligned}
&g	\iif g \in \SO(V_{2n}),\\
&-g	\other.
\end{aligned}
\right.
\]
Hence it induces the bijection
\[
\Irr(\Oo(V_{2n})) \rightarrow \Irr(\Oo(V_{2n})), \ 
\sigma \mapsto (\omega_\sigma \circ \det) \otimes \sigma, 
\]
where $\omega_\sigma$ is the central character of $\sigma$, 
which is regarded as a character of $\{\pm1\}$.
\par

Let $V_{2n}^\bullet$ be a companion space of $V_{2n}$.
We define an equivalence relation $\sim_{\det}$ on $\Irr(\Oo(V_{2n}^\bullet))$ by 
\[
\sigma \sim_{\det} \sigma \otimes \det
\]
for $\sigma \in \Irr(\Oo(V_{2n}^\bullet))$.
The restriction and the induction give a canonical bijection 
\[
\Irr(\Oo(V_{2n}^\bullet))/\sim_{\det} \longleftrightarrow \Irr(\SO(V_{2n}^\bullet))/\sim_\ep.
\]
Put $\Phi(\Oo(V_{2n})) = \Phi(\SO(V_{2n}))/\sim_\epsilon$ and 
$\Phi_*(\Oo(V_{2n})) = \Phi_*(\SO(V_{2n}))/\sim_\epsilon$
for $* \in \{\temp, \disc, \gen\}$.
Also, we define $\Phi^\epsilon(\Oo(V_{2n})) = \Phi^\epsilon(\SO(V_{2n}))/\sim_\epsilon$.
Namely, $\Phi(\Oo(V_{2n}))$ is the set of equivalence classes of orthogonal representations
of $\WD_F$ with dimension $2n$ and determinant $\chi_V$.
We call an element in $\Phi(\Oo(V_{2n}))$ an $L$-parameter for $\Oo(V_{2n})$.
\par

We describe the local Langlands correspondence for $\Oo(V_{2n})$.
\begin{des}[LLC for $\Oo(V_{2n})$]\label{desO}
Let $V_{2n}$ be an orthogonal space associated to $(d,c)$,
and $\chi_V = (\cdot, d)$ be the discriminant character of $V_{2n}$.
\begin{enumerate}
\item
There exists a canonical surjection
\[
\bigsqcup_{V_{2n}^\bullet}
\Irr(\Oo(V_{2n}^\bullet)) \rightarrow \Phi(\Oo(V_{2n})).
\]
where $V_{2n}^\bullet$ runs over all companion spaces of $V_{2n}$.
For $\phi \in \Phi(\Oo(V_{2n}))$,
we denote by $\Pi_\phi$ the inverse image of $\phi$ under this map, 
and call $\Pi_\phi$ the $L$-packet of $\phi$.

\item
We have
\[
\bigsqcup_{V_{2n}^\bullet}
\Irr_*(\Oo(V_{2n}^\bullet)) = \bigsqcup_{\phi \in \Phi_*(\Oo(V_{2n}))}\Pi_\phi
\]
for $* \in \{\temp, \disc\}$.

\item
The following are equivalent:
\begin{itemize}
\item
$\phi \in \Phi^\epsilon(\Oo(V_{2n}))$;
\item
some $\sigma \in \Pi_\phi$ satisfies $\sigma \otimes \det \not\cong \sigma$;
\item
all $\sigma \in \Pi_\phi$ satisfy $\sigma \otimes \det \not\cong \sigma$.
\end{itemize}
Here, $\Phi^\epsilon(\Oo(V_{2n}))$ is the subset of $\Phi(\Oo(V_{2n}))$ consisting
of $\phi$ which contains an irreducible orthogonal representation of $\WD_F$ with odd dimension.

\item
For each $c' \in F^\times$, there exists a bijection (not depending on $\psi$)
\[
\iota_{c'} \colon \Pi_\phi \rightarrow \widehat{A_{\phi}},
\]
which satisfies the (twisted) endoscopic character identities.

\item
For $\sigma \in \Pi_\phi$ and $c' \in F^\times$, the following are equivalent:
\begin{itemize}
\item
$\sigma \in \Irr(\Oo(V_{2n}))$;
\item
$\iota_{c'}(\sigma)(z_\phi) = \chi_V(c'/c)$.
\end{itemize}

\item
Assume that $\phi=\phi_\tau+\phi_0+\phi_\tau^\vee$, where 
$\phi_0$ is an element in $\Phi_\temp(\Oo(V_{2n_0}))$ and 
$\phi_\tau$ is an irreducible tempered representation of $\WD_F$
which corresponds to $\tau \in \Irr_\temp(\GL_k(F))$ with $n=n_0+k$.
Then the induced representation 
\[
\Ind_P^{\Oo(V_{2n})}(\tau \otimes \sigma_0)
\]
is a multiplicity-free direct sum of tempered representations of $\Oo(V_{2n})$,
where $P$ is a parabolic subgroup of $\Oo(V_{2n})$ with
Levi subgroup $M_P=\GL_k(F) \times \Oo(V_{2n_0})$ and
$\sigma_0 \in \Pi_{\phi_0}$.
The $L$-packet $\Pi_\phi$ is given by
\[
\Pi_\phi = \{\sigma\ |\ 
\text{$\sigma \subset \Ind_P^{\Oo(V_{2n})}(\tau \otimes \sigma_0)$ 
for some $\sigma_0 \in \Pi_{\phi_0}$}\}.
\]
Moreover if $\sigma \subset \Ind_P^{\Oo(V_{2n})}(\tau \otimes \sigma_0)$, 
then $\iota_{c'}(\sigma)|A_{\phi_0} = \iota_{c'}(\sigma_0)$ for $c' \in F^\times$.

\item
Assume that 
\[
\phi=\phi_{\tau_1}|\cdot|_F^{s_1} + \dots +\phi_{\tau_r}|\cdot|_F^{s_r} + \phi_0
+(\phi_{\tau_1}|\cdot|_F^{s_1} + \dots +\phi_{\tau_r}|\cdot|_F^{s_r})^\vee,
\]
where $\phi_0$ is an element in $\Phi_\temp(\Oo(V_{2n_0}))$, 
$\phi_{\tau_i}$ is an irreducible tempered representation of $\WD_F$
which corresponds to $\tau_i \in \Irr_\temp(\GL_{k_i}(F))$ with $n=n_0+k_1+\dots+k_r$
and $s_i$ is a real number with $s_1 \geq \dots \geq s_r>0$.
Then the $L$-packet $\Pi_\phi$ consists of
the unique irreducible quotients $\sigma$ of 
the standard modules
\[
\Ind_P^{\Oo(V_{2n})}(\tau_1|\det|_F^{s_1} \otimes \dots \otimes \tau_r|\det|_F^{s_r} 
\otimes \sigma_0),
\]
where $\sigma_0$ runs over elements of $\Pi_{\phi_0}$
and $P$ is a parabolic subgroup of $\Oo(V_{2n})$ with
Levi subgroup $M_P=\GL_{k_1}(F) \times \dots \times \GL_{k_r}(F) \times \Oo(V_{2n_0})$.
Moreover if $\sigma$ is the unique irreducible quotient of 
$\Ind_P^{\Oo(V_{2n})}(\tau_1|\det|_F^{s_1} \otimes \dots \otimes \tau_r|\det|_F^{s_r} 
\otimes \sigma_0)$,
then $\iota_{c'}(\sigma)|A_{\phi_0} = \iota_{c'}(\sigma_0)$ for $c' \in F^\times$.

\item
For $\phi \in \Phi(\Oo(V_{2n})) = \Phi(\SO(V_{2n}))/\sim_\epsilon$, 
the image of $\Pi_\phi$ under the map 
\[
\Irr(\Oo(V_{2n}^\bullet)) \rightarrow \Irr(\Oo(V_{2n}^\bullet))/\sim_{\det} 
\rightarrow \Irr(\SO(V_{2n}^\bullet))/\sim_\epsilon
\]
is the packet $\Pi_{\phi}^0$ in Weak LLC for $\SO(V_{2n})$, and
the diagram
\[
\begin{CD}
\Pi_\phi @>\iota_{c'}>> \widehat{A_\phi}\\
@VVV @VVV\\
\Pi_\phi^0 @>\iota_{c'}>> \widehat{A_\phi^+}
\end{CD}
\]
is commutative for $c' \in F^\times$.
\item
For $c' \in F^\times$ and $\sigma \in \Pi_\phi$, 
the determinant twist $\sigma \otimes \det$ also belongs to $\Pi_\phi$, and
\[
\iota_{c'}(\sigma \otimes \det)(a) = \iota_{c'}(a) \cdot (-1)^{\det(a)}
\]
for $a \in A_\phi$.
\end{enumerate}
\end{des}

As Weak LLC for $\SO(V_{2n})$, the following theorem holds.
\begin{thm}[\cite{Ar}]\label{LLC-O}
Let $V_{2n}$ be an orthogonal space associated to $(d,c)$.
Put $E=F(\sqrt{d})$.
Then there exist a surjective map
\[
\Irr_{\temp}(\Oo(V_{2n})) \rightarrow \Phi_\temp(\Oo(V_{2n}))
\]
with the inverse image $\Pi_\phi$ of $\phi \in \Phi_\temp(\Oo(V_{2n}))$, 
and a bijection 
\[
\iota_{c'} \colon \Pi_\phi \rightarrow (A_\phi/\pair{z_\phi})\widehat{\ }
\]
for $c' \in cN_{E/F}(E^\times)$ 
satisfying Desideratum \ref{desO} (2), (3), (4), (6), (8), and (9).
Moreover, using the Langlands classification, 
we can extend the map $\sigma \mapsto \phi$ to a surjective map
\[
\Irr(\Oo(V_{2n})) \rightarrow \Phi(\Oo(V_{2n}))
\]
which satisfies Desideratum \ref{desO} (7). 
\end{thm}

In fact, Arthur established Theorem \ref{LLC-O} first and by using Desideratum \ref{desO} (8), 
he then defined the $L$-packets $\Pi_\phi^{0}$ for $\SO(V_{2n})$
(Theorem \ref{LLC-SO}).
\par

\begin{rem}
As we mentioned in Remark \ref{remSO}, 
M{\oe}glin's work in \cite[\S 1.4, Theorem 1.4.1]{M2} 
seems to extend Theorem \ref{LLC-O} to the pure inner forms as well.
Also, when $d \not\in F^{\times2}$, 
we can define $L$-packets $\Pi_\phi$ and a bijection
$\iota_{c'} \colon \Pi_\phi \rightarrow \widehat{A_\phi}$ 
for any $c' \in F^\times$ similar to Remark \ref{remSO}.
However, motivated by Prasad conjecture (Conjecture \ref{P O} below), 
we should define $\iota_{c'}(\sigma)$ for $\sigma \in \Irr(\Oo(V_{2n}))$ by 
\[
\iota_{c'}(\sigma) = \iota_{c}(\sigma) \otimes \eta_{\phi\chi_V, c'/c}.
\]
See also Desideratum \ref{propO} and Hypothesis \ref{Hypo} below.
\end{rem}

The following is an analogue of Proposition \ref{propSO}.
\begin{des}\label{propO}
Let $V_{2n}$ be an orthogonal space associated to $(d,c)$.
Let $\phi \in \Phi(\Oo(V_{2n}))$ and $\sigma \in \Pi_\phi$.
We write $\phi\chi_V = \phi \otimes \chi_V$.
\begin{enumerate}
\item
For $c_1,c_2 \in F^\times$, we have
\[
\iota_{c_2}(\sigma) = \iota_{c_1}(\sigma) \otimes \eta_{\phi\chi_V, c_2/c_1}
\]
as a character of $A_\phi$.

\item
$\phi$ is generic, i.e., $L(s,\phi,\Ad)$ is regular at $s=1$
if and only if
$\Pi_\phi$ contains a $\mu_{c'}^\ep$-generic representation $\sigma$ 
for each $c' \in F^\times$ and $\ep \in \{\pm1\}$.

\item
If $\phi$ is generic, then for each $c' \in F^\times$, 
\begin{itemize}
\item
$\sigma^+ \in \Pi_\phi$ is $\mu_{c'}^+$-generic if and only if
$\iota_{c'}(\sigma^+)$ is the trivial representation of $A_\phi$; 
\item
$\sigma^- \in \Pi_\phi$ is $\mu_{c'}^-$-generic if and only if
$\iota_{c'}(\sigma^-)$ is given by $A_\phi \ni a \mapsto (-1)^{\det(a)}$.
\end{itemize}

\item
If both $\Oo(V_{2n})$ and $\phi$ are unramified, 
then $\Pi_\phi$ contains a unique unramified representation $\sigma$, and
it corresponds to the trivial representation of $A_\phi$ under $\iota_{c}$. 

\end{enumerate}
\end{des}

Under Desideratum \ref{desO}, Proposition \ref{propSO} and Hypothesis \ref{Hypo}, 
Desideratum \ref{propO} will be proven in \S \ref{proof propO} below.
Note that $\eta_{\phi\chi_V, c_2/c_1} | A_\phi^+ = \eta_{\phi, c_2/c_1}$
since $\dim(\phi^a)$ is even for $a \in A_\phi^+$.

\subsection{Hypothesis}\label{secH}
To establish Desideratum \ref{propO} and two main local theorems, 
we will use a very delicate hypothesis, which is an intertwining relation.
\par

Let $V_{2n}$ be an orthogonal space associated to $(d,c)$, and 
$V$ be a companion space of $V_{2n}$.
For a fixed positive integer $k$, we set
\[
X=Fv_1 \oplus \dots \oplus Fv_k,\quad
X^*=Fv_1^* \oplus \dots \oplus Fv_k^*
\]
to be $k$-dimensional vector spaces over $F$.
Let $V' = V \oplus X \oplus X^*$ be the orthogonal space define by
\[
\pair{v_i,v_j}_{V'}=\pair{v_i^*,v_j^*}_{V'}=\pair{v_i,v_0}_{V'}=\pair{v_i^*,v_0}_{V'}=0,
\quad \pair{v_i,v_j^*}_{V'}=\delta_{i,j}
\]
for any $i,j=1,\dots,k$ and $v_0\in V$.
Let $P = M_PU_P$ be the maximal parabolic subgroup of $\Oo(V')$ stabilizing $X$, 
where $M_P$ is the Levi component of $P$ stabilizing $X^*$.
Hence
\[
M_P \cong \GL(X) \times \Oo(V).
\]
Using the basis $\{v_1,...,v_k\}$ of $X$,
we obtain an isomorphism $m_P \colon \GL_k(F) \rightarrow \GL(X)$.
Let $\phi_\tau$ be an orthogonal tempered representation of $\WD_F$ of dimension $k$, 
and $\tau$ be the tempered representation of $\GL_k(F)$ on a space $\VV_\tau$ 
associated to $\phi_\tau$.
For $s \in \C$, we realize the representation $\tau_s \coloneqq \tau \otimes |\det|_F^s$
on $\VV_\tau$ by setting $\tau_s(a) \coloneqq |\det(a)|_F^s \tau(a)v$
for $v \in \VV_\tau$ and $a \in \GL_k(F)$.
Let $\sigma \in \Irr_\temp(\Oo(V))$.
Assume that $\sigma \in \Pi_{\phi_\sigma}$ with $\phi_\sigma \in \Phi_\temp^\epsilon(\Oo(V_{2n}))$, 
i.e., $\sigma \not\cong \sigma \otimes \det$ and $\sigma|\SO(V)$ is irreducible.
We define a normalized intertwining operator
\begin{align*}
R_{c'}(w,\tau_s \otimes \sigma) &\colon \Ind_P^{\Oo(V')}(\tau_s \otimes \sigma) 
\rightarrow \Ind_P^{\Oo(V')}(\tau_{s} \otimes \sigma)
\end{align*}
by (the meromorphic continuations of) the integral
\[
R_{c'}(w,\tau_s \otimes \sigma)f_s(h') = 
e(V)^k \cdot \chi_V(c'/c)^k \cdot |c'|_F^{k\rho_P} \cdot r(\tau_s \otimes \sigma)^{-1}\cdot
\AA_{w}\left(
\int_{U_P} f_s(\cl{w}_{c'}^{-1} u_P h') du_P
\right)
\]
for $f_s \in \Ind_P^{\Oo(V')}(\tau_s \otimes \sigma)$.
Here, 
\begin{itemize}
\item
$w$ is the non-trivial element in the relative Weyl group $W(M_P) (\cong \Z/2\Z)$ for $M_P$;
\item
$\cl{w}_{c'} \in \Oo(V)$ is the representative of $w$ given by
\[
\cl{w}_{c'} = w_P \cdot m_P(c' \cdot a) \cdot ((-1)^k\1_V),
\]
where 
$w_P \in \Oo(V')$ is defined by $w_P v_i = -v_i^*$, $w_P v_i^* =-v_i$ and $w_P|V = \1_V$, 
and $a \in \GL_k(F)$ is given by
\[
a=
\begin{pmatrix}
&&(-1)^{n-k+1}\\
&\iddots&\\
(-1)^n&&
\end{pmatrix};
\]
\item
$e(V) = \iota_{c}(z_{\phi_\sigma}) \in \{\pm1\}$, i.e., 
\[
e(V) = \left\{
\begin{aligned}
&1	\iif \text{$V$ is associated to $(d,c)$},\\
&-1	\other;
\end{aligned}
\right.
\]
\item
$\rho_P=m+(k-1)/2$, so that the modulus character $\delta_P$ of $P$
satisfies that $\delta_P(m_P(a)) = |\det(a)|_F^{2\rho_P}$ for $a \in \GL_k(F)$;
\item
$r(\tau_s \otimes \sigma)$ is the normalizing factor given by
\[
r(\tau_s \otimes \sigma) = 
\lam(E/F, \psi)^k 
\frac{L(s, \phi_\tau \otimes \phi_\sigma)}
{\ep(s,\phi_\tau \otimes \phi_\sigma, \psi)L(1+s, \phi_\tau \otimes \phi_\sigma)}
\frac{L(-2s, (\wedge_2)^\vee \circ \phi_\tau)}
{\ep(-2s, (\wedge_2)^\vee \circ \phi_\tau)L(1-2s, (\wedge_2)^\vee \circ \phi_\tau)}, 
\]
where $\wedge_2$ is the representation of $\GL_k(\C)$ on 
the space of skew-symmetric $(k,k)$-matrices, and
$\lam(E/F, \psi)$ is the Langlands $\lam$-factor associated to $E=F(\sqrt{\disc(V)})=F(\sqrt{d})$.
\item
$du_P$ is the Haar measure of $U_P$ given in \cite{Ar} (see also \cite[\S 6.1]{At});
\item
$\AA_w \colon w(\tau \otimes \sigma) \rightarrow \tau \otimes \sigma$
is the intertwining isomorphism defined in \cite{Ar} (see also \cite[\S 6.3]{At}), 
where $w(\tau \otimes \sigma)(m) \coloneqq (\tau \otimes \sigma)(\cl{w}_{c'}^{-1}m\cl{w}_{c'})$ 
for $m \in M_P$.
\end{itemize}
\par

We expect that
the intertwining operators and the local Langlands correspondence are related as follows:
\begin{hyp}\label{Hypo}
Notation is as above.
\begin{enumerate}
\item
The normalized intertwining operator $R_{c'}(w,\tau_s \otimes \sigma)$ is
holomorphic at $s=0$.
We put $R_{c'}(w,\tau \otimes \sigma) \coloneqq R_{c'}(w,\tau_0 \otimes \sigma)$.
\item
Suppose that $\phi_\tau$ is a multiplicity-free sum of irreducible orthogonal tempered representations.
Put $\phi_{\sigma'} = \phi_\tau \oplus \phi_\sigma \oplus \phi_\tau$, and 
denote by $a \in A_{\phi_{\sigma'}}$ the element corresponding to $\phi_\tau$.
Let $\sigma' \in \Pi_{\phi_{\sigma'}}$ be an irreducible constituent of 
$\Ind_P^{\Oo(V')}(\tau \otimes \sigma)$.
Then we have
\[
R_{c'}(w,\tau \otimes \sigma) | \sigma' = \iota_{c'}(\sigma')(a)
\]
for any $c' \in F^\times$.
\end{enumerate}
\end{hyp}

In special cases, Hypothesis \ref{Hypo} has been established:
\begin{thm}
Hypothesis \ref{Hypo} holds in the following cases:
\begin{itemize}
\item
The case when $V=V_{2n}$ and $c' \in cN_{E/F}(E^\times)$.
\item
The case when $k$ is even and $d \not=1$ in $F^\times/F^{\times2}$
under assuming Desideratum \ref{desSO}.
\end{itemize}
\end{thm}
\begin{proof}
In the first case, Hypothesis \ref{Hypo} is 
Proposition 2.3.1 and Theorems 2.2.1, 2.2.4, 2.4.1 and 2.4.4 in \cite{Ar}.
The second case follows from Arthur's results above and \cite[Proposition 3.3]{At}.
\end{proof}

The cases when Hypothesis \ref{Hypo} has not yet been verified are
\begin{itemize}
\item
the non-quasi-split even orthogonal case; and
\item
the case when $k$ is odd and $d \not=1$ in $F^\times/F^{\times2}$.
\end{itemize}
In general, Hypothesis \ref{Hypo} would follow from similar results to \cite{Ar} and \cite{KMSW}.

\begin{rem}
Recall that we need to choose an isomorphism 
\[
\Oo(V) \longleftrightarrow \SO(V) \rtimes \pair{\theta}
\]
to translate Arthur's result.
There exist two choices of isomorphisms, 
which are determined by $\pm \epsilon \leftrightarrow \theta$.
We have chosen the isomorphism such that $\epsilon \leftrightarrow \theta$.
If one chooses the other isomorphism $-\epsilon \leftrightarrow \theta$, 
one should replace the representative $\cl{w}_{c'}$ of $w \in W(M_P)$ as
\[
\cl{w}'_{c'} = -w_P \cdot m_P(c' \cdot a) \cdot ((-1)^k\1_V) = - \cl{w}_{c'}.
\]
Note that 
\[
f_s(\cl{w}'_{c'}u_Ph') = \omega_{s}(-1) \cdot f_s(\cl{w}_{c'}u_Ph')
\]
for $f_s \in \Ind_P^{\Oo(V')}(\tau_s \otimes \sigma)$, 
where $\omega_s$ is the central character of $\Ind_P^{\Oo(V')}(\tau_s \otimes \sigma)$.
This is compatible the bijection 
\[
\Irr(\Oo(V)) \rightarrow \Irr(\Oo(V)),\ \sigma \mapsto (\omega_\sigma \circ \det) \otimes \sigma
\]
as in \S \ref{LLC O}.
Hence all results below are independent of the choice of the isomorphism 
$\Oo(V) \cong \SO(V) \rtimes \pair{\theta}$.
\end{rem}

%\subsection{Proof of Desideratum \ref{propO}}\label{proof propO}
\subsection{Proof of Desideratum \ref{propO}}\label{proof propO}
In this subsection, we prove Desideratum \ref{propO} under Hypothesis \ref{Hypo}.
First, we treat the tempered case.
\begin{thm}\label{des temp}
Assume Desiderata \ref{desSO}, \ref{desO} and Hypothesis \ref{Hypo}.
Then Desideratum \ref{propO} holds for $\phi \in \Phi_\temp(\Oo(V))$.
In particular, it holds for the $L$-packets $\Pi_\phi \cap \Irr(\Oo(V_{2n}))$ of quasi-split $\Oo(V_{2n})$
and $c_1,c_2,c' \in cN_{E/F}(E^\times)$ unconditionally.
\end{thm}
\begin{proof}
Note that Proposition \ref{propSO} holds since we assume Desideratum \ref{desSO}.
\par

First, we consider (1).
Let $\phi \in \Phi_\temp(\Oo(V))$ and $\sigma \in \Pi_\phi$.
We have to show that
\[
\iota_{c_1}(\sigma)(a) = \iota_{c_2}(\sigma)(a) \cdot \det(\phi^a\chi_V)(c_1/c_2)
\] 
for any $a \in A_\phi$ and $c_1, c_2 \in F^\times$.
Fix $a \in A_\phi$ and consider the parameter
\[
\phi' = \phi^a \oplus \phi \oplus \phi^a.
\]
Let $\tau \in \Irr(\GL_k(F))$ be the representation corresponding to $\phi^a$, 
where $k= \dim(\phi^a)$, and put $\sigma' = \Ind_P^{\Oo(V')}(\tau \otimes \sigma)$ as above.
Then $A_\phi = A_{\phi'}$ since $\phi$ contains $\phi^a$.
Hence $\sigma'$ is irreducible and $\sigma' \in \Pi_{\phi'}$ by Desideratum \ref{desO} (6).
Moreover, we have
\[
\iota_{c'}(\sigma')|A_{\phi} = \iota_{c'}(\sigma)
\]
for any $c' \in F^\times$.
By Hypothesis \ref{Hypo}, 
$R_{c_i}(w, \tau \otimes \sigma)$ is the scalar operator with eigenvalue $\iota_{c_i}(\sigma)(a)$
for $i=1,2$.
By definition, we have 
\[
R_{c_1}(w, \tau \otimes \sigma) = 
\chi_V(c_1/c_2)^k \cdot \omega_\tau(c_1/c_2)
\cdot R_{c_2}(w, \tau \otimes \sigma),
\]
where $\omega_\tau$ is the central character of $\tau$, which is equal to $\det(\phi^a)$.
Since
\[
\chi_V(c_1/c_2)^k \cdot \omega_\tau(c_1/c_2) = \det(\phi^a\chi_V)(c_1/c_2),
\]
we have
\[
\iota_{c_1}(\sigma)(a) = \det(\phi^a\chi_V)(c_1/c_2) \cdot \iota_{c_2}(\sigma)(a), 
\]
as desired.
\par

The assertion (2) follows from Lemma \ref{gen} and Proposition \ref{propSO} (2).
\par

Next, we consider (3).
Let $\phi \in \Phi_\temp(\Oo(V))$. 
Note that $\phi$ is generic.
By Proposition \ref{propSO} (2) and Desideratum \ref{propO} (2), 
for each $c' \in F^\times$, 
there exists a $\mu_{c'}^+$-generic representation $\sigma \in \Pi_\phi$ such that
$\iota_{c'}(\sigma)|A_{\phi}^+ = \1$.
We may assume that $\sigma$ is a representation of $\Oo(V)$ with
$V$ associated to $(d,c')$.
We have to show that $\iota_{c'}(\sigma)=\1$.
If $\phi \not\in \Phi^\epsilon(\Oo(V))$, then $A_\phi^+ = A_\phi$ so that
we have nothing to prove.
Hence we may assume that $\phi$ contains an irreducible orthogonal representation $\phi_0$ 
with odd dimension $k$.
Let $a_0 \in A_\phi$ be the element corresponding to $\phi_0$. 
For $s \in \C$, consider the parameter
\[
\phi' = \phi_0 \oplus \phi \oplus \phi_0.
\]
Let $\tau_s=\tau|\cdot|_F^s \in \Irr(\GL_k(F))$ be the representation corresponding to 
$\phi_0|\cdot|_F^s$.
We may assume that $\tau_s$ is realized on a space $\VV_\tau$, which is independent of $s\in\C$.
Put $\sigma'_s = \Ind_P^{\Oo(V')}(\tau_s \otimes \sigma)$ as above.
Then by Desideratum \ref{desO} (6), $\sigma'_0$ is irreducible and $\sigma'_0 \in \Pi_{\phi'}$.
Moreover, the canonical injection $A_\phi \hookrightarrow A_{\phi'}$ is bijective, 
and
we have
\[
\iota_{c'}(\sigma'_0)|A_{\phi} = \iota_{c'}(\sigma).
\]
Note that $\tau = \tau_0$ is tempered, so that generic. 
Fix a nonzero homomorphism
\[
l \colon \tau \otimes \sigma \rightarrow \C
\]
such that
\[
l(\sigma(u)v)=\mu_{c'}^+(u)\ell(v)
\]
for $u \in U' \cap M_P$ and $v \in \tau \otimes \sigma$, 
where $U' = U'_0 \rtimes \pair{\epsilon}$ with the maximal unipotent subgroup $U'_0$ of $\SO(V')$
as in \S \ref{sec gen}, 
and $M_P = \GL_k(F) \times \Oo(V)$ is the Levi subgroup of $P$.
For $f_s \in \sigma'_s$, we put
\[
l_s(f_s) = \int_{U_0}l(f_s(\cl{w}_{c'}^{-1}u_0))\mu_{c'}^{-1}(u_0)du_0, 
\]
where $\cl{w}_{c'} \in \SO(V)$ is the representative of $w$ defined in \S \ref{secH}.
Then by \cite[Proposition 2.1]{CS} and \cite[Proposition 3.1]{S1}, 
$l_s(f_s)$ is absolutely convergent for $\re(s) \gg 0$, and
holomorphic continuation to $\C$.
Moreover, $l_0$ gives a nonzero map
\[
l_0 \colon \sigma'_0 \rightarrow \C
\]
such that
\[
l_0(\sigma'_0(u')f_0) = \mu_{c'}^+(u') l_0(f_0)
\]
for $u' \in U'$ and $f_0 \in \sigma'_0$.
By a result of Shahidi (\cite[Theorem 3.5]{S2}), we have
\[
l_0 \circ R_{c'}(w,\tau \otimes \sigma) = l_0.
\]
See also \cite[Theorem 2.5.1]{Ar}.
This equation together with Hypothesis \ref{Hypo} shows that
\[
\iota_{c'}(\sigma)(a_0) = \iota_{c'}(\sigma'_0)(a_0) = 1.
\]
Since $[A_\phi:A_\phi^+]=2$, we have $\iota_{c'}(\sigma) = \1$, as desired.
\par

Finally, we consider (4).
Suppose that $\Oo(V)$ and $\phi \in \Phi_\temp(\Oo(V))$ are unramified.
By Desideratum \ref{desSO} and Lemma \ref{unram}, 
$\Pi_\phi$ contains a unique unramified representation $\sigma$, 
which satisfies that $\iota_c(\sigma)|A_{\phi}^+ = \1$.
We have to show that $\iota_c(\sigma) = \1$.
We may assume that $\phi \in \Phi^\epsilon(\Oo(V))$.
Since $\iota_c(\sigma)|A_{\phi}^+ = \1$, by Proposition \ref{propSO} (3), 
we see that $\sigma$ is $\mu_c$-generic, i.e., 
there is a nonzero homomorphism $l \colon \sigma \rightarrow \C$ such that
$l(\sigma(u_0)v)=\mu_c(u_0)l(v)$ for $u_0 \in U_0$ and $v \in \sigma$.
By the Casselman--Shalika formula \cite[Theorem 5.4]{CS},
we have $l|\sigma^{K_0} \not=0$, i.e., if $v \in \sigma$ is a nonzero $K_0$-fixed vector, 
then $l(v) \not=0$.
Since $\sigma \otimes \det \not\cong \sigma$, we have
\[
l \in \Hom_U(\sigma, \mu_c^\delta)
\]
for some $\delta \in \{\pm1\}$.
However, if $v \in \sigma$ is a nonzero $K$-fixed vector, then we have
\[
\delta \cdot l(v) = \mu_c^\delta(\epsilon) \cdot l(v) = l(\sigma(\epsilon)v)=l(v).
\]
This shows that $\sigma$ is $\mu_c^+$-generic, and so that $\iota_c(\sigma) = \1$
by Lemma \ref{gen} and Desideratum \ref{propO} (3).
\end{proof}

Now we treat the general case.
\begin{cor}
Assume Desiderata \ref{desSO}, \ref{desO} and Hypothesis \ref{Hypo}.
Then Desideratum \ref{propO} holds in general.
In particular, it holds for the $L$-packets $\Pi_\phi \cap \Irr(\Oo(V_{2n}))$ of quasi-split $\Oo(V_{2n})$
and $c_1,c_2,c' \in cN_{E/F}(E^\times)$ unconditionally.
\end{cor}
\begin{proof}
This follows from the compatibility of LLC and the Langlands quotients 
(Desideratum \ref{desO} (7)).
\end{proof}

\begin{rem}
\begin{enumerate}
\item
Kaletha proved Proposition \ref{propSO} (1) in \cite[Theorem 3.3]{Ka} 
by comparing transfer factors.
One may feel that the proof of Theorem \ref{des temp} (1) differs from Kaletha's proof.
However to prove Hypothesis \ref{Hypo}, one would need a similar argument to \cite{Ka}.
Hence the proof of Theorem \ref{des temp} (1) would be essentially the same as 
the one of \cite[Theorem 3.3]{Ka}.
\item
In \cite{At2}, the first author gave a proof of ``only if'' part of Proposition \ref{propSO} (3).
This proof is essentially the same as the proof of Desideratum \ref{propO} (3) (Theorem \ref{des temp}).
\end{enumerate}
\end{rem}

%\section{Prasad's conjecture}
%\section{Prasad's conjecture}
\section{Prasad's conjecture}
Prasad's conjecture describes precisely the local theta correspondence for 
$(\Oo(V_{2n}),\Sp(W_{2n}))$ in terms of the local Langlands correspondence
for $\Oo(V_{2n})$ and $\Sp(W_{2n})$.
A weaker version of this conjecture has been proven by the first author \cite{At}.
In this section, we state Prasad's conjecture and
give a proof for the full version.
\par

%\subsection{Local Langlands correspondence for $\Sp(W_{2m})$}
\subsection{Local Langlands correspondence for $\Sp(W_{2m})$}\label{LLC Sp}
Let $W_{2m}$ be a symplectic space over $F$ with dimension $2m$.
The associated symplectic group is denoted by $\Sp(W_{2m})$.
Fix an $F$-rational Borel subgroup $B'=T'U'$ of $\Sp(W_{2m})$.
By \cite[\S 12]{GGP}, there is a canonical bijection (depending on the choice of $\psi$)
\[
F^\times/F^{\times2} \rightarrow \{\text{$T'$-orbits of generic characters of $U'$}\},\ c\mapsto \mu'_c.
\]
\par

The Langlands dual group of $\Sp(W_{2m})$ is the complex Lie group $\SO(2m+1, \C)$, 
and $W_F$ acts on $\SO(2m+1, \C)$ trivially.
We denote the $L$-group of $\Sp(W_{2m})$ by ${}^L (\Sp(W_{2m})) = \SO(2m+1,\C) \times W_F$.
An $L$-parameter of $\Sp(W_{2m})$ is an admissible homomorphism
\[
\bphi \colon \WD_F \rightarrow {}^L (\Sp(W_{2m})) = \SO(2m+1,\C) \times W_F.
\]
We put 
\[
\Phi(\Sp(W_{2m})) = 
\{\text{$\SO(2m+1,\C)$-conjugacy classes of $L$-parameters of $\Sp(W_{2m})$}\}.
\]
For an $L$-parameter $\bphi \colon \WD_F \rightarrow {}^L(\Sp(W_{2m}))$,
by composing with the projection $\SO(2m+1,\C) \times W_F \rightarrow \SO(2m+1,\C)$, 
we obtain a map
\[
\phi \colon \WD_F \rightarrow \SO(2m+1,\C).
\]
The map $\bphi \mapsto \phi$ gives an identification
\[
\Phi(\Sp(W_{2m})) = \{\phi \colon \WD_F \rightarrow \SO(2m+1, \C)\}
/(\text{$\SO(2m+1, \C)$-conjugacy}).
\]
Namely, we regard $\Phi(\Sp(W_{2m}))$ as 
the set of equivalence classes of orthogonal representations of 
$\WD_F$ with dimension $2m+1$ and trivial determinant.
We denote the subset of $\Phi(\Sp(W_{2m}))$ consisting of equivalence classes of
tempered (\resp discrete, generic) representations by $\Phi_\temp(\Sp(W_{2n}))$
(\resp $\Phi_\disc(\Sp(W_{2m}))$, $\Phi_\gen(\Sp(W_{2m}))$).
Then we have a sequence 
\[
\Phi_\disc(\Sp(W_{2m})) \subset \Phi_\temp(\Sp(W_{2n})) \subset \Phi_\gen(\Sp(W_{2m})).
\]

The following theorem are due to Arthur \cite{Ar} supplemented by some results of many others
(c.f., see the proof of Proposition \ref{propSO}).
See also \cite[\S 3, \S 6.3]{At} and \cite[Theorem 3.3]{Ka}.
\begin{thm}
There exist a surjective map
\[
\Irr_{\temp}(\Sp(W_{2m})) \rightarrow \Phi_\temp(\Sp(W_{2m}))
\]
with the inverse image $\Pi_\phi$ of $\phi \in \Phi_\temp(\Sp(W_{2m}))$, 
and a bijection 
\[
\iota'_{c} \colon \Pi_\phi \rightarrow \widehat{A_\phi^+}
\]
for $c \in F^\times$ which satisfy analogues of 
Desideratum \ref{desO} (2), (4), and (6).
Moreover, using the Langlands classification, 
we can extend the map $\pi \mapsto \phi$ to a surjective map
\[
\Irr(\Sp(W_{2m})) \rightarrow \Phi(\Sp(W_{2m}))
\]
which satisfies an analogue of Desideratum \ref{desO} (7). 
In addition, an analogue to Proposition \ref{propSO} holds.
In particular, we have
\[
\iota'_{c_1}(\pi) = \iota'_{c_2}(\pi) \cdot \eta_{\phi,c_1/c_2}
\]
for $\pi \in \Pi_{\phi}$ and $c_1,c_2 \in F^\times$.
\end{thm}

Note that for $\phi \in \Phi(\Sp(W_{2n}))$, we have
\[
A_\phi = A_\phi^+ \oplus \pair{z_\phi}.
\]
Hence we may identify $\widehat{A_\phi^+}$ with
\[
(A_\phi/\pair{z_\phi})\widehat{\ } \subset \widehat{A_\phi}.
\]
Via this identification, we regard $\iota'_c$ as an injection
\[
\iota'_{c} \colon \Pi_\phi \rightarrow \widehat{A_\phi}.
\]
\par

Let $\phi_\tau$ be an orthogonal tempered representation of $\WD_F$, and
$\tau \in \Irr(\GL_k(F))$ be the tempered representation corresponding to $\phi_\tau$.
In \cite{Ar}, Arthur has defined
a normalized intertwining operator $R(w', \tau \otimes \pi)$ on 
$\Ind_Q^{\Sp(W_{2m'})}(\tau \otimes \pi)$ for $\pi \in \Irr_\temp(\Sp(W_{2m}))$, 
where $m'=m+k$ and $Q$ is a parabolic subgroup of $\Sp(W_{2m'})$ whose Levi subgroup is
$M_Q \cong \GL_k(F) \times \Sp(W_{2m})$.
See also \cite[\S 6. 3]{At}.
Note that $\Ind_Q^{\Sp(W_{2m'})}(\tau \otimes \pi)$ is multiplicity-free.
An analogue of Hypothesis \ref{Hypo} is given as follows:

\begin{prop}\label{R v.s. S}
Let $\phi_\pi \in \Phi_\temp(\Sp(W_{2m}))$ and $\pi \in \Pi_{\phi_\pi}$. 
We put $\phi_{\pi'} = \phi_\tau \oplus \phi_\pi \oplus \phi_\tau$.
We denote by $a' \in A_{\phi_{\pi'}}$ the element corresponding to $\phi_\tau$.
Let $\pi'$ be an irreducible constituent of $\Ind_Q^{\Sp(W_{2m'})}(\tau \otimes \pi)$.
Then we have
\[
R(w', \tau \otimes \pi)|\pi' = \iota'_1(\pi')(a').
\]
\end{prop}
\begin{proof}
This follows from Theorems 2.2.1 and 2.4.1 in \cite{Ar}.
\end{proof}

%\subsection{Local theta correspondence}
\subsection{Local theta correspondence}\label{theta}
We introduce the local theta correspondence induced by a Weil representation 
$\omega_{W,V,\psi}$ of $\Sp(W_{2m}) \times \Oo(V_{2n})$, and recall some basic general results.
\par

We have fixed a non-trivial additive character $\psi$ of $F$.
We denote a Weil representation of $\Sp(W_{2m}) \times \Oo(V_{2n})$ by 
$\omega=\omega_{W,V,\psi}$.
Let $\sigma \in \Irr(\Oo(V_{2n}))$.
Then the maximal $\sigma$-isotypic quotient of $\omega$ is of the form
\[
\Theta(\sigma) \boxtimes \sigma,
\]
where $\Theta(\sigma) = \Theta_{W,V,\psi}(\sigma)$ is a smooth representation of $\Sp(W_{2m})$.
It was shown by Kudla \cite{Ku} that $\Theta(\sigma)$ has finite length (possibly zero).
The maximal semi-simple quotient of $\Theta(\sigma)$ is denoted by 
$\theta(\sigma) = \theta_{W,V,\psi}(\sigma)$.
\par

Similarly, for $\pi \in \Irr(\Sp(W_{2m}))$, we obtain smooth finite length representations 
$\Theta(\pi)=\Theta_{V,W,\psi}(\pi)$ and $\theta(\pi)=\theta_{V,W,\psi}(\pi)$ of $\Oo(V_{2n})$.
The Howe duality conjecture, which was proven by Waldspurger \cite{W1} 
if the residue characteristic is not 2 and by Gan--Takeda \cite{GT1}, \cite{GT2} in general, says that 
$\theta(\sigma)$ and $\theta(\pi)$ are irreducible (if they are nonzero).

%\subsection{Prasad's conjecture}
\subsection{Prasad's conjecture}
Let $V$ be an orthogonal space associated to $(d,c)$, 
and $W$ be a symplectic space with $\dim(V)=\dim(W)=2n$.
We denote the discriminant character of $V$ by $\chi_V$.
Let $\phi \in \Phi(\Oo(V))$, and put
\[
\phi'=(\phi \oplus \1) \otimes \chi_V.
\]
Then we have $\phi' \in \Phi(\Sp(W))$.
Moreover we have a canonical injection $A_\phi \hookrightarrow A_{\phi'}$.
We denote the image of $a \in A_\phi$ by $a' \in A_{\phi'}$.
One should not confuse $z_\phi'$ with $z_{\phi'}$.
They satisfy $z_\phi'=e_1'+z_{\phi'}$, where $e_1' \in A_{\phi'}$ is the element 
corresponding to $\chi_V \subset \phi'$.
\par

\begin{lem}\label{kappa}
For any $\phi \in \Phi(\Oo(V))$, the map 
\[
A_\phi \hookrightarrow A_{\phi'} \twoheadrightarrow A_{\phi'}/\pair{z_{\phi'}}
\]
is surjective.
It is not injective if and only if $\phi$ contains $\1$.
In this case, the kernel of this map is generated by $e_1+z_{\phi}$, 
where $e_1 \in A_\phi$ is the element corresponding to $\1$.
\end{lem}
\begin{proof}
The map $A_\phi \hookrightarrow A_{\phi'}$ is not surjective 
if and only if $\phi \in \Phi(\Oo(V))$ and $\phi$ does not contain $\1$.
In this case the cokernel of this map is generated by $e_1'$.
Since $z_\phi'=e_1'+z_{\phi'}$, we have the surjectivity of
$A_\phi \hookrightarrow A_{\phi'} \twoheadrightarrow A_{\phi'}/\pair{z_{\phi'}}$.
\par

By comparing the order of $A_\phi$ with the one of $A_{\phi'}/\pair{z_{\phi'}}$, 
we see that $A_\phi \hookrightarrow A_{\phi'} \twoheadrightarrow A_{\phi'}/\pair{z_{\phi'}}$
 is not injective if and only if $\phi$ contains $\1$.
In this case, the order of the kernel is $2$.
Since $(e_1+z_\phi)' = e_1'+ z_{\phi'} = z_{\phi'}$
the kernel is generated by $e_1+z_\phi$.
\end{proof}

Prasad's conjecture is stated as follows:
%Prasad's conjecture
\begin{conj}[Prasad's conjecture for $(\Oo(V_{2n}), \Sp(W_{2n}))$]\label{P O}
Let $V$ and $W$ be an orthogonal space associated to $(d,c)$ 
and a symplectic space with $\dim(V) = \dim(W) = 2n$, respectively. 
We denote by $\chi_V = \chi_d$ the discriminant character of $V$.
Let $\phi \in \Phi(\Oo(V))$ and put $\phi'=(\phi \oplus \1)\otimes \chi_V \in \Phi(\Sp(W))$.
For $\sigma \in \Pi_\phi$, we have the following: 
\begin{enumerate}
\item
$\Theta_{W,V^\bullet,\psi}(\sigma)$ is zero if and only if 
$\phi$ contains $\1$ and 
$\iota_{c'}(\sigma) (z_\phi + e_1) = -1$,  
where $e_1 \in A_{\phi}$ is the element 
corresponding to $\1 \subset \phi$.
\item
Assume that $\pi = \theta_{W,V^\bullet,\psi}(\sigma)$ is nonzero.
Then $\pi \in \Pi_{\phi'}$ and $\iota'_{c'}(\pi) |A_\phi = \iota_{c'}(\sigma)$ for $c' \in F^\times$.
\end{enumerate}
\end{conj}
\par

\begin{rem}\label{2=>1}
\begin{enumerate}
\item
Recall that for $\pi \in \Pi_{\phi'} \subset \Irr(\Sp(W_m))$, 
the character $\iota'_{c'}(\pi)$ of $A_{\phi'}$ factors through $A_{\phi'}/\pair{z_{\phi'}}$.
By Lemma \ref{kappa}, we see that
$\iota'_{c'}(\pi)$ is determined completely by its restriction to $A_\phi$.
\item
By \cite[Theorem C.5]{GI1}, we know that
\begin{itemize}
\item
if $\phi$ does not contain $\1$, then both $\Theta_{W,V^\bullet,\psi}(\sigma)$
and $\Theta_{W,V^\bullet,\psi}(\sigma \otimes \det)$ are nonzero;
\item
if $\phi$ contains $\1$, then exactly one of $\Theta_{W,V^\bullet,\psi}(\sigma)$
or $\Theta_{W,V^\bullet,\psi}(\sigma \otimes \det)$ is nonzero;
\item
if $\pi = \theta_{W,V^\bullet,\psi}(\sigma)$ is nonzero, 
then $\pi \in \Pi_{\phi'}$.
\end{itemize}
Hence Conjecture \ref{P O} (1) follows from (2) since $z_\phi'=e_1'+z_{\phi'}$.
\end{enumerate}
\end{rem}

The first main theorem is as follows:
\begin{thm}\label{main1}
Assume Desideratum \ref{desO} and Hypothesis \ref{Hypo}.
Then Prasad's conjecture for $(\Oo(V_{2n}), \Sp(W_{2n}))$ (Conjecture \ref{P O}) holds.
In particular, it holds unconditionally 
when $V^\bullet = V$ and $c' \in cN_{E/F}(E^\times)$ with $E=F(\sqrt{d})$.
\end{thm}
\par

A weaker version of Prasad's conjecture (Conjecture \ref{P O}), 
which is formulated by using Weak LLC for $\SO(V)$ or its translation into $\Oo(V)$
(i.e., by using $A_\phi^+$), 
was proven by \cite{At} under Desideratum \ref{desSO} and Hypothesis \ref{Hypo}.

\begin{thm}[{\cite[\S 5.5]{At}}]\label{P SO}
Assume Desideratum \ref{desSO} and Hypothesis \ref{Hypo} for even $k$.
Let $\phi \in \Phi(\Oo(V))$ and put $\phi'=(\phi \oplus \1)\otimes \chi_V \in \Phi(\Sp(W))$ 
as in Conjecture \ref{P O}.
For $\sigma \in \Pi_\phi$, if $\pi = \theta_{W,V^\bullet,\psi}(\sigma)$ is nonzero, then
$\pi \in \Pi_{\phi'}$ and 
\[
\iota'_{c'}(\pi) |A_\phi^+ = \iota_{c'}(\sigma) |A_\phi^+ 
\]
for $c' \in F^\times$.
In particular, the same unconditionally holds 
when $V^\bullet = V$ and $c' \in cN_{E/F}(E^\times)$ with $E=F(\sqrt{d})$.
\end{thm}
\par

We may consider the theta correspondence for $(\Sp(W_{2n-2}), \Oo(V_{2n}))$.
There is also Prasad's conjecture for $(\Sp(W_{2n-2}), \Oo(V_{2n}))$.
\begin{conj}[Prasad's conjecture for $(\Sp(W_{2n-2}), \Oo(V_{2n}))$]\label{P Sp}
Let $V$ be an orthogonal space associated to $(d,c)$ with $\dim(V)=2n$, 
and $W$ be a symplectic space with $\dim(W)=2n-2$.
We denote by $\chi_V = \chi_d$ the discriminant character of $V$.
Let $\phi' \in \Phi(\Sp(W))$ and put $\phi=(\phi' \otimes \chi_V) \oplus \1 \in \Phi(\Oo(V))$.
For a companion space $V^\bullet$ of $V$, we put
\[
e(V^\bullet)=
\left\{
\begin{aligned}
&\chi_V(c'/c)	\iif V^\bullet = V,\\
&-\chi_V(c'/c)	\iif V^\bullet \not= V
\end{aligned}
\right.
\] 
Let $\pi \in \Pi_{\phi'}$.
\begin{enumerate}
\item
$\Theta_{V^\bullet, W, \psi}(\pi)=0$ if and only if 
$\phi'$ contains $\chi_V$ and $\iota'_{c'}(\pi)(e_1+z_{\phi'}) = -e(V^\bullet)$.
\item
Assume that $\sigma = \theta_{V^\bullet,W,\psi}(\pi)$ is nonzero.
Then $\sigma \in \Pi_\phi$ and so that there is a canonical injection $A_{\phi'} \hookrightarrow A_\phi$.
Moreover, $\iota_{c'}(\sigma)$ satisfies that
\begin{itemize}
\item
$\iota_{c'}(\sigma)(z_{\phi}) = e(V^\bullet)$;
\item
$\iota_{c'}(\sigma)| A_{\phi'} = \iota'_{c'}(\pi)$
for $c' \in F^\times$.
\end{itemize}
\end{enumerate}
\end{conj}

The following theorem shows that Conjecture \ref{P O} implies Conjecture \ref{P Sp}.
\begin{thm} \label{thmPSp}
Assume Desideratum \ref{desO} and Hypothesis \ref{Hypo}
(so that Conjecture \ref{P O} holds by Theorem \ref{main1}).
Then Prasad's conjecture for $(\Sp(W_{2n-2}), \Oo(V_{2n}))$ (Conjecture \ref{P Sp}) holds.
In particular, it holds unconditionally 
when $V^\bullet = V$ and $c' \in cN_{E/F}(E^\times)$ with $E=F(\sqrt{d})$.
\end{thm}
\begin{proof}
The equation $\iota_{c'}(\sigma)(z_{\phi}) = e(V^\bullet)$ follows from
Desideratum \ref{desO} (5) and Proposition \ref{propSO} (1).
Under assuming Desideratum \ref{desSO} and Hypothesis \ref{Hypo} for even $k$, 
the first author showed that $\iota_{c'}(\sigma)| A_{\phi'}^+ = \iota'_{c'}(\pi)| A_{\phi'}^+$
for $c' \in F^\times$ (\cite[Theorem 1.7]{At}).
Hence it suffices to show the equation
\[
\iota_{c'}(\sigma)(e_1+z_{\phi})=1, 
\]
where $e_1$ is the element of $A_{\phi}$ corresponding to $\1$.
This equation follows from Conjecture \ref{P O} (Theorem \ref{main1}) 
together with the tower property (see \cite{Ku}).
\end{proof}

%\subsection{Proof of Prasad's conjecture}
\subsection{Proof of Prasad's conjecture}
In this subsection, we prove Theorem \ref{main1}.
\par

Recall that there is a sequence
\[
\Phi^\epsilon_\temp(\Oo(V)) \subset \Phi_\temp(\Oo(V)) \subset \Phi(\Oo(V)).
\]
First, we reduce Conjecture \ref{P O} to the case when $\phi \in \Phi^\epsilon_\temp(\Oo(V))$.

\begin{lem}\label{reduce temp}
If Prasad's conjecture (Conjecture \ref{P O}) holds for any $\phi_0 \in \Phi_\temp(\Oo(V))$, 
then it holds for any $\phi \in \Phi(\Oo(V))$.
\end{lem}
\begin{proof}
This follows from a compatibility of LLC, Langlands quotients and theta lifts
(Desideratum \ref{des} (7) and \cite[Proposition C.4]{GI1}).
\end{proof}

\begin{lem}\label{reduce ep}
Assume Desideratum \ref{desSO} and Hypothesis \ref{Hypo}.
Then Prasad's conjecture (Conjecture \ref{P O}) holds for any 
$\phi \in \Phi_\temp(\Oo(V))\setminus \Phi^\epsilon_\temp(\Oo(V))$. 
\end{lem}
\begin{proof}
Since $A_\phi^+=A_\phi$, this follows from Theorem \ref{P SO}.
\end{proof}
\par

Hence Prasad's conjecture (Conjecture \ref{P O}) is reduced to 
the case when $\phi \in \Phi^\epsilon_\temp(\Oo(V))$.
For this case,
the following is the key proposition:

\begin{prop}\label{IR}
Let $V_{2n}$ and $W_{2n}$ be an orthogonal space associated to $(d,c)$ 
and a symplectic space with $\dim(V_{2n})=\dim(W_{2n})=2n$, respectively.
Fix a positive integer $k$.
For a companion space $V$ of $V_{2n}$, put $V' = V \oplus \H^k$.
Also we set $V_{2n+2k} = (V_{2n})'$, $W=W_{2n}$ and $W' = W_{2n+2k} = W \oplus \H^k$.
Let $\phi_\tau$ be an irreducible orthogonal tempered representation of $\WD_F$, 
and $\tau \in \Irr_\temp(\GL_k(F))$ be the corresponding representation.
For $\phi_\sigma \in \Phi^\epsilon_\temp(\Oo(V_{2n}))$, put 
\begin{align*}
\phi_{\sigma'} &= \phi_\tau \oplus \phi_\sigma \oplus \phi_\tau 
\in \Phi^\epsilon_\temp(\Oo(V_{2n+2k})), \\
\phi_\pi &= (\phi_\sigma \oplus \1) \otimes \chi_V \in \Phi_\temp(\Sp(W_{2n})) \text{ and }\\
\phi_{\pi'} &= (\phi_{\sigma'} \oplus \1) \otimes \chi_V = 
\phi_\tau\chi_V \oplus \phi_\pi \oplus \phi_\tau\chi_V
\in \Phi_\temp(\Sp(W_{2n+2k})).
\end{align*}
Let $\sigma \in \Pi_{\sigma}$, $\sigma' \in \Pi_{\sigma}$, $\pi \in \Pi_{\phi_\pi}$ 
and $\pi' \in \Pi_{\phi_\pi'}$ such that
$\sigma \in \Irr(\Oo(V))$ for a companion space $V$ of $V_{2n}$,
$\sigma' \subset \Ind_P^{\Oo(V')}(\tau \otimes \sigma)$ and
$\pi' \subset \Ind_Q^{\Sp(W')}(\tau\chi_V \otimes \pi)$, 
where $P \subset \Oo(V')$ and $Q \subset \Sp(W')$ are suitable parabolic subgroups.
Suppose that
\begin{itemize}
\item
$\phi_\tau$ is not the trivial representation of $\WD_F$;
\item
$\pi' = \theta_{W',V',\psi}(\sigma')$.
\end{itemize}
We denote by $a \in A_{\phi_{\sigma'}}$ and $a' \in A_{\phi_{\pi'}}$
the elements corresponding to $\phi_\tau$ and $\phi_\tau\chi_V$, respectively.
Then we have
\[
\iota_{c'}(\sigma')(a) = \iota_{c'}'(\pi')(a')
\]
for $c' \in F^\times$.
\end{prop}
\begin{proof}
The argument is similar to those of \cite[\S 8]{GI2} and \cite[\S 7]{At}, but it has one difference.
So we shall give a sketch of the proof.
\par

Let $\omega=\omega_{W,V,\psi}$ and $\omega'=\omega_{W',V',\psi}$.
We use a mixed model $\Sc'= \Sc(V' \otimes Y^*) \otimes \Sc \otimes \Sc(X^* \otimes W)$
for $\omega'$, where $\Sc$ is a space of $\omega$ (see \cite[\S 6.2]{At}).
For $\varphi \in \Sc'$, we define a map $\hat{f}(\varphi) \colon \Sp(W') \times \Oo(V') \rightarrow \Sc$
as in \cite[\S 8.1]{GI2} and \cite[\S 7.1]{At}.
By a similar argument to the proof of \cite[Theorem 8.1]{GS}, 
we have $\pi = \theta_{W,V,\psi}(\sigma)$ (see also \cite[Proposition C.4]{GI1}).
Fix a nonzero $\Sp(W) \times \Oo(V)$-equivariant map
\[
\TT_{00} \colon \omega \times \sigma^\vee \rightarrow \pi.
\]
For $\varphi \in \Sc'$, $\Phi_s \in \Ind_P^{\Oo(V')}(\tau|\cdot|^s_F \otimes \sigma^\vee)$, 
$g \in \Sp(W')$, $\ch{v} \in \tau^\vee$ and $\ch{v}_0 \in \pi^\vee$, 
consider the integral
\[
\pair{\TT_s(\varphi, \Phi_s)(g),\ch{v} \otimes \ch{v_0}}
\coloneqq
L(s+1,\tau)^{-1} \cdot
\int_{U_P\Oo(V)\bs\Oo(V')}\pair{\TT_{00}(\hat{f}(\varphi)(g,h),\pair{\Phi_s(h),\ch{v}}),\ch{v}_0}dh.
\]
Then one can show that
\begin{enumerate}
\item
the integral $\pair{\TT_s(\varphi, \Phi_s)(g),\ch{v} \otimes \ch{v_0}}$ 
is absolutely convergent for $\re(s) >-1$ and admits a holomorphic continuation to $\C$;
\item
$\TT_s$ gives an $\Sp(W') \times \Oo(V')$-equivalent map
\[
\TT_s \colon \omega' \otimes \Ind_P^{\Oo(V')}(\tau|\cdot|^s_F \otimes \sigma^\vee)
\rightarrow \Ind_Q^{\Sp(W')}(\tau\chi_V|\cdot|^s_F \otimes \pi).
\]
\end{enumerate}
See \cite[Lemmas 8.1--8.2]{GI2} and \cite[Proposition 7.2]{At}.
The one difference is that 
our case does not satisfy an analogue of \cite[Lemma 8.3]{GI2}.
So we have to modify this lemma.
One can show that
\begin{enumerate}
\setcounter{enumi}{2}
\item
if $L(-s,\tau^\vee)$ is regular at $s=0$, then
for any $\Phi \in \Ind_P^{\Oo(V')}(\tau \otimes \sigma^\vee)$ with $\Phi \not=0$, 
there exists $\varphi \in \omega'$ such that $\TT_0(\varphi',\Phi) \not=0$.
\end{enumerate}
Since $\phi_\tau$ is irreducible and tempered, 
$L(-s,\tau^\vee)$ is regular at $s=0$ if and only if 
$\phi_\tau$ is not the trivial representation of $\WD_F$.
\par

By the same calculation as \cite[Proposition 8.4]{GI2} and \cite[Corollary 7.4]{At}, 
one can show that
\begin{enumerate}
\setcounter{enumi}{3}
\item
for $\Phi\in\Ind_P^{\Oo(V')}(\tau\otimes \sigma^\vee)$ and $\varphi \in \omega'$, we have
\[
R(w',\tau\chi_V\otimes\pi)\TT_0(\varphi,\Phi)
=\omega_{\tau\chi_V}(c')\cdot
\TT_{0}(\varphi,R_{c'}(w,\tau\otimes\sigma^\vee)\Phi).
\]
\end{enumerate}
Here, we use the fact that
\[
\gamma_V^{-1} \cdot \lam(E/F,\psi) = e(V) \cdot \chi_V(c),
\]
where $\gamma_V$ is the Weil constant associated to $V$ which appears on 
the explicit formula for $\omega'$, 
and $\lam(E/F,\psi)$ is the Langlands constant which appears on 
the normalizing factor of $R_{c'}(w,\tau\otimes\sigma^\vee)$.
\par

By the same argument as \cite[Lemma 7.5]{At}, 
(3) and (4) together with Hypothesis \ref{Hypo} and Proposition \ref{R v.s. S} imply that
\begin{enumerate}
\setcounter{enumi}{4}
\item
$\iota_{c'}(\sigma')(a) = \omega_{\tau\chi_V}(c') \cdot \iota_{1}'(\pi')(a')$.
\end{enumerate}
Since $\omega_{\tau\chi_V}(c') = \det(\phi_\tau \chi_V)(c') = \det(\phi_{\pi'}^{a'})(c')$, 
we have
\begin{enumerate}
\setcounter{enumi}{5}
\item
$\omega_{\tau\chi_V}(c') \cdot \iota_{1}'(\pi')(a') = \iota_{c'}'(\pi')(a')$.
\end{enumerate}
The equations (5) and (6) imply the desired equation.
\end{proof}

Theorem \ref{P SO} and Proposition \ref{IR} imply Prasad's conjecture (Theorem \ref{main1}).
\begin{proof}[Proof of Theorem \ref{main1}]
By Remark \ref{2=>1} (2) and Lemmas \ref{reduce temp} and \ref{reduce ep}, 
we only consider Conjecture \ref{P O} (2) for $\phi_{\sigma} \in \Phi_\temp^\epsilon(\Oo(V))$.
Hence $\phi_\sigma$ contains an irreducible orthogonal representation $\phi_0$
with odd dimension $k_0$.
Put $\phi_\pi = (\phi_\sigma \oplus \1) \otimes \chi_V$.
Let $\sigma \in \Pi_{\phi_\sigma}$ and assume that 
$\pi = \theta_{W,V^\bullet,\psi}(\sigma)$ is nonzero.
Hence we have $\pi \in \Pi_{\phi_\pi}$.
Let $a_0 \in A_{\phi_\sigma}$ (\resp $a_0' \in A_{\phi_\pi}$) be the element corresponding to
$\phi_0$ (\resp $\phi_0\chi_V$).
Since $[A_{\phi_\sigma}:A_{\phi_\sigma}^+]=2$, by Theorem \ref{P SO}, 
it is enough to show that
\[
\iota_{c'}(\sigma)(a_0) = \iota_{c'}'(\pi)(a_0')
\]
for $c' \in F^\times$.
We choose an irreducible orthogonal tempered representation $\phi_\tau$ of $\WD_F$ 
such that
\begin{itemize}
\item
$\phi_\tau$ is not the trivial representation;
\item
$\phi_\tau$ is not contained in $\phi_\sigma$;
\item
$k = \dim(\phi_\tau)$ is odd.
\end{itemize}
Put $\phi_{\sigma'}=\phi_\tau \oplus \phi_\sigma \oplus \phi_\tau$ and
$\phi_{\pi'} = \phi_\tau \chi_V \oplus \phi_\pi \oplus \phi_\tau\chi_V$.
Let $a_\tau \in A_{\phi_{\sigma'}}$ (\resp $a_\tau' \in A_{\phi_{\pi'}}$) be the element 
corresponding to $\phi_\tau$ (\resp $\phi_\tau\chi_V$).
The claims (2) and (3) in the proof of Proposition \ref{IR}, 
there exist $\sigma' \subset \Ind_P^{\Oo({V^\bullet}')}(\tau \otimes \sigma)$ and
$\pi' \subset \Ind_Q^{\Sp(W')}(\tau\chi_V \otimes \pi)$ such that
$\pi'=\theta_{W',{V^\bullet}',\psi}(\sigma')$.
By Proposition \ref{IR}, we have
\[
\iota_{c'}(\sigma')(a_\tau) = \iota_{c'}'(\pi')(a_\tau').
\]
On the other hand, we know 
\[
\iota_{c'}(\sigma')(a_\tau+a_0) = \iota_{c'}'(\pi')(a_\tau' + a_0')
\]
by Theorem \ref{P SO}.
Hence we have
\[
\iota_{c'}(\sigma')(a_0) = \iota_{c'}'(\pi')(a_0').
\]
Since $\iota_{c'}(\sigma')|A_{\phi_\sigma} = \iota_{c'}(\sigma)$ and
$\iota_{c'}'(\pi')|A_{\phi_\pi} = \iota_{c'}'(\pi)$, we have
\[
\iota_{c'}(\sigma)(a_0) = \iota_{c'}'(\pi)(a_0)'.
\]
This completes the proof.
\end{proof}

\begin{rem}
One may feel that Prasad's conjecture (Conjecture \ref{P O}) can be proven
by a similar way to \cite[Theorem 1.7 (\S 7)]{At} without assuming 
the weaker version of Prasad's conjecture (Theorem \ref{P SO}).
However, because of the lack of an analogue of \cite[Lemma 8.3]{GI2}, 
the same method as \cite{At} can not be applied to Prasad's conjecture for 
$(\Oo(V_{2n}), \Sp(W_{2n}))$ when $\phi \in \Phi(\Oo(V_{2n}))$ contains $\1$.
\end{rem}

%\section{Gross--Prasad conjecture}
%\section{Gross--Prasad conjecture}
\section{Gross--Prasad conjecture}
Gross and Prasad gave a conjectural answer for a restriction problem for 
special orthogonal groups.
For the tempered case, this conjecture has been proven by Waldspurger
\cite{W2}, \cite{W3}, \cite{W4}, \cite{W5}.
In this section, we recall the Gross--Prasad conjecture and
consider an analogous restriction problem for orthogonal groups.

%\subsection{Local Langlands correspondence for $\Oo(V_{2n+1})$}
\subsection{Local Langlands correspondence for $\Oo(V_{2n+1})$}\label{LLC O(odd)}
Let $V_{m}$ be an orthogonal space of dimension $m$.
Recall that the discriminant of $V_m$ is defined by
\[
\disc(V_m) = 2^{-m}(-1)^{\half{m(m-1)}}\det(V_{m}) \in F^\times/F^{\times2}.
\]
An orthogonal space $V_m^\bullet$ is a companion space of $V_m$ if
$\dim(V_m^\bullet) = \dim(V_m)$ and $\disc(V_m^\bullet) = \disc(V_m)$.
\par

Let $V_{2n+1}$ be an orthogonal space over $F$ with dimension $2n+1$.
We denote the orthogonal group and the special orthogonal group associated to $V_{2n+1}$ 
by $\Oo(V_{2n+1})$ and $\SO(V_{2n+1})$, respectively.
Suppose that $\Oo(V_{2n+1})$ is split.
\par

We say that a representation $\phi$ of $\WD_F$ is symplectic if
$\phi$ admits a non-degenerate symplectic bilinear form which is $\WD_F$-invariant.
More precisely, see \cite[\S 3]{GGP}. 
\par

The Langlands dual group of $\SO(V_{2n+1})$ is the complex Lie group $\Sp(2n,\C)$,
and $W_F$ acts on $\Sp(2n,\C)$ trivially.
We denote the $L$-group of $\SO(V_{2n+1})$ by ${}^L(\SO(V_{2n+1})) = \Sp(2n,\C) \times W_F$.
An $L$-parameter of $\SO(V_{2n+1})$ is an admissible homomorphism
\[
\bphi \colon \WD_F \rightarrow {}^L(\SO(V_{2n+1})) = \Sp(2n,\C) \times W_F.
\]
We put 
\[
\Phi(\SO(V_{2n+1})) = 
\{\text{$\Sp(2n,\C)$-conjugacy classes of $L$-parameters of $\SO(V_{2n+1})$}\}.
\]
For an $L$-parameter $\bphi \colon \WD_F \rightarrow {}^L(\SO(V_{2n+1}))$, 
by composing with the projection $\Sp(2n,\C) \times W_F \rightarrow \Sp(2n,\C)$, 
we obtain a map
\[
\phi \colon \WD_F \rightarrow \Sp(2n,\C).
\]
The map $\bphi \mapsto \phi$ gives an identification
\[
\Phi(\SO(V_{2n+1})) = \{ \phi \colon \WD_F \rightarrow \Sp(2n, \C)\}/
(\text{$\Sp(2n,\C)$-conjugacy}).
\]
Namely, we regard $\Phi(\SO(V_{2n+1}))$ as the set of equivalence classes of 
symplectic representations of $\WD_F$ with dimension $2n$.
We denote the subset of $\Phi(\SO(V_{2n+1}))$ consisting of equivalence classes of tempered (\resp discrete, generic) representations by $\Phi_\temp(\SO(V_{2n+1}))$
(\resp $\Phi_\disc(\SO(V_{2n+1}))$, $\Phi_\gen(\SO(V_{2n+1}))$).
Then we have a sequence
\[
\Phi_\disc(\SO(V_{2n+1})) \subset \Phi_\temp(\SO(V_{2n+1})) \subset \Phi_\gen(\SO(V_{2n+1})).
\]
\par

We expect a similar desideratum for $\SO(V_{2n+1})$ to Desideratum \ref{desO}.
Namely, for $\phi \in \Phi(\SO(V_{2n+1}))$, we expect there are an $L$-packet 
\[
\Pi_{\phi}^0 \subset \bigsqcup_{V_{2n+1}^\bullet}\Irr(\SO(V_{2n+1}^\bullet)), 
\]
and a canonical bijection
\[
\iota \colon \Pi_{\phi}^0 \rightarrow \widehat{A_{\phi}},
\]
which satisfy similar properties to Desideratum \ref{desO}.
Here, $V_{2n+1}^\bullet$ runs over all companion spaces of $V_{2n+1}^\bullet$
Note that $A_\phi = A_\phi^+$ for $\phi \in \Phi(\SO(V_{2n+1}))$.
\par 

For the (quasi-)split case, it is known by Arthur \cite{Ar}.
\begin{thm}
There exist a surjective map
\[
\Irr_{\temp}(\SO(V_{2n+1})) \rightarrow \Phi_\temp(\SO(V_{2n+1}))
\]
with the inverse image $\Pi_\phi^{0}$ of $\phi \in \Phi_\temp(\SO(V_{2n+1}))$, 
and a canonical bijection 
\[
\iota \colon \Pi_\phi^0 \rightarrow (A_\phi^+/\pair{z_\phi})\widehat{\ }
\]
which satisfy analogues of 
Desideratum \ref{desO} (2), (4), and (6).
Moreover, using the Langlands classification, 
we can extend the map $\tau \mapsto \phi$ to a surjective map
\[
\Irr(\SO(V_{2n+1})) \rightarrow \Phi(\SO(V_{2n+1}))
\]
which satisfies an analogue of Desideratum \ref{desO} (7). 
\end{thm}
 
M{\oe}glin's work in \cite[\S 1.4, Theorem 1.4.1]{M2} 
seems to extend This theorem to the pure inner forms as well.
\par

Since $\Oo(V_{2n+1}^\bullet)$ is the direct product
\[
\Oo(V_{2n+1}^\bullet) = \SO(V_{2n+1}^\bullet) \times \{\pm\1_{V_{2n+1}^\bullet}\}, 
\]
any $\tau \in \Irr(\Oo(V_{2n+1}^\bullet))$ is determined by 
its restriction $\tau|\SO(V_{2n+1}^\bullet) \in \Irr(\SO(V_{2n+1}^\bullet))$ and
its central character $\omega_{\tau} \in \{\pm\1_{V_{2n+1}^\bullet}\}\widehat{\ } \cong \{\pm1\}$.
We define $\Phi(\Oo(V_{2n+1}))$ by
\[
\Phi(\Oo(V_{2n+1})) \coloneqq \Phi(\SO(V_{2n+1})) \times \{\pm1\}.
\]
For $(\phi,b) \in \Phi(\Oo(V_{2n+1}))$, we put
\[
\Pi_{\phi,b} = \{\tau \in \Irr(\Oo(V_{2n+1}^\bullet))\ |\ 
\tau|\SO(V_{2n+1}^\bullet) \in \Pi_\phi^0,\ \omega_\tau(-1) = b\}.
\]
Then we have a canonical bijection
\[
\Pi_{\phi,b} \xrightarrow{\Res} \Pi_\phi^0 \xrightarrow{\iota} \widehat{A_\phi}.
\]
which is also denoted by $\iota$.
Also we have
\[
\Pi_{\phi,-b} = \Pi_{\phi,b} \otimes \det.
\]

%\subsection{Gross--Prasad conjecture for special orthogonal groups}
\subsection{Gross--Prasad conjecture for special orthogonal groups}
In this subsection, we recall the Gross--Prasad conjecture.
\par

Let $V_{m+1}$  be an orthogonal space of dimension $m+1$,
and $V_{m}$ be a non-degenerate subspace of $V_{m+1}$ with codimension $1$.
We denote by $V_{\even}$ (\resp $V_{\odd}$) the space $V_m$ or $V_{m+1}$ such that
$\dim(V_{\even})$ is even (\resp $\dim(V_\odd)$ is odd). 
Suppose that $\SO(V_{m}) \times \SO(V_{m+1})$ is quasi-split.
We put $c=-\disc(V_{\odd})/\disc(V_{\even}) \in F^\times/F^{\times2}$.
Then $V_{\even}$ is associated to $(\disc(V_{\even}),c)$.
We say that a pair $(V_{m}^\bullet, V_{m+1}^\bullet)$ of companion spaces of $(V_m,V_{m+1})$
is relevant if $V_m^\bullet \subset V_{m+1}^\bullet$.
Then we have a diagonal map
\[
\Delta \colon \Oo(V_{m}^\bullet) \rightarrow \Oo(V_{m}^\bullet) \times \Oo(V_{m+1}^\bullet).
\]
By \cite{AGRS} and \cite{AGRS2}, 
for $\sigma_0 \in \Irr(\SO(V_{\even}^\bullet))$ and $\tau_0 \in \Irr(\SO(V_{\odd}^\bullet))$, 
we have
\[
\dim_\C\Hom_{\Delta\SO(V_{m}^\bullet)}(\sigma_0 \boxtimes \tau_0, \C) \leq 1.
\]
Choose $\epsilon \in \Oo(V_{m}^\bullet)$ such that $\det(\epsilon)=-1$.
We extend $\tau_0$ to an irreducible representation $\tau$ of $\Oo(V_{\odd}^\bullet)$.
For $\varphi \in \Hom_{\Delta\SO(V_{m}^\bullet)}(\sigma_0 \boxtimes \tau_0, \C)$, we put
\[
\varphi' = \varphi \circ (\1 \boxtimes \tau(\epsilon)).
\]
Then we have 
$\varphi' \in \Hom_{\Delta\SO(V_{m}^\bullet)}(\sigma_0^\epsilon \boxtimes \tau_0, \C)$, 
and the map $\varphi \mapsto \varphi'$ gives an isomorphism
\[
\Hom_{\Delta\SO(V_{m}^\bullet)}(\sigma_0 \boxtimes \tau_0, \C)
\cong
\Hom_{\Delta\SO(V_{m}^\bullet)}(\sigma_0^\epsilon \boxtimes \tau_0, \C).
\]
Therefore, $\dim_\C\Hom_{\Delta\SO(V_{m}^\bullet)}(\sigma_0 \boxtimes \tau_0, \C)$
depends only on 
\[
([\sigma_0], \tau_0) \in \Irr(\SO(V_{\even}))/\sim_\epsilon \times \Irr(\SO(V_{\odd})).
\]
The Gross--Prasad conjecture determines this dimension
in terms of Weak LLC for $\SO(V_\even)$ and LLC for $\SO(V_\odd)$.
\par

Let $\phi \in \Phi_\temp(\SO(V_{\even}))/\sim_\epsilon$ and $\phi' \in \Phi_\temp(\SO(V_{\odd}))$.
Following \cite[\S 6]{GGP}, for semi-simple elements $a \in C_\phi$ and $a' \in C_{\phi'}$, 
we put
\begin{align*}
\chi_{\phi'}(a) &= \ep(\phi^a \otimes \phi') \cdot \det(\phi^a)(-1)^{\half{1}\dim(\phi')},\\
\chi_{\phi}(a') &= \ep(\phi \otimes \phi'^{a'}) \cdot \det(\phi)(-1)^{\half{1}\dim(\phi'^{a'})}.
\end{align*}
Here, $\ep(\phi^a \otimes \phi') =  \ep(1/2, \phi^a \otimes \phi', \psi)$
and $\ep(\phi \otimes \phi'^{a'}) = \ep(1/2, \phi \otimes \phi'^{a'}, \psi)$
are the local root numbers, which are independent of the choice of $\psi$.
By \cite[Proposition 10.5]{GP}, $\chi_{\phi'}$ and $\chi_\phi$
define characters on $A_\phi$ and on $A_{\phi'}$, respectively.
\par

The following is a result of Waldspurger \cite{W2}, \cite{W3}, \cite{W4}, \cite{W5}.
\begin{thm}[Gross--Prasad conjecture for special orthogonal groups]\label{GPSO}
Let $V_{m+1}$  be an orthogonal space of dimension $m+1$,
and $V_{m}$ be a non-degenerate subspace of $V_{m+1}$ with codimension $1$.
Suppose that $\SO(V_{m}) \times \SO(V_{m+1})$ is quasi-split.
We put $c=-\disc(V_{\odd})/\disc(V_{\even}) \in F^\times/F^{\times2}$, so that
$V_{\even}$ is associated to $(\disc(V_{\even}),c)$.
Assume 
\begin{itemize}
\item
Weak LLC for $\SO(V_{\even})$ (Desideratum \ref{desSO}); 
\item
LLC for $\SO(V_{\odd})$ (an analogue of Desideratum \ref{desO} for $\SO(V_{\odd})$).
\end{itemize}
Let $\phi \in \Phi_\temp(\SO(V_{\even}))/\sim_\epsilon$ and $\phi' \in \Phi_\temp(\SO(V_{\odd}))$.
Then there exists a unique pair $([\sigma_0], \tau_0) \in \Pi_\phi^0 \times \Pi_{\phi'}^0$ such that
$\sigma_0 \boxtimes \tau_0$ is a representation of 
$\SO(V_{\even}^\bullet) \times \SO(V_{\odd}^\bullet)$ with 
a relevant pair $(V_{\even}^\bullet, V_{\odd}^\bullet)$ 
of companion spaces of $(V_{\even}, V_{\odd})$, and
\[
\Hom_{\Delta\SO(V_{2n}^\bullet)}(\sigma_0 \boxtimes \tau_0, \C) \not=0.
\]
Moreover, $\iota_c([\sigma_0]) \times \iota(\tau_0)$ satisfies that 
\[
\iota_c([\sigma_0]) \times \iota(\tau_0) = (\chi_{\phi'}|A_\phi^+) \times \chi_{\phi}.
\]
In particular, the same unconditionally holds for quasi-split $\SO(V_{m}) \times \SO(V_{m+1})$.
\end{thm}

%\subsection{Gross--Prasad conjecture for orthogonal groups}
\subsection{Gross--Prasad conjecture for orthogonal groups}
Let $V_{m+1}^\bullet$  be an orthogonal space of dimension $m+1$,
and $V_{m}^\bullet$ be a non-degenerate subspace of $V_{m+1}$ with codimension $1$.
In \cite{AGRS}, Aizenbud, Gourevitch, Rallis and Schiffmann
showed that
\[
\dim_\C\Hom_{\Delta\Oo(V_{m}^\bullet)}(\sigma \boxtimes \tau, \C) \leq 1
\]
for $\sigma \in \Irr(\Oo(V_{\even}^\bullet))$ and $\tau \in \Irr(\Oo(V_{\odd}^\bullet))$.
The following conjecture determines this dimension for
$(\sigma, \tau) \in \Irr_\temp(\Oo(V_{\even}^\bullet)) \times \Irr_\temp(\Oo(V_{\odd}^\bullet))$.
\par

Let $\phi \in \Phi(\Oo(V_{2n}))$.
For $b \in \{\pm1\}$ and $a \in C_\phi$, we put
\[
d_{\phi,b}(a) = b^{\dim(\phi^a)}.
\]
By \cite[\S 4]{GGP}, $d_{\phi,b}$ defines a character on $A_\phi$.
Note that $d_{\phi,b}$ is trivial on $A_\phi^+$.

%Gross--Prasad conjecture for orthogonal groups
\begin{conj}[Gross--Prasad conjecture for orthogonal groups]\label{GPO}
Let $V_{m+1}$ be an orthogonal space of dimension $m+1$,
and $V_{m}$ be a non-degenerate subspace of $V_{m+1}$ with codimension $1$.
Suppose that $\Oo(V_{m}) \times \Oo(V_{m+1})$ is quasi-split.
We put $c=-\disc(V_{\odd})/\disc(V_{\even}) \in F^\times/F^{\times2}$, so that
$V_{\even}$ is associated to $(\disc(V_{\even}),c)$.
Let $\phi \in \Phi_\temp(\Oo(V_{\even}))$ and $(\phi',b) \in \Phi_\temp(\Oo(V_{\odd}))$.
Then there exists a unique pair $(\sigma, \tau) \in \Pi_\phi \times \Pi_{\phi',b}$ such that
$\sigma \boxtimes \tau$ is a representation of $\Oo(V_\even^\bullet) \times \Oo(V_\odd^\bullet)$
with a relevant pair $(V_{\even}^\bullet, V_{\odd}^\bullet)$ 
of companion spaces of $(V_{\even}, V_{\odd})$, and
\[
\Hom_{\Delta\Oo(V_{m}^\bullet)}(\sigma \boxtimes \tau, \C) \not=0.
\]
Moreover, $\iota_c(\sigma) \times \iota(\tau)$ satisfies that
\[
\iota_c(\sigma) \times \iota(\tau) = (\chi_{\phi'}\cdot d_{\phi,b}) \times \chi_{\phi}.
\]
\end{conj}

\begin{rem}\label{SOO}
Let $V_{2n+1} = V_{2n} \oplus L$ be an orthogonal space of dimension $2n+1$,
and $V_{2n}$ be a non-degenerate subspace of $V_{2n+1}$ with codimension $1$.
The stabilizer of the line $L$ in $\SO(V_{2n+1})$ is the subgroup:
\[
\mathrm{S}(\Oo(V_{2n}) \times \Oo(L))
= \{(g_1,g_2) \in \Oo(V_{2n}) \times \Oo(L)\ |\ \det(g_1) = \det(g_2)\}, 
\]
which is isomorphic to $\Oo(V_{2n})$ by the first projection.
Then the restriction problem of $\Oo(V_{2n}) \subset \Oo(V_{2n+1})$ 
is equivalent to the one of $\mathrm{S}(\Oo(V_{2n}) \times \Oo(L)) \subset \SO(V_{2n+1})$.
Indeed, let $\tau$ be an irreducible representation of $\SO(V_{2n+1})$, 
and $\tau^b$ be the extension of $\tau$ to $\Oo(V_{2n+1})$ 
satisfying $\tau^b(-\1_{V_{2n+1}}) = b \cdot \id$ for $b \in \{\pm\}$.
For $\sigma \in \Irr(\Oo(V_{2n}))$, define $\sigma^b \in \Irr(\Oo(V_{2n}))$
by 
\[
\sigma^b(g) = 
\left\{
\begin{aligned}
&\sigma(g)	\iif \det(g) = 1,\\
&b \cdot \omega_{\sigma}(-1) \cdot \sigma(g)	\iif \det(g) = -1.
\end{aligned}
\right.
\]
Here, $\omega_\sigma$ denotes the central character of $\sigma$, 
which is regarded as a character of $\{\pm1\}$.
We regard $\sigma^b$ as an irreducible representation of $\mathrm{S}(\Oo(V_{2n}) \times \Oo(L))$
by pulling back via the first projection.
Then we have an identification
\[
\Hom_{\Oo(V_{2n})}(\tau^b \otimes \sigma, \C) 
= 
\Hom_{\mathrm{S}(\Oo(V_{2n}) \times \Oo(L))}(\tau \otimes \sigma^b, \C).
\]
Using this equation, we see that
a result of Prasad (\cite[Theorem 4]{P3}) follows from Conjecture \ref{GPO} for $m=2$.
\end{rem}

In \S \ref{low}, we review another result of Prasad \cite{P3} for a low rank case and 
check that it is compatible with Conjecture \ref{GPO}.
In \S \ref{GPOproof}, we will prove Conjecture \ref{GPO} under assuming 
LLC for $\Oo(V_m) \times \Oo(V_{m+1})$  and Hypothesis \ref{Hypo}.

%\subsection{Low rank cases}\label{low}
\subsection{Low rank cases}\label{low}
In \cite{P3}, D. Prasad extended a theorem on 
trilinear forms of three representations of $\GL_2(F)$ (\cite[Theorem 1.4]{P1}).
In this subsection, we check this theorem follows from Conjecture \ref{GPO}.
\par

First, we recall a theorem on trilinear forms.
Let $D$ be the (unique) quaternion division algebra over $F$.
For an irreducible representation $\pi$ of $\GL_2(F)$, 
let $\pi'$ be the Jacquet--Langlands lift of $\pi$
if $\pi$ is an essentially discrete series representation, 
and put $\pi' = 0$ otherwise.
Also, for a representation $\phi$ of $\WD_F$, 
if $\det(\phi) = \1$, we write $\ep(\phi) = \ep(1/2, \phi, \psi)$, 
which is independent of a non-trivial additive character $\psi$ of $F$.

\begin{thm}[{\cite[Theorem 1.4]{P1}}]\label{triple}
For $i=1,2,3$, 
let $\pi_i$ be an irreducible infinite-dimensional representation of 
$\GL_2(F)$ with central character $\omega_{\pi_i}$.
Assume that $\omega_{\pi_1} \omega_{\pi_2} \omega_{\pi_3} = \1$.
We denote the representation of $\WD_F$ corresponding to $\pi_i$ by $\phi_i$.
Then:
\begin{itemize}
\item
there exists a nonzero $\GL_2(F)$-invariant linear form on $\pi_1 \otimes \pi_2 \otimes \pi_3$
if and only if $\ep(\phi_1 \otimes \phi_2 \otimes \phi_3) = 1$;
\item
there exists a nonzero $D^\times$-invariant linear form on $\pi'_1 \otimes \pi'_2 \otimes \pi'_3$
if and only if $\ep(\phi_1 \otimes \phi_2 \otimes \phi_3) = -1$.
\end{itemize}
\end{thm}

When $\pi_i$ is tempered for each $i$ and $\omega_{\pi_1} \omega_{\pi_2} = \omega_{\pi_3} = \1$, 
this theorem is a special case of GP conjecture for special orthogonal groups (Theorem \ref{GPSO}).
Recall that there exist two exact sequences
\begin{align*}
&\begin{CD}
1 @>>> F^\times @>{\Delta_1}>> \GL_2(F) \times \GL_2(F) 
@>{\rho_1}>> \mathrm{GSO}(V_4) @>>> 1,
\end{CD}
\\
&\begin{CD}
1 @>>> F^\times @>{\Delta_2}>> D^\times \times D^\times 
@>{\rho_2}>> \mathrm{GSO}(V_4') @>>> 1,
\end{CD}
\end{align*}
where
\begin{itemize}
\item
$V_4 =  \mathrm{M}_2(F)$ which is regarded as an orthogonal space with
\[
(
\begin{pmatrix}
a_1,b_1\\ c_1,d_1
\end{pmatrix},
\begin{pmatrix}
a_2,b_2\\ c_2,d_2
\end{pmatrix}
)_{V_4} = \tr(
\begin{pmatrix}
a_1,b_1\\ c_1,d_1
\end{pmatrix}
\begin{pmatrix}
d_2,-b_2\\ -c_2,a_2
\end{pmatrix}
);
\]
\item
$V_4'=D$ is regarded as an orthogonal space with
\[
(x,y)_{V_4'} = \tau(xy^*), 
\]
where $y \mapsto y^*$ is the main involution and $\tau(x)$ is the reduced trace of $x$;
\item
for an orthogonal space $V_{2n}$ with even dimension $2n$, 
the similitude special orthogonal group $\mathrm{GO}(V_{2n})$ is defined by
\[
\mathrm{GO}(V_{2n}) = \{ g \in \GL(V_{2n})\ |\ 
\pair{gv_1, gv_2}_{V_{2n}} = \nu(g)\pair{v_1,v_2}_{V_{2n}},\ \nu(g) \in F^\times
\text{ for any $v_1, v_2 \in V_{2n}$}\}
\]
and $\mathrm{GSO}(V_{2n})$ is defined by
\[
\mathrm{GSO}(V_{2n}) = \{g \in \mathrm{GO}(V_{2n})\ |\ \det(g) = \nu(g)^n\};
\]
\item
$\Delta_1$ and $\Delta_2$ are the diagonal embeddings;
\item
$\rho_1$ and $\rho_2$ are given by
\[
\rho_1(g_1,g_2)x = g_1 x g_2^{-1},
\quad
\rho_2(g'_1,g'_2)x' = g'_1 x' {g'_2}^{-1}
\]
for $g_1,g_2 \in \GL_2(F)$, $x \in V_4$, $g'_1,g'_2 \in D^\times$ and $x' \in V'_4$.
\end{itemize}
Hence if $\omega_{\pi_1} \omega_{\pi_2} = \omega_{\pi_3} = \1$, 
then $\pi_1 \otimes \pi_2$ (\resp $\pi'_1 \otimes \pi'_2$) is regarded as 
a representation $\cl{\sigma}$ of $\mathrm{GSO}(V_{4})$ 
(\resp $\cl{\sigma}'$ of $\mathrm{GSO}(V'_4)$).
The restriction of $\cl{\sigma}$ to $\SO(V_4)$ 
(\resp $\cl{\sigma}'$ to $\SO(V'_4)$) decomposes into a direct sum of irreducible representations, 
i.e., 
\[
\cl{\sigma}|\SO(V_{4}) = \sigma_1 \oplus \dots \oplus \sigma_r
\quad
(\resp 
\cl{\sigma}'|\SO(V'_{4}) = \sigma'_1 \oplus \dots \oplus \sigma'_{r'}
).
\]
Let $\phi_\sigma = \phi_1 \otimes \phi_2$.
Then we have $\phi_\sigma \in \Phi(\SO(V_{4}))/\sim_\epsilon$, 
and the $L$-packet $\Pi_{\phi_\sigma}^0$ is given by
\[
\Pi_{\phi_\sigma}^0 = \{[\sigma_i]\ |\ i=1, \dots, r\} \cup \{[\sigma'_i]\ |\ i=1,\dots, r'\}.
\]
Here, if $\pi'_1 \otimes \pi'_2 = 0$, we neglect $\{[\sigma'_i]\}$.
\par

On the other hand, if $\omega_{\pi_3} = \1$, 
then $\pi_3$ (\resp $\pi'_3$) is regarded as 
a representation $\tau$ of $\SO(V_{3})$ 
(\resp $\tau'$ of $\SO(V'_3)$), 
where $V_3$ (\resp $V'_3$) is the orthogonal space of dimension $3$, discriminant $-1$, 
and such that $\SO(V_3)$ is split (\resp $\SO(V'_3)$ is not split).
We identify $V_3$ (\resp $V'_3$) with the subspace of $V_4$ (\resp $V'_4$)
consisting of trace zero elements (\resp reduced trace zero elements).
Let $\phi_\tau = \phi_3 \colon \WD_F \rightarrow \SL_2(\C)$.
Then we have $\phi_\tau \in \Phi(\SO(V_{3}))$, 
and the $L$-packet $\Pi_{\phi_\tau}^0$ is given by
\[
\Pi_{\phi_\tau}^0 = \{ \tau,\ \tau'\}.
\]
Here, if $\pi'_3 = 0$, we neglect $\tau'$.
\par

We embed $\GL_2(F)$ into $\GL_2(F) \times \GL_2(F)$
(\resp $D^\times$ into $D^\times \times D^\times$) as the diagonal subgroup.
This embedding induces the inclusion $\SO(V_3) \hookrightarrow \SO(V_4)$ 
(\resp $\SO(V'_3) \hookrightarrow \SO(V'_4)$).
Then we conclude that
\begin{itemize}
\item
there exists a nonzero $\GL_2(F)$-invariant linear form on $\pi_1 \otimes \pi_2 \otimes \pi_3$
if and only if 
$\Hom_{\SO(V_3)}(\sigma_i \otimes \tau, \C) \not=0$ for some $i=1,\dots, r$;
\item
there exists a nonzero $D^\times$-invariant linear form on $\pi'_1 \otimes \pi'_2 \otimes \pi'_3$
if and only if
$\Hom_{\SO(V'_3)}(\sigma'_i \otimes \tau', \C) \not=0$ for some $i=1,\dots, r'$;
\item
$\ep(\phi_1 \otimes \phi_2 \otimes \phi_3) = \ep(\phi_\sigma \otimes \phi_\tau) = 
\chi_{\phi_\tau}(z_{\phi_\sigma}) = \chi_{\phi_\sigma}(z_{\phi_\tau})$.
\end{itemize}
Hence Theorem \ref{GPSO} implies
Theorem \ref{triple} for tempered $\pi_1$, $\pi_2$ and $\pi_3$ 
when $\omega_{\pi_1}\omega_{\pi_2} = \omega_{\pi_3} = \1$.
\par

Next, we recall \cite[Theorem 3]{P3}, which is an extension of Theorem \ref{triple}.
This is the case when $\pi_1 = \pi_2$.
Note that $\pi_1 \otimes \pi_1 = \Sym^2(\pi_1) \oplus \wedge^2(\pi_1)$.
Also, for a representation $\phi_1$ of $\WD_F$, we have 
$\phi_1 \otimes \phi_1 = \Sym^2(\phi_1) \oplus \wedge^2(\phi_1)$.

\begin{thm}[{\cite[Theorem 3]{P3}}] \label{triple2}
Let $\pi_1$ and $\pi_3$ be irreducible admissible infinite-dimensional representations of $\GL_2(F)$.
Assume that $\omega_{\pi_1}^2 \omega_{\pi_3} = \1$.
We denote the representation of $\WD_F$ corresponding to $\pi_i$ by $\phi_i$. 
Then:
\begin{itemize}
\item
$\Sym^2(\pi_1) \otimes \pi_3$ has a $\GL_2(F)$-invariant linear form 
if and only if $\ep(\Sym^2(\phi_1) \otimes \phi_3) = \omega_{\pi_1}(-1)$ and 
$\ep(\wedge^2(\phi_1) \otimes \phi_3) = \omega_{\pi_1}(-1)$;
\item
$\wedge^2(\pi_1) \otimes \pi_3$ has a $\GL_2(F)$-invariant linear form 
if and only if $\ep(\Sym^2(\phi_1) \otimes \phi_3) = -\omega_{\pi_1}(-1)$ and 
$\ep(\wedge^2(\phi_1) \otimes \phi_3) = -\omega_{\pi_1}(-1)$.
\end{itemize}
\end{thm}

We check that GP conjecture (Conjecture \ref{GPO}) implies 
this theorem for tempered $\pi_1$ and $\pi_3$ such that
$\omega_{\pi_1}^2 = \1$ and $\omega_{\pi_3} = \1$.
Consider the group
\[
(\GL_2(F) \times \GL_2(F)) \rtimes \pair{c}, 
\]
where $c^2 = 1$ and $c$ acts on $\GL_2(F) \times \GL_2(F)$ by
the exchange of the two factors of $\GL_2(F)$.
Then $\rho_1 \colon \GL_2(F) \times \GL_2(F) \rightarrow \mathrm{GSO}(V_4)$ gives a surjection
\[
\rho_1 \colon (\GL_2(F) \times \GL_2(F)) \rtimes \pair{c} \rightarrow 
\mathrm{GSO}(V_4) \rtimes \pair{c} \cong \mathrm{GO}(V_4).
\]
Here, we identify $c$ as the element in $\Oo(V_4)$ which acts on $V_3$ by $-1$
and on the orthogonal complement of $V_3$ by $+1$.
There are two extensions of the representation
$\cl{\sigma}|\SO(V_{4}) = \sigma_1 \oplus \dots \oplus \sigma_r$
of $\SO(V_4)$ on $\pi_1 \otimes \pi_1 = \Sym^2(\pi_1) \oplus \wedge^2(\pi_1)$
to $\Oo(V_4)$.
We denote by $\cl{\sigma}^\pm|\Oo(V_{4}) = \sigma^\pm_1 \oplus \dots \oplus \sigma^\pm_r$
the extension such that $c$ acts on $\Sym^2(\pi_1)$ by $\pm1$ and
on $\wedge^2(\pi_1)$ by $\mp1$, respectively.
On the other hand, the representation $\tau$ of $\SO(V_3)$ on $\pi_3$ has
two extensions $\tau^\pm$ to $\Oo(V_3)$, which satisfies that $\tau^\pm(-\1_{V_3}) = \pm\id$.
Since $c$ centralizes the diagonal subgroup $\GL_2(F)$ of $\GL_2(F) \times \GL_2(F)$, 
we see that
\begin{itemize}
\item
$\Sym^2(\pi_1) \otimes \pi_3$ has a $\GL_2(F)$-invariant linear form 
if and only if 
$\Hom_{\Oo(V_3)}(\sigma_i^+ \otimes \tau^+, \C) \not=0$
for some $i=1, \dots, r$;
\item
$\wedge^2(\pi_1) \otimes \pi_3$ has a $\GL_2(F)$-invariant linear form 
if and only if 
$\Hom_{\Oo(V_3)}(\sigma_i^+ \otimes \tau^-, \C) \not=0$
for some $i=1, \dots, r$.
\end{itemize}
\par

Since $\pi_1$ is generic, 
$\sigma_i^\pm$ is $\mu_{a}^{b}$-generic for some $a \in F^\times$ and $b \in \{\pm\}$.
We claim that $b = \pm \omega_{\pi_1}(a)$.
Fix a nonzero $\psi$-Whittaker functional $l \colon \pi_1 \rightarrow \C$
and $x \in \pi_1$ such that $l(x) \not= 0$.
For each $a \in F^\times$, 
put
\[
l_a = l \circ \pi_1(
\begin{pmatrix}
a & 0 \\ 0 & 1
\end{pmatrix}
),
\quad
x_a = \pi_1(
\begin{pmatrix}
a & 0 \\ 0 & 1
\end{pmatrix}^{-1}
)x.
\]
We define a basis $\{v,e_a,e_a', v'\}$ of $V_4 = \mathrm{M}_2(F)$ by 
\[
v = \begin{pmatrix}
0&1\\ 0 & 0
\end{pmatrix},\ 
e_a = \begin{pmatrix}
1& 0 \\ 0 & a
\end{pmatrix},\ 
e_a' = \begin{pmatrix}
1& 0 \\ 0 & -a
\end{pmatrix},\ 
v = \begin{pmatrix}
0&0 \\ -1 & 0
\end{pmatrix}.
\]
This basis makes $V_4$ the orthogonal space associated to $(1,a)$.
Also we have
\[
\rho_1(
\begin{pmatrix}
1 & n_1 \\ 0 & 1
\end{pmatrix},
\begin{pmatrix}
1 & n_2 \\ 0 & 1
\end{pmatrix}
)
(v,e_a,e_a',v')
= 
(v,e_a,e_a',v')
\begin{pmatrix}
1 & n_1a-n_2 & -n_1a-n2 & -n_1n_2\\
0 & 1 & 0 & \half{-n_1 + n_2a^{-1}}\\
0 & 0& 1 &\half{-n_1-n_2a^{-1}}\\
0 & 0 & 0 & 1
\end{pmatrix}.
\]
Hence $l_a \otimes l_{-1} \colon \pi_1 \otimes \pi_1 \rightarrow \C$ gives 
a $\mu_a$-Whittaker functional $l_0$ on $\cl{\sigma}^\pm$.
Let $v_0 = x_a \otimes x_{-1} \in \cl{\sigma} = \pi_1 \otimes \pi_1$.
Note that $V_3$ is the orthogonal complement of $Fe_1$ in $V_4$.
Since we regard $c$ as the nontrivial element in the center of $\Oo(V_3)$, 
it acts on $v$, $e_1$, $e_1'$, and $v'$ by 
\[
cv = -v,\ ce_1 = e_1,\ ce_1' = -e_1',\ cv'=-v'.
\]
We define $\epsilon_a \in \Oo(V_4)$ so that 
$\epsilon_a v = v$, $\epsilon_a e_a = e_a$, $\epsilon_a e_a' = -e_a'$, and $\epsilon_a v'=v'$.
Then 
\[
\epsilon_a = \rho_1(
\begin{pmatrix}
-1 & 0 \\ 0 & a
\end{pmatrix},
\begin{pmatrix}
-a & 0 \\ 0 & 1
\end{pmatrix}
)
\cdot c.
\]
Therefore we have
\begin{align*}
b \cdot l_0(v_0)
&= l_0(\epsilon_a v_0)
\\&= (l_a \otimes l_{-1}) \circ \pi_1(
\begin{pmatrix}
-1 & 0 \\ 0 & a
\end{pmatrix}) \otimes \pi_1(
\begin{pmatrix}
-a & 0 \\ 0 & 1
\end{pmatrix}
)
\circ c
(x_a \otimes x_{-1})
\\&= \pm \omega_{\pi_1}(a) \cdot
(l_a \otimes l_{-1}) \circ \pi_1(
\begin{pmatrix}
-a^{-1} & 0 \\ 0 & 1
\end{pmatrix}) \otimes \pi_1(
\begin{pmatrix}
-a & 0 \\ 0 & 1
\end{pmatrix}
)
(x_{-1} \otimes x_{a})
\\&=
\pm \omega_{\pi_1}(a) \cdot
(l_a \otimes l_{-1})(x_{a} \otimes x_{-1})
\\&= \pm \omega_{\pi_1}(a) \cdot l_0(v_0).
\end{align*}
Hence $b = \pm \omega_{\pi_1}(a)$, as desired.
\par

By Desideratum \ref{propO} (1) and (3), 
we have 
\[
\iota_{1}(\sigma_i^\pm)(s) = (\pm\omega_{\pi_1}(a))^{\det(s)} \cdot \eta_{\phi_\sigma, a}(s)
\]
for $s \in A_{\phi_\sigma}$,
In particular, if we denote the element in $A_{\phi_\sigma}$ corresponding to 
$\wedge^2(\phi_1) = \det(\phi_1) = \omega_{\pi_1}$ by $s_0$, 
we have
\[
\iota_{1}(\sigma_i^\pm)(s_0) = \pm1.
\]
Also we have
\[
(\chi_{\phi_\tau} \cdot d_{\phi_\sigma, \pm1}) (s_0) = 
\pm \ep(\wedge(\phi_1) \otimes \phi_3) \cdot \omega_{\pi_1}(-1).
\]
Hence by GP conjecture (Conjecture \ref{GPO}), we see that:
\begin{itemize}
\item
$\Hom_{\Oo(V_3)}(\sigma_i^+ \otimes \tau^\pm, \C) \not=0$ for some $i=1, \dots, r$
if and only if
$\pm \ep(\wedge(\phi_1) \otimes \phi_3) \cdot \omega_{\pi_1}(-1) = +1$.
\end{itemize}
This implies Theorem \ref{triple2} for tempered $\pi_1$ and $\pi_3$
such that $\omega_{\pi_1}^2 = \omega_{\pi_3} = \1$.
\par

In fact, there is an analogous theorem (\cite[Theorem 6]{P4}) for the quaternion algebra case 
(in which case the product of the two root numbers is $-1$).
If one knew $\iota_1({\sigma_i'}^\pm)$ explicitly
(as we have done for $\sigma_i^\pm$ by using Desideratum \ref{propO} (3)), 
one would show that this theorem follows from Conjecture \ref{GPO}.
Conversely, by using Prasad's theorem \cite[Theorem 6]{P4} and Conjecture \ref{GPO}, 
we may conclude that
\[
\iota_1({\sigma_i'}^\pm)(s_0) = \mp1.
\]

%\subsection{Proof of Conjecture \ref{GPO}}
\subsection{Proof of Conjecture \ref{GPO}}\label{GPOproof}
In this subsection, we prove that Prasad's conjecture (Conjecture \ref{P O}) implies 
the Gross--Prasad conjecture (Conjecture \ref{GPO}).
The second main theorem is as follows:
\begin{thm}\label{main2}
Assume 
\begin{itemize}
\item
LLC for $\Oo(V_m) \times \Oo(V_{m+1})$ 
(Desideratum \ref{desO} and the analogue of Desideratum \ref{des} for $\SO(V_{\odd})$);
\item
Hypothesis \ref{Hypo}
(which implies Prasad's conjecture \ref{P O} by Theorem \ref{main1}).
\end{itemize}
Then the Gross--Prasad conjecture (Conjecture \ref{GPO}) holds. 
In particular, it unconditionally holds for quasi-split $\Oo(V_m) \times \Oo(V_{m+1})$.
\end{thm}
\par

First, we consider the case when $m=2n$ is even.
We need the following lemma for this case:
\begin{lem}\label{lemma1}
Let $\sigma_0\in \Irr(\SO(V_{2n}^\bullet))$ and $\tau_0 \in \Irr(\SO(V_{2n+1}^\bullet))$.
For $b \in \{\pm1\}$, we denote by $\tau^b$ the extension of $\tau$ to $\Oo(V_{2n+1}^\bullet)$
such that $\tau^b(-\1_{V_{2n+1}})=b\cdot \id$.
Assume that
\[
\Hom_{\SO(V_{2n}^\bullet)}(\sigma_0 \otimes \tau_0, \C) \not=0.
\]
\begin{enumerate}
\item
There exists a unique irreducible constituent $\sigma^b$ of 
$\Ind_{\SO(V_{2n}^\bullet)}^{\Oo(V_{2n}^\bullet)}(\sigma_0)$ such that
\[
\Hom_{\Oo(V_{2n}^\bullet)}(\sigma^b \otimes \tau^b, \C) \not=0.
\]
\item
If $\sigma_0^\epsilon \cong \sigma_0$, then
the correspondence 
\[
\{\pm1\}\ni b \mapsto \sigma^b \in \{\text{irreducible constituents of 
$\Ind_{\SO(V_{2n}^\bullet)}^{\Oo(V_{2n}^\bullet)}(\sigma_0)$}\}
\] 
is bijective.
\end{enumerate}
\end{lem}
\begin{proof}
By the Frobenius reciprocity, we have
\begin{align*}
0 &\not= \Hom_{\SO(V_{2n}^\bullet)}(\sigma_0 \otimes \tau_0, \C)
\cong\Hom_{\SO(V_{2n}^\bullet)}(\tau_0, \sigma_0^\vee)\\
&\cong\Hom_{\SO(V_{2n}^\bullet)}(\tau^b|\SO(V_{2n}^\bullet), \sigma_0^\vee)
\cong\Hom_{\Oo(V_{2n}^\bullet)}
(\tau^b|\Oo(V_{2n}^\bullet), \Ind_{\SO(V_{2n}^\bullet)}^{\Oo(V_{2n}^\bullet)}(\sigma_0^\vee))\\
&\cong\Hom_{\Oo(V_{2n}^\bullet)}
(\Ind_{\SO(V_{2n}^\bullet)}^{\Oo(V_{2n}^\bullet)}(\sigma_0) \otimes \tau^b, \C)
\end{align*}
for any $b \in \{\pm1\}$.
Hence if $\sigma_0^\epsilon \not\cong \sigma_0$, then
$\Ind_{\SO(V_{2n}^\bullet)}^{\Oo(V_{2n}^\bullet)}(\sigma_0)$ is irreducible,
so that the first assertion is trivial.
\par

Next, suppose that $\sigma_0^\epsilon \cong \sigma_0$.
Then $\Ind_{\SO(V_{2n}^\bullet)}^{\Oo(V_{2n}^\bullet)}(\sigma_0) \cong \sigma_1 \oplus \sigma_2$, 
where $\sigma_1, \sigma_2 \in \Irr(\Oo(V_{2n}^\bullet))$ satisfy 
$\sigma_1 \not\cong \sigma_2$ and $\sigma_i|\SO(V_{2n}^\bullet) = \sigma_0$ for $i \in \{1,2\}$.
Let $\epsilon \in \Oo(V_{2n}^\bullet)$ be as in \S \ref{Quadratic}.
It satisfies that $\det(\epsilon)=-1$ and $\epsilon^2=\1_{V_{2n}}$.
Note that $\sigma_2(\epsilon)=-\sigma_1(\epsilon)$. 
Fix a nonzero homomorphism $f \in \Hom_{\SO(V_{2n}^\bullet)}(\sigma_0 \otimes \tau_0, \C)$, and put
\[
f_i=f \circ (\sigma_i(\epsilon) \otimes \tau^b(\epsilon)).
\]
Then $f_i \in \Hom_{\SO(V_{2n}^\bullet)}(\sigma_0 \otimes \tau_0, \C)$,
and we have $(f_i)_i=f$ and $(f_1)_2=-f$.
Since $\dim\Hom_{\SO(V_{2n}^\bullet)}(\sigma_0 \otimes \tau_0, \C)=1$ by \cite{AGRS2}, 
there exists $c_i \in \{\pm1\}$ such that
$f_i=c_i \cdot f$ and $c_1 \not= c_2$.
If $c_i=+1$, then $f \in \Hom_{\Oo(V_{2n}^\bullet)}(\sigma_i \otimes \tau^b, \C)$.
In this case, we have $\sigma^b = \sigma_i$.
Also, if we replace $b$ with $-b$, then $c_i$ must be replaced by $-c_i$.
Hence the second assertion holds.
\end{proof}

\begin{proof}[Proof of Theorem \ref{main2}]
First, we consider the case when $m=2n$ is even.
Let $\phi \in \Phi_\temp(\Oo(V_{2n}))$ and $(\phi',b) \in \Phi_\temp(\Oo(V_{2n+1}))$.
Theorem \ref{GPSO} and Lemma \ref{lemma1} imply that
there exists a unique $(\sigma, \tau) \in \Pi_\phi \times \Pi_{\phi',b}$
such that $\sigma \boxtimes \tau$ is a representation of 
$\Oo(V_{2n}^\bullet) \times \Oo(V_{2n+1}^\bullet)$ with
a relevant pair $(V_{2n}^\bullet, V_{2n+1}^\bullet)$ of companion spaces of $(V_{2n}, V_{2n+1})$, 
and 
\[
\Hom_{\Delta\Oo(V_{2n}^\bullet)}(\sigma \boxtimes \tau, \C) \not=0.
\]
Moreover, we have $\iota(\tau) = \chi_\phi$.
\par

We show that $\iota_c(\sigma) = \chi_{\phi'} \cdot d_{\phi,b}$.
If $\phi \not\in \Phi^\epsilon(\Oo(V_{2n}))$, then
$A_\phi = A_\phi^+$ and so that $d_{\phi,b} = \1$.
Hence the desired equation follows from Theorem \ref{GPSO}.
\par

Now we assume that $\phi \in \Phi_\temp^\epsilon(\Oo(V_{2n}))$.
Then $\sigma \not\cong \sigma \otimes \det$.
Since $d_{\phi, -1}$ is the non-trivial character of $A_\phi$ which is trivial on $A_\phi^+$, 
by Lemma \ref{lemma1} (2), it suffices to show that
if 
\[
b = \ep(\half{1},\tau, \psi) \cdot e(V_{2n+1}^\bullet)
\]
and $(\sigma, \tau) \in \Pi_\phi \times \Pi_{\phi',b}$ satisfies that
$\Hom_{\Delta\Oo(V_{2n}^\bullet)}(\sigma \boxtimes \tau, \C) \not=0$, 
then $\iota_c(\sigma) = \chi_{\phi'} \cdot d_{\phi,b}$.
Here, $\ep(s,\tau, \psi)$ is
the standard $\ep$-factor defined by the doubling method (see \cite[\S 10]{LR}) and
\[
e(V_{2n+1}^\bullet) = \left\{
\begin{aligned}
&1	\iif \text{$\Oo(V_{2n+1}^\bullet)$ is split},\\
&-1	\other.
\end{aligned}
\right.
\]
Note that 
\[
e(V_{2n+1}^\bullet) = \iota(\tau)(z_{\phi'}) = \chi_{\phi}(z_{\phi'}) 
= \ep(\phi \otimes \phi') \cdot \chi_{V_{2n}}(-1)^n
\]
by Theorem \ref{GPSO}.
\par

To obtain the desired formula, we shall use the theta correspondence as in \S \ref{theta}.
Let $W_{2n}$ be a symplectic space of dimension $2n$.
We consider the theta correspondence for $(\Oo(V_{2n}^\bullet), \Sp(W_{2n}))$ 
and $(\Oo(V_{2n+1}^\bullet), \Mp(W_{2n}))$.
Here, $\Mp(W_{2n})$ is the metaplectic group associated to $W_{2n}$, i.e., 
the unique topological double cover of $\Sp(W_{2n})$. 
Since $\omega_\tau(-1) = b = \ep(1/2,\tau, \psi)e(V_{2n+1}^\bullet)$, 
by \cite[Theorem 11.1]{GI1}, we have
\[
\Theta_{W_{2n},V_{2n+1}^\bullet, \psi}(\tau) \not=0, 
\]
so that $\rho \coloneqq \theta_{W_{2n},V_{2n+1}^\bullet, \psi}(\tau)$ is 
an irreducible genuine representation of $\Mp(W_{2n})$.
Let $L$ be the orthogonal complement of $V_{2n}^\bullet$ in $V_{2n+1}^\bullet$.
Note that $\disc(L)=-c$.
Considering the following see-saw
\[
\xymatrix{
\Oo(V_{2n+1}^\bullet) \ar@{-}[d] \ar@{-}[dr]& \Sp(W_{2n}) \times \Mp(W_{2n}) \ar@{-}[d]\\
\Oo(V_{2n}^\bullet) \times \Oo(L) \ar@{-}[ur]& \Mp(W_{2n})
}
\]
we see that
\[
\Hom_{\Mp(W_{2n})}(\Theta_{W_{2n},V_{2n}^\bullet, \psi}(\sigma) \otimes \omega_{\psi_{-c}}, \rho) \not=0.
\]
In particular, $\pi \coloneqq \Theta_{W_{2n},V_{2n}^\bullet, \psi}(\sigma)$ is nonzero.
Since $\sigma$ is tempered, by \cite[Proposition C.4]{GI1},
$\pi$ is irreducible, so that $\pi = \theta_{W_{2n},V_{2n}^\bullet, \psi}(\sigma)$.
By Prasad's conjecture (Conjecture \ref{P O}), we have
$\pi \in \Pi_{\phi_{\pi}}$ with $\phi_\pi = (\phi \oplus \1) \otimes \chi_{V_{2n}}$ and
\[
\iota'_c(\pi) |A_\phi = \iota_c(\sigma).
\]
Here, we regard $A_\phi$ as a subgroup of $A_{\phi_\pi}$ via
the canonical injection $A_\phi \hookrightarrow A_{\phi_\pi}$.
On the other hand, by the Gross--Prasad conjecture for the symplectic-metaplectic case, 
which has established by \cite[Theorem 1.3]{At} 
(using Theorems \ref{GPSO}, \ref{thmPSp} and results of \cite{GS}) 
and \cite[Proposition 18.1]{GGP}, 
we have
\[
\iota'_c(\pi)(a') = \ep(\phi_\pi^{a'}\chi_{c} \otimes \phi_{\rho^\vee}) 
\cdot \ep(\phi_\pi\chi_{c} \otimes \phi_{\rho^\vee})^{\det(a')} \cdot \det(\phi_\pi^{a'})(-1)^{n}
\]
for $a' \in A_{\phi_\pi}$.
Here, $\phi_{\rho^\vee}$ is the $L$-parameter for $\rho^\vee$. 
It is given by $\phi_{\rho^\vee}=\phi' \otimes \chi_{-1}\chi_{V_{2n+1}}$ (see \cite[\S 3.6]{At}).
Recall that $\iota'_c(\pi)$ is a priori a character of $A_{\phi_\pi}^+$, but
by using the isomorphism $A_{\phi_\pi}^+ \cong A_{\phi_\pi}/\pair{z_{\phi_\pi}}$, 
we regard $\iota'_c(\pi)$ as a character of $A_{\phi_\pi}$ which is trivial at $z_{\phi_\pi}$.
If $a \in A_\phi \subset A_{\phi_\pi}$, then we have
$\phi_\pi^{a} = \phi^a \otimes \chi_{V_{2n}}$, so that 
\[
\det(\phi_\pi^{a})(-1) = \det(\phi^a)(-1) \cdot \chi_{V_{2n}}(-1)^{\dim(\phi^a)}.
\]
Since $\det(a) = \dim(\phi^a) \pmod2$ and
$\chi_{V_{2n}}\chi_{V_{2n+1}}\chi_{c}\chi_{-1}=\1$, we have
$\ep(\phi_\pi^{a}\chi_{c} \otimes \phi_{\rho^\vee}) = \ep(\phi^a \otimes \phi')$ and
\[
\ep(\phi_\pi\chi_{c} \otimes \phi_{\rho^\vee})^{\det(a)}
= \big(\ep(\phi \otimes \phi') \cdot \ep(\phi') \big)^{\dim(\phi^a)}.
\]
Since $\chi_{\phi'}(a) = \ep(\phi^a \otimes \phi') \cdot \det(\phi^a)(-1)^n$ and
$\ep(\phi') \cdot \ep(\phi \otimes \phi') \cdot \chi_{V_{2n}}(-1)^n 
= \ep(\phi') \cdot e(V_{2n+1}^\bullet) =b$, we have
\[
\iota_c(\sigma)(a) = \iota'_c(\pi)(a) = \chi_{\phi'}(a) \cdot b^{\dim(\phi^a)}
\]
for $a \in A_\phi$.
This completes the proof of Theorem \ref{main2} when $m=2n$.
\par

% m=2n-1
Next, we consider the case when $m=2n-1$ is odd.
The proof is similar to that of \cite[Theorem 19.1]{GGP}.
Let $\phi \in \Phi_\temp(\Oo(V_{2n}))$ and $(\phi',b) \in \Phi_\temp(\Oo(V_{2n-1}))$.
Suppose that $V_{2n-1}^\bullet \subset V_{2n}^\bullet$.
Let $L^\bullet=Fe_0$ be the orthogonal complement to $V_{2n-1}^\bullet$ in $V_{2n}^\bullet$
and $e_0 \in L^\bullet$ such that $\pair{e_0,e_0}_{V_{2n}^\bullet}=2c$.
Set 
\[
V_{2n+1}^\bullet = V_{2n}^\bullet \oplus (-L^\bullet) = V_{2n}^\bullet \oplus Ff_0, 
\]
where $\pair{f_0,f_0}_{V_{2n+1}^\bullet}=-2c$.
Put $v_0 = e_0+f_0$ and $X^\bullet = Fv_0$.
Then we have
\[
V_{2n+1}^\bullet = X^\bullet \oplus V_{2n-1}^\bullet \oplus (X^\bullet)^*.
\]
Let $P=M_PU_P$ be the parabolic subgroup of $\Oo(V_{2n+1}^\bullet)$
stabilizing the line $X^\bullet$, where $M_P \cong \GL(X^\bullet) \times \Oo(V_{2n-1}^\bullet)$
is the Levi subgroup of $P$ stabilizing $(X^\bullet)^*$.
Choose a unitary character $\chi$ of $F^\times \cong \GL(X^\bullet)$ which satisfies the condition of 
\cite[Theorem 15,1]{GGP}, and such that the induced representation 
\[
\Ind_{P(X^\bullet)}^{\Oo(V_{2n+1}^\bullet)}(\chi \otimes \tau)
\]
is irreducible for any $\tau \in \Pi_{\phi',b}$.
Note that in \cite{GGP}, one consider the unnormalized induction, but in this paper, 
we consider the normalized induction.
Then by \cite[Theorem 15,1]{GGP}, we have
\[
\Hom_{\Oo(V_{2n}^\bullet)}(\Ind_{P(X^\bullet)}^{\Oo(V_{2n+1}^\bullet)}(\chi \otimes \tau) 
\otimes \sigma, \C)
\cong 
\Hom_{\Oo(V_{2n-1}^\bullet)}(\tau \otimes \sigma, \C).
\]
\par

Put $\phi'' = \chi \oplus \phi' \oplus \chi^{-1}$.
Then 
\[
\Pi_{\phi'', \chi(-1)b} = \{\Ind_{P(X^\bullet)}^{\Oo(V_{2n+1}^\bullet)}(\chi \otimes \tau) \ |\ 
\tau \in \Pi_{\phi'}
 \}
\]
and
\[
\iota(\Ind_{P(X^\bullet)}^{\Oo(V_{2n+1}^\bullet)}(\chi \otimes \tau)) = \iota(\tau)
\]
as a character of $A_{\phi''} = A_{\phi'}$.
Applying the even case above to $\phi \in \Phi_\temp(\Oo(V_{2n}))$ and 
$(\phi'', \chi(-1)b) \in \Phi_\temp(\Oo(V_{2n+1}))$, 
we see that
there exists a unique pair $(\sigma, \tau) \in \Pi_\phi \times \Pi_{\phi',b}$ such that
$\sigma \boxtimes \tau$ is a representation of $\Oo(V_{2n}^\bullet) \times \Oo(V_{2n-1}^\bullet)$
with a relevant pair $(V_{2n}^\bullet, V_{2n-1}^\bullet)$ 
of companion spaces of $(V_{2n}, V_{2n-1})$, and
\[
\Hom_{\Delta\Oo(V_{2n-1}^\bullet)}(\sigma \boxtimes \tau, \C) \not=0.
\]
Moreover, we have
\begin{align*}
\iota_c(\sigma)(a) &= \chi_{\phi''}(a) \cdot (\chi(-1) b)^{\dim(\phi^a)},\\
\iota(\tau) &= \iota(\Ind_{P(X^\bullet)}^{\Oo(V_{2n+1}^\bullet)}(\chi \otimes \tau)) = \chi_\phi.
\end{align*}
By the definition, we have
\[
\frac{\chi_{\phi''}(a)}{\chi_{\phi'}(a)} = 
\ep(\phi^a \otimes (\chi \oplus \chi^{-1})) \cdot \det(\phi^a)(-1)
= \det(\phi^a \otimes \chi)(-1) \cdot \det(\phi^a)(-1)
= \chi(-1)^{\dim(\phi^a)}.
\]
Hence we have 
\[
\iota_c(\sigma)(a) = \chi_{\phi'}(a) \cdot b^{\dim(\phi^a)}, 
\]
as desired.
This completes the proof of Theorem \ref{main2} when $m=2n-1$.
\end{proof}

By Theorem \ref{main2}, 
we have established the Gross--Prasad conjecture (Conjecture \ref{GPO})
under Prasad's conjecture (Conjecture \ref{P O}).
As in \cite{GGP}, one may consider the general codimension case, 
and may prove this for tempered $L$-parameters similarly.
Also, one may consider the generic case, i.e., the $L$-parameters $\phi$ and $\phi'$ are generic.
It would follow from a similar argument to \cite{MW}.

%\section{Arthur's multiplicity formula}
%\section{Arthur's multiplicity formula}
\section{Arthur's multiplicity formula for $\SO(V_{2n})$}
The final main theorem is the so-called Arthur's multiplicity formula, 
which describes a spectral decomposition of the discrete automorphic spectrum for $\Oo(V_{2n})$.
In this section, we recall the local and global $A$-parameters, and 
Arthur's multiplicity formula for $\SO(V_{2n})$.
Then we will establish an analogous formula for $\Oo(V_{2n})$ in the next section.
\par

%\subsection{Notation and measures}\label{notation}
\subsection{Notation and measures}\label{notation}
Let $\F$ be a number field, $\A$ be the ring of adeles of $\F$.
We denote by $\A_{\fin} = {\prod}'_{v<\infty}\F_v$ and $\F_\infty = \prod_{v\mid \infty}\F_v$
the ring of finite adeles and infinite adeles, respectively.
As in the precious sections, we write $V_{2n}$ for 
an orthogonal space associated to $(d,c)$ for some $c,d \in \F^\times$.
Let $\Oo(V_{2n})$ (\resp $\SO(V_{2n})$) 
be a quasi-split orthogonal (\resp special orthogonal) group over $\F$.
We denote by $\chi_V = \otimes_v \chi_{V,v} \colon \A^\times/\F^\times \rightarrow \C^\times$ 
the discriminant character.
\par

For each $v$, we fix a maximal compact subgroup $K_v$ of $\Oo(V_{2n})(\F_v)$ such that
$K_v$ is special if $v$ is non-archimedean.
Moreover, if $\Oo(V_{2n})(\F_v)$ is unramified, we choose $K_v$ as in \S \ref{sec.unram}, 
which is hyperspecial.
Also, we take $\epsilon_v \in K_v$ such that $\det(\epsilon_v)=-1$, $\epsilon_v^{2}=\1_{V_{2n}}$
and that $\epsilon = (\epsilon_v)_v \in \Oo(V_{2n})(\A)$ is in $\Oo(V_{2n})(\F)$. 
Put $K_{0,v} = K_v \cap \SO(V_{2n})(\F_v)$.
Note that $\epsilon_v^{-1} K_{0,v} \epsilon_v = K_{0,v}$.
\par

Let $\mu_2=\{\pm1\}$ be the group of order $2$.
We regard $\mu_2$ as an algebraic group over $\F$.
There exists an exact sequence of algebraic group over $\F$:
\[
\begin{CD}
1 @>>> \SO(V_{2n}) @>>> \Oo(V_{2n}) @> \det >> \mu_2 @>>> 1.
\end{CD}
\]
For $t=(t_v)_v \in \mu_2(\A)$, we define $\epsilon_t =(\epsilon_{t,v})_v \in \Oo(V_{2n})(\A)$ by
\[
\epsilon_{t,v} = \left\{
\begin{aligned}
&\1_{V_{2n}} 	\iif t_v=1,\\
&\epsilon_v	\iif t_v=-1.
\end{aligned}
\right.
\]
\par

We take the Haar measures $dg_v$, $dh_v$, and $dt_v$ on 
$\Oo(V_{2n})(\F_v)$, $\SO(V_{2n})(\F_v)$ and $\mu_2(\F_v)$, respectively, so that
\[
\vol(K_v, dg_v) = \vol(K_{0,v}, dh_v) = \vol(\mu_2(\F_v), dt_v) = 1.
\]
Then they induce the Haar measures $dg$, $dh$ and $dt$ on 
$\Oo(V_{2n})(\A)$, $\SO(V_{2n})(\A)$ and $\mu_2(\A)$, respectively, 
satisfying that
\[
\int_{\Oo(V_{2n})(\F) \bs \Oo(V_{2n})(\A)}f(g)dg
=\int_{\mu_2(\F)\bs \mu_2(\A)}\left(
\int_{\SO(V_{2n})(\F) \bs \SO(V_{2n})(\A)}f(h \epsilon_t)dh
\right)dt
\]
for any smooth function $f$ on $\Oo(V_{2n})(\F) \bs \Oo(V_{2n})(\A)$.
\par

%\subsection{Local $A$-parameters}
\subsection{Local $A$-parameters}\label{localA}
In this subsection, we fix a place $v$ of $\F$, 
and introduce local $A$-parameters for $\Oo(V_{2n})(\F_v)$ and $\SO(V_{2n})(\F_v)$.
\par

Let $W_{\F_v}$ be the Weil group and $\WD_{\F_v}$ be the Weil--Deligne group of $\F_v$, 
i.e.,
\[
\WD_{\F_v} = \left\{
\begin{aligned}
&W_{\F_v} \times \SL_2(\C) \iif \text{$v$ is non-archimedean},\\
&W_{\F_v}	\iif \text{$v$ is archimedean}.
\end{aligned}
\right.
\]
\par

A local $A$-parameter for $\SO(V_{2n})(\F_v)$
is an admissible homomorphism
\[
\bpsi \colon \WD_{\F_v} \times \SL_2(\C) \rightarrow 
{}^L(\SO(V_{2n})) = \SO(2n,\C) \rtimes W_{\F_v}
\]
such that $\bpsi(W_{\F_v})$ projects a relatively compact subset of $\SO(2n,\C)$.
We put
\[
\Psi(\SO(V_{2n})(\F_v)) = 
\{\text{$\SO(2n,\C)$-conjugacy classes of local $A$-parameters of $\SO(V_{2n})(\F_v)$}\}.
\]	
In \S \ref{Lpara}, we have defined a map ${}^L(\SO(V_{2n})) \rightarrow \Oo(2n,\C)$.
By composing with this map, $\bpsi \in \Psi(\SO(V_{2n})(\F_v))$ gives a representation
\[
\psi \colon \WD_{\F_v} \times \SL_2(\C) \rightarrow \Oo(2n,\C).
\]
We may regard $\psi$ as an orthogonal representation of $\WD_{\F_v} \times \SL_2(\C)$.
The map $\bpsi \mapsto \psi$ gives an identification
\[
\Psi(\SO(V_{2n})(\F_v)) = 
\{
\psi \colon \WD_{\F_v} \times \SL_2(\C) \rightarrow \Oo(2n,\C)\ |\ 
\det(\psi) = \chi_{V,v}
\}/(\text{$\SO(2n,\C)$-conjugacy}).
\]
Namely, we regard $\Psi(\SO(V_{2n})(\F_v))$ as the set of 
$\SO(M)$-conjugacy classes of orthogonal representations $(\psi,M)$ 
of $\WD_{\F_v} \times \SL_2(\C)$ with $\dim(M) = 2n$ and $\det(\psi) = \chi_{V,v}$.
We say that $\psi \in \Psi(\SO(V_{2n})(\F_v))$ is tempered if
$\psi|\SL_2(\C)=\1$, i.e., $\psi$ is a tempered representation of $\WD_{\F_v}$.
We denote by $\Psi_\temp(\SO(V_{2n})(\F_v))$ the subset of $\Psi(\SO(V_{2n})(\F_v))$
consisting of the classes of tempered representation.
Also we put $\Psi(\SO(V_{2n})(\F_v))/\sim_\epsilon$ to be
the set of equivalence classes of orthogonal representations $(\psi,M)$ 
of $\WD_{\F_v} \times \SL_2(\C)$ with $\dim(M) = 2n$ and $\det(\psi) = \chi_{V,v}$.
Then there exists a canonical surjection
\[
\Psi(\SO(V_{2n})(\F_v)) \twoheadrightarrow \Psi(\SO(V_{2n})(\F_v))/\sim_\epsilon
\]
such that the order of each fiber is one or two.
We also denote by $\Psi_\temp(\SO(V_{2n})(\F_v))/\sim_\epsilon$
the image of $\Psi_\temp(\SO(V_{2n})(\F_v))$.
\par

On the other hand, we put
\[
\Psi(\Oo(V_{2n})(\F_v)) = 
\{\text{$\Oo(2n,\C)$-conjugacy classes of local $A$-parameters of $\SO(V_{2n})(\F_v)$}\}.
\]
We call an element in $\Psi(\Oo(V_{2n})(\F_v))$ an $A$-parameter of $\Oo(V_{2n})(\F_v)$.
Then we have a canonical identification 
$\Psi(\Oo(V_{2n})(\F_v)) = \Psi(\SO(V_{2n})(\F_v))/\sim_\epsilon$.
Under this identification, we put 
$\Psi_\temp(\Oo(V_{2n})(\F_v)) = \Psi_\temp(\SO(V_{2n})(\F_v))/\sim_\epsilon$.
\par

Let $\psi \in \Psi(\Oo(V_{2n})(\F_v)) = \Psi(\SO(V_{2n})(\F_v))/\sim_\epsilon$ 
be a local $A$-parameter.
We put
\[
\Sc_{\psi} = \pi_0(\cent(\im(\psi),\Oo_{2n}(\C))/\{\pm\1_{2n}\})
\quad\text{and}\quad
\Sc_{\psi}^+ = \pi_0(\cent(\im(\psi),\SO_{2n}(\C))/\{\pm\1_{2n}\}).
\]

%\subsection{Local $A$-packets}
\subsection{Local $A$-packets}
We denote by $\Irr_\unit(\Oo(V_{2n}))(\F_v)$ be the set of equivalence classes of
irreducible unitary representations of $\Oo(V_{2n})(\F_v)$.
By Theorems 2.2.1 and 2.2.4 in \cite{Ar},
there exist a finite set $\Pi_{\psi}$ with maps 
\[
\Pi_{\psi} \rightarrow \Irr_\unit(\Oo(V_{2n})(\F_v))
\]
and
\[
\iota_c \colon \Pi_{\psi} \rightarrow \widehat{\Sc_\psi}, 
\]
which satisfy certain character identities.
Using the multiplicity function
\[
m_1 \colon \Irr_\unit(\Oo(V_{2n})(\F_v)) \rightarrow \Z_{\geq0}
\]
which gives the order of the fibers in $\Pi_{\psi}$, 
we may regard $\Pi_{\psi}$ as a multiset on $\Irr_\unit(\Oo(V_{2n})(\F_v))$.
We call $\Pi_\psi$ the local $A$-packets for $\Oo(V_{2n})(\F_v)$ associated to $\psi$.
Note that $m_1(\sigma) = m_1(\sigma \otimes \det)$ by \cite[Theorem 2.2.4]{Ar}.
If $\psi$ is tempered, $\Pi_{\psi}$ coincides with the $L$-packet described in \S \ref{LLC O}.
In particular, if $\psi$ is tempered, then $\Pi_\psi$ is multiplicity-free.
\par

We denote by $\Pi_\psi^0$ the image of $\Pi_\psi$
under the canonical map 
\[
\Irr_\unit(\Oo(V_{2n})(\F_v)) \rightarrow \Irr_\unit(\Oo(V_{2n})(\F_v))/\sim_{\det}
\rightarrow \Irr_\unit(\SO(V_{2n})(\F_v))/\sim_{\epsilon}.
\]
Namely, $\Pi_\psi^0$ is a multiset on $\Irr_\unit(\SO(V_{2n})(\F_v))/\sim_{\epsilon}$ 
with the multiplicity function
\[
m_0 \colon \Irr_\unit(\SO(V_{2n})(\F_v))/\sim_{\epsilon} \rightarrow \Z_{\geq0}
\]
such that 
\[
m_0([\sigma]) = m_1(\sigma) = m_1(\sigma \otimes \det),
\]
where $[\sigma] \in \Irr_\unit(\SO(V_{2n})(\F_v))/\sim_{\epsilon}$ 
is the image of $\sigma \in \Irr_\unit(\Oo(V_{2n})(\F_v))$.
We call $\Pi_\psi^0$ the local $A$-packets for $\SO(V_{2n})(\F_v)$ associated to $\psi$.
Moreover there exist a map
\[
\iota_c \colon \Pi_\psi^0 \rightarrow \widehat{\Sc_\psi^+}
\]
which satisfies certain character identities and such that the diagram
\[
\begin{CD}
\Pi_\psi @>\iota_c>> \widehat{\Sc_\psi}\\
@VVV@VVV\\
\Pi_\psi^0 @>\iota_c>> \widehat{\Sc_\psi^+}
\end{CD}
\]
is commutative.
\par

Recall that the local $A$-packets $\Pi_\psi$ and $\Pi_\psi^0$ are multisets.
However, it is expected that the $A$-packets are sets.
\begin{conj}\label{mult}
Let $\psi \in \Psi(\Oo(V_{2n})(\F_v)) = \Psi(\SO(V_{2n})(\F_v))/\sim_\epsilon$ 
be a local $A$-parameter, and
$\Pi_\psi$ and $\Pi_\psi^0$ be the local $A$-packets.
Then for $\sigma \in \Irr_\unit(\Oo(V_{2n}))$ and 
$[\sigma] \in \Irr_\unit(\SO(V_{2n})(\F_v))/\sim_{\epsilon}$, we have
\[
m_1(\sigma) \leq 1
\quad\text{and}\quad
m_0([\sigma]) \leq 1.
\]
In other words, $\Pi_\psi$ and $\Pi_\psi^0$ are subsets of 
$\Irr_\unit(\Oo(V_{2n}))$ and $\Irr_\unit(\SO(V_{2n})(\F_v))/\sim_{\epsilon}$, respectively.
\end{conj}

\begin{prop}\label{true}
Conjecture \ref{mult} for $\psi$ holds for the following cases:
\begin{itemize}
\item
when $\psi = \phi$ is tempered $A$-parameter;
\item
when $\F_v$ is non-archimedean.
\end{itemize}
\end{prop}
\begin{proof}
For a tempered $A$-parameter $\psi = \phi$, 
the local $A$-packet $\Pi_\psi$ coincides with
the local $L$-packet $\Pi_\phi$, which are (multiplicity-free) sets.
When $\F_v$ is non-archimedean, Conjecture \ref{mult} is proven by M{\oe}glin \cite{M2}.
See also Xu's paper \cite[Theorem 8.10]{X}.
\end{proof}

\begin{rem}
By Proposition \ref{true}, 
only when $\F_v$ is archimedean and $\psi$ is non-tempered, 
Conjecture \ref{mult} for $\psi$ is not verified.
However, it is known by \cite{AMR} that
a part of local $A$-packets (in the archimedean case) 
coincides with Adams--Johnson $A$-packets, 
which are (multiplicity-free) sets.
\end{rem}

\par

Let $\psi \in \Psi(\Oo(V_{2n})(\F_v)) = \Psi(\SO(V_{2n})(\F_v))/\sim_\epsilon$.
Then $\psi$ gives a local $L$-parameter
$\phi_{\psi}$ defined by
\[
\phi_{\psi}(w) = \psi(w,
\begin{pmatrix}
|w|_{\F_v}^{\half{1}}&0\\0&|w|_{\F_v}^{-\half{1}}
\end{pmatrix}
)
\]
for $w \in \WD_{\F_v}$.
Here, $|w|_{\F_v}$ is the extension to $\WD_{\F_v}$ of the absolute value on $W_{\F_v}$, 
which is trivial on $\SL_2(\C)$.
We put
\[
\Sc_{\phi_\psi} = \pi_0(\cent(\im(\phi_\psi),\Oo_{2n}(\C))/\{\pm\1_{2n}\})
\quad\text{and}\quad
\Sc_{\phi_\psi}^+ = \pi_0(\cent(\im(\phi_\psi),\SO_{2n}(\C))/\{\pm\1_{2n}\}).
\]
Then we have $L$-packets 
\[
\Pi_{\phi_\psi} \subset \Irr_\unit(\Oo(V_{2n})(\F_v))
\quad\text{and}\quad 
\Pi_{\phi_\psi}^0 \subset \Irr_\unit(\SO(V_{2n})(\F_v))/\sim_\epsilon
\]
and bijections 
\[
\iota_c \colon \Pi_{\phi_\psi} \rightarrow \widehat{\Sc_{\phi_\psi}}
\quad\text{and}\quad
\iota_c \colon \Pi_{\phi_\psi}^0 \rightarrow \widehat{\Sc_{\phi_\psi}^+}.
\]
Moreover we have a canonical surjection 
\[
\Sc_{\psi} \twoheadrightarrow \Sc_{\phi_{\psi}}
\quad\text{and}\quad
\Sc_{\psi}^+ \twoheadrightarrow \Sc_{\phi_{\psi}}^+.
\]
Proposition 7.4.1 in \cite{Ar} says that $\Pi_{\phi_{\psi}}^0$
is contained in $\Pi_{\psi}^0$ and the diagram
\[
\begin{CD}
\Pi_{\phi_{\psi}}^0 @>>> \Pi_{\psi}^0\\
@V\iota_cVV @VV\iota_cV\\
\widehat{\Sc_{\phi_{\psi}}^+} @>>> \widehat{\Sc_{\psi}^+}
\end{CD}
\]
is commutative.
We prove an analogue of this statement.

\begin{prop}\label{LvsA}
Assume Conjecture \ref{mult}.
Let $\psi \in \Psi(\Oo(V_{2n})(\F_v))$. 
We denote by $\phi_\psi$ the $L$-parameter given by $\psi$.
Let $\Pi_\psi$ be the local $A$-packet of $\psi$, which is a multiset of $\Irr_\unit(\Oo(V_{2n})(\F_v))$, 
and $\Pi_{\phi_\psi}$ be the local $L$-packet of $\phi_\psi$, 
which is a subset of $\Irr_\unit(\Oo(V_{2n})(\F_v))$.
Then $\Pi_{\phi_{\psi}}$ is contained in $\Pi_{\psi}$ and the diagram
\[
\begin{CD}
\Pi_{\phi_{\psi}} @>>> \Pi_{\psi}\\
@V\iota_cVV @VV\iota_cV\\
\widehat{\Sc_{\phi_{\psi}}} @>>> \widehat{\Sc_{\psi}}
\end{CD}
\]
is commutative.
\end{prop}
\begin{proof}
The first assertion follows from the result for $\Pi_{\phi_{\psi}}^0$ and $\Pi_{\psi}^0$.
We write 
$\iota^L_c = \iota_c \colon \Pi_{\phi_{\psi}} \rightarrow \widehat{\Sc_{\phi_{\psi}}}$
for the left arrow in the diagram, 
and
$\iota^A_c = \iota_c \colon \Pi_{\psi} \rightarrow \widehat{\Sc_{\psi}}$
for the right arrow in the diagram.
Let $a \in \Sc_{\psi}$ and 
$s \in \Oo_{2n}(\C)$ be a (suitable) semi-simple representative of $a$.
We denote by $\psi^a$ the $(-1)$-eigenspace of $s$ in $\psi$, 
which is a representation of $\WD_{\F_v} \times \SL_2(\C)$.
For the last assertion, we may assume that $d=\dim(\psi^a)$ is odd.
Put $\psi' = \psi \oplus \psi^a \oplus (\psi^a)^\vee$.
Then we have 
\[
\phi_{\psi'} = \phi_{\psi} \oplus \phi_{\psi}^a \oplus (\phi_{\psi}^a)^\vee.
\]
Take $\sigma \in \Pi_{\phi_{\psi}}$.
Let $\tau$ be an irreducible representation of $\GL_d(\F_v)$ corresponding to $\phi_{\psi}^a$.
We denote the normalized intertwining operator of 
$\Ind_P^{\Oo(V_{2n+2d})(\F_v)}(\tau \boxtimes \sigma)$ by 
$R_c(w, \tau \boxtimes \sigma)$, 
where $P=MN$ be a suitable parabolic subgroup of $\Oo(V_{2n+2d})(\F_v)$
with the Levi factor $M \cong \GL_d(\F_v) \times \Oo(V_{2n})(\F_v)$, 
and $w$ is a suitable representative of an element in the relative Weyl group $W(M)$ of $M$.
Let $\sigma'$ be an irreducible constituent of $\Ind_P^{\Oo(V_{2n+2d})(\F_v)}(\tau \boxtimes \sigma)$.
Then $\sigma' \in \Pi_{\phi_{\psi'}} \subset \Pi_{\psi'}$.
Regarding $\sigma'$ as an element in $\Pi_{\phi_{\psi'}}$, 
by Theorems 2.2.4 and 2.4.4 in \cite{Ar} together with Conjecture \ref{mult}, 
$R_c(w, \sigma \boxtimes \tau)$ induces a scalar operator on $\sigma'$ with
eigenvalue $\iota^L_c(\sigma')(a)$.
On the other hand, regarding $\sigma'$ as an element in $\Pi_{\psi'}$, 
by the same theorems and conjecture, 
$R_c(w, \sigma \boxtimes \tau)$ induces a scalar operator on $\sigma'$ with
eigenvalue $\iota^A_c(\sigma')(a)$.
Hence 
\[
\iota^L_c(\sigma')(a)=\iota^A_c(\sigma')(a).
\]
Since $\iota^L_c(\sigma')(a) = \iota^L_c(\sigma)(a)$ and
$\iota^A_c(\sigma')(a) = \iota^A_c(\sigma)(a)$, we have
\[
\iota^L_c(\sigma)(a)=\iota^A_c(\sigma)(a),
\]
as desired.
\end{proof}
\par

We remark on unramified representations.
Suppose that $\Oo(V_{2n})(\F_v)$ is unramified, which is equivalent that 
$v$ is non-archimedean and
$c, d \in \oo_v^\times$, where $\oo_v$ is the ring of integers of $\F_v$.
We say that $\psi \in \Psi(\Oo(V_{2n})(\F_v))$ is unramified if
$\psi|\WD_{\F_v}$ is trivial on $I_{\F_v} \times \SL_2(\C)$, 
where $I_{\F_v}$ is the inertia subgroup of $W_{\F_v}$.

\begin{cor}\label{Aunram}
If $\psi \in \Psi(\Oo(V_{2n})(\F_v))$ is unramified, 
then $\Pi_\psi$ has a unique unramified representation $\sigma$.
It satisfies $\iota_c(\sigma) = \1$.
\end{cor}
\begin{proof}
The uniqueness follows from \cite[p. 18 Proposition]{M} and Lemma \ref{unram}.
The unique unramified representation $\sigma \in \Pi_\psi$ belongs to the subset $\Pi_{\phi_{\psi}}$.
By Proposition \ref{LvsA} and Desideratum \ref{propO} (4), we have $\iota_c(\sigma)=1$.
\end{proof}
\par

%\subsection{Hypothetical Langlands group and its substitute}
\subsection{Hypothetical Langlands group and its substitute}
We denote the set of irreducible unitary cuspidal automorphic representations of 
$\GL_{m}(\A)$ by $\AA_\cusp(\GL_m)$.
Let $\LL_\F$ be the hypothetical Langlands group of $\F$.
It is expected that there exists a canonical bijection
\[
\AA_\cusp(\GL_m) \longleftrightarrow 
\{\text{$m$-dimensional irreducible unitary representations of $\LL_\F$}\}.
\]
We want to use $\LL_\F$ as a global analogue of Weil--Deligne group $\WD_{\F_v}$.
Namely, for a connected reductive group $G$, 
we want to define a global $A$-parameters of $G$ by
an admissible homomorphism
\[
\psi \colon \LL_\F \times \SL_2(\C) \rightarrow {}^LG.
\]
\par

In this paper, we do not assume the existence of $\LL_\F$.
So we have to modify the definition of global $A$-parameters.
For the definition of global $A$-parameters, 
we use elements in $\AA_\cusp(\GL_m)$ instead of 
$m$-dimensional irreducible unitary representations of $\LL_\F$.
For this reason, we will define not $\Psi_2(\SO(V_{2n}))$ 
but only $\Psi_2(\SO(V_{2n}))/\sim_\epsilon$
as well as $\Psi_2(\Oo(V_{2n}))$ in the next subsection.

%\subsection{Global $A$-parameters and localization}\label{globalA}
\subsection{Global $A$-parameters and localization}\label{globalA}
Let $V_{2n}$ be an orthogonal space associated to $(d,c)$ for some $c,d \in \F^\times$.
We denote by $\chi_V \colon \A^\times/\F^\times \rightarrow \C^\times$ 
the discriminant character of $V_{2n}$.
A discrete global $A$-parameter for $\SO(V_{2n})$ and $\Oo(V_{2n})$ is 
a symbol 
\[
\Sigma = \Sigma_1[d_1] \boxplus \dots \boxplus \Sigma_l[d_l],
\]
where
\begin{itemize}
\item
$1 \leq l \leq 2n$  is an integer;
\item
$\Sigma_i \in \AA_{\cusp}(\GL_{m_i})$;
\item
$d_i$ is a positive integer such that $\sum_{i=1}^l m_id_i=2n$;
\item
if $d_i$ is odd, then $L(s,\Sigma_i,\Sym^2)$ has a pole at $s=1$;
\item
if $d_i$ is even,  then $L(s,\Sigma_i,\wedge^2)$ has a pole at $s=1$;
\item
if we denote the central character of $\Sigma_i$ by $\omega_i$, then
$\omega_1^{d_1} \cdots \omega_{l}^{d_l} = \chi_V$;
\item
if $i \not=j$ and $\Sigma_i \cong \Sigma_j$, then $d_i \not= d_j$.
\end{itemize}
Two global $A$-parameters $\Sigma = \Sigma_1[d_1] \boxplus \dots \boxplus \Sigma_l[d_l]$ and
$\Sigma' = \Sigma'_1[d'_1] \boxplus \dots \boxplus \Sigma'_{l'}[d'_{l'}]$
are said equivalent if $l=l'$ and there exists a permutation $\sigma \in \mathfrak{S}_l$ such that
$d'_i = d_{\sigma(i)}$ and $\Sigma'_i \cong \Sigma_{\sigma(i)}$
for each $i$.
We denote by $\Psi_2(\Oo(V_{2n})) = \Psi_2(\SO(V_{2n}))/\sim_\epsilon$
the set of equivalence classes of discrete global $A$-parameters 
for $\SO(V_{2n})$ and $\Oo(V_{2n})$.
Let $\Psi_2^\epsilon(\Oo(V_{2n}))$ be the subset of $\Psi_2(\Oo(V_{2n}))$ consisting of
$\Sigma=\boxplus_{i=1}^l\Sigma_i[d_i]$ as above such that $m_id_i$ is odd for some $i$.
Also, we put
\[
\Psi_{2,\temp}(\Oo(V_{2n}))=
\{\Sigma = \Sigma_1[d_1] \boxplus \dots \boxplus \Sigma_l[d_l] \in \Psi_{2}(\Oo(V_{2n}))\ |\ 
d_i=1 \text{ for any $i$}\},
\]
and $\Psi_{2,\temp}^\epsilon(\Oo(V_{2n}))
=\Psi_{2,\temp}(\Oo(V_{2n})) \cap \Psi_2^\epsilon(\Oo(V_{2n}))$.
We define $\Psi_2^\epsilon(\SO(V_{2n}))/\sim_\epsilon$ and 
$\Psi_{2,\temp}(\SO(V_{2n}))/\sim_\epsilon$ similarly.
\par

Let $\Sigma = \Sigma_1[d_1] \boxplus \dots \boxplus \Sigma_l[d_l] \in \Psi_2(\Oo(V_{2n}))$.
For each place of $v$ of $\F$, 
we denote the $m_i$-dimensional representation of $\WD_{\F_v}$ corresponding to 
$\Sigma_{i,v} \in \Irr(\GL_{m_i}(\F_v))$ by $\phi_{i,v}$.
Because of the lack of the generalized Ramanujan conjecture, 
$\phi_{i,v}$ is not necessarily a tempered representation. 
We define a representation 
$\Sigma_v \colon \WD_{\F_v} \times \SL_2(\C) \rightarrow \GL_{2n}(\C)$ by
\[
\Sigma_v = (\phi_{1,v} \boxtimes S_{d_1}) \oplus \dots \oplus (\phi_{l,v} \boxtimes S_{d_l}), 
\]
where $S_{d}$ is the unique irreducible algebraic representation of $\SL_2(\C)$ of dimension $d$.
We call $\Sigma_v$ the localization of $\Sigma$ at $v$. 
By \cite[Proposition 1.4.2]{Ar}, the representation $\Sigma_v$ 
factors through $\Oo_{2n}(\C) \hookrightarrow \GL_{2n}(\C)$.
In particular, if $\Sigma \in \Psi_{2,\temp}(\Oo(V_{2n}))$, 
then $\Sigma_v \in \Phi(\Oo(V_{2n})(\F_v))$.
\par

Let $\Sigma = \Sigma_1[d_1] \boxplus \dots \boxplus \Sigma_l[d_l] \in \Psi_2(\Oo(V_{2n}))$ 
be a global $A$-parameter with $\Sigma_i \in \AA_\cusp(\GL_{m_i})$, 
and $\Sigma_v = (\phi_{1,v} \boxtimes S_{d_1}) \oplus \dots \oplus (\phi_{l,v} \boxtimes S_{d_l})$ 
be the localization at $v$.
We also write $\Sigma_{i,v} = \phi_{i,v} \otimes S_{d_i}$.
We put
\[
A_\Sigma = \bigoplus_{i=1}^l (\Z/2\Z)a_{\Sigma_{i}[d_i]} \cong (\Z/2\Z)^l.
\]
Namely, $A_\Sigma$ is a free $(\Z/2\Z)$-module of rank $l$ and 
$\{a_{\Sigma_{1}[d_1]}, \dots, a_{\Sigma_{l}[d_l]}\}$
is a basis of $A_\Sigma$ with $a_{\Sigma_{i}[d_i]}$ associated to $\Sigma_{i}[d_i]$.
We define $A_\Sigma^+$ by the kernel of the map 
$A_\Sigma \ni a_{\Sigma_i[d_i]} \mapsto (-1)^{m_id_i} \in \{\pm1\}$.
Also, we put
\[
z_\Sigma = \sum_{i=1}^l m_id_i \cdot a_{\Sigma_i[d_i]} \in A_\Sigma^+ \subset A_\Sigma.
\]
We define the global component groups $\Sc_\Sigma$ and $\Sc_\Sigma^+$ by
\[
\Sc_\Sigma = A_\Sigma/\pair{z_\phi}
\quad\text{and}\quad
\Sc_\Sigma^+ = A_\Sigma^+/\pair{z_\phi}.
\]
On the other hand, as in \S \ref{localA}, we put
\[
\Sc_{\Sigma_v} = \pi_0(\cent(\im(\Sigma_v),\Oo_{2n}(\C))/\{\pm\1_{2n}\})
\quad\text{and}\quad
\Sc_{\Sigma_v}^+ = \pi_0(\cent(\im(\Sigma_v),\SO_{2n}(\C))/\{\pm\1_{2n}\}).
\]
Then we have a map 
\[
\Sc_\Sigma \rightarrow \Sc_{\Sigma_v},\ a_{\Sigma_i[d_i]} \mapsto -\1_{\Sigma_{i,v}}
\]
for each place $v$, where
$-\1_{\Sigma_{i,v}}$ is the image of the element in $\cent(\im(\Sigma_v),\Oo_{2n}(\C))$
which acts on $\Sigma_{i,v}$ by $-\id$ and acts on $\Sigma_{j,v}$ trivially for any $j \not=i$.
Hence we obtain the diagonal maps
\[
\Delta \colon \Sc_{\Sigma} \rightarrow \prod_v \Sc_{\Sigma_v}
\quad\text{and}\quad
\Delta \colon \Sc_{\Sigma}^+ \rightarrow \prod_v \Sc_{\Sigma_v}^+.
\]
\par

Let $\Sigma \in \Psi_2(\Oo(V_{2n}))$ be a global $A$-parameter
and $\psi = \Sigma_v$ be the localization at $v$.
We emphasize that $\psi$ does not necessarily belong to $\Psi(\Oo(V_{2n})(\F_v))$
defined in \S \ref{localA}.
We can decompose
\[
\psi = \psi_1|\cdot|_{\F_v}^{s_1} \oplus \dots \oplus \psi_r|\cdot|_{\F_v}^{s_r} \oplus
\psi_0 \oplus \psi_r^\vee|\cdot|_{\F_v}^{-s_r} \oplus \dots \oplus \psi_1^\vee|\cdot|_{\F_v}^{-s_1}, 
\]
where
\begin{itemize}
\item
$\psi_i$ is an irreducible representation of $\WD_{\F_v} \times \SL_2(\C)$ 
of dimension $d_i$ such that $\psi_i(W_{\F_v})$ is bounded;
\item
$\psi_0 \in \Psi(\Oo(V_{2n_0})(\F_v))$;
\item
$d_1+\dots+d_r+n_0=n$ and $s_1 \geq \dots \geq s_r > 0$.
\end{itemize}
We define a representation $\phi_{\psi_i}$ of $\WD_{\F_v}$ by
\[
\phi_{\psi_i}(w) = \psi_i(w, 
\begin{pmatrix}
|w|_{\F_v}^{\half{1}}&0\\0&|w|_{\F_v}^{-\half{1}}
\end{pmatrix}
),
\]
and we denote by $\tau_{\psi_i}$ the irreducible representation of $\GL_{d_i}(\F_v)$
corresponding to $\phi_{\psi_i}$.
Let $\Pi_{\psi_0}$ be the local $A$-packet associated to $\psi_0$, 
which is a multiset of $\Irr_\unit(\Oo(V_{2n_0})(\F_v))$.
For $\sigma_0 \in \Pi_{\psi_0}$, we put
\[
I(\sigma_0) = 
\Ind_{P(\F_v)}^{\Oo(V_{2n})(\F_v)}
(\tau_{\psi_1}|\cdot|_{\F_v}^{s_1} \otimes \dots \otimes \tau_{\psi_r}|\cdot|_{\F_v}^{s_r} \otimes \sigma_0), 
\]
where $P$ is a parabolic subgroup of $\Oo(V_{2n})$ with
Levi subgroup $M_P=\GL_{d_1} \times \dots \times \GL_{d_r} \times \Oo(V_{2n_0})$.
Note that $I(\sigma_0)$ depends on not only $\sigma_0$ but also $\psi$.
\par

To establish the global main result (Theorem \ref{main3} below), 
we need the following conjecture (see \cite[Conjecture 8.3.1]{Ar}).

\begin{conj}\label{std}
Let $\Sigma \in \Psi_2(\Oo(V_{2n}))$ be a global $A$-parameter.
We decompose $\Sigma_v$ as above, so that $\psi_0 \in \Psi(\Oo(V_{2n_0})(\F_v))$.
Then the induced representation $I(\sigma_0)$ is
irreducible for $\sigma_0 \in \Pi_{\psi_0}$.
Moreover, if $I(\sigma_0) \cong I(\sigma_0')$ for $\sigma_0, \sigma_0' \in \Pi_{\psi_0}$, 
then $\sigma_0 \cong \sigma_0'$.
\end{conj}

For tempered $A$-parameters, the irreducibility and the disjointness 
follows from a result of Heiermann \cite{H}.
\begin{prop}\label{heier}
Conjecture \ref{std} holds for the following cases:
\begin{itemize}
\item
when $\Sigma \in \Psi_{2,\temp}(\Oo(V_{2n}))$; 
\item
when $\F_v$ is non-archimedean.
\end{itemize}
\end{prop}
\begin{proof}
The second case is \cite[Proposition 5.1]{M3}.
So we shall prove the first case.
\par

Let $\Sigma$ be an element in $\Psi_{2,\temp}(\Oo(V_{2n}))$, 
$\phi = \Sigma_v$ be the localization of $\Sigma$ at $v$, 
and  $\phi_0 = \psi_0$ be as above.
Note that $\phi \in \Phi(\Oo(V_{2n})(\F_v))$.
Then the $L$-packet $\Pi_{\phi}$ is equal to the set of the Langlands quotients of 
the standard module $I(\sigma_0)$, 
where $\sigma_0$ runs over $\Pi_{\phi_0}$.
Heiermann's result \cite{H} (together with the Clifford theory) 
asserts that if $\phi \in \Phi_\gen(\Oo(V_{2n})(\F_v))$, 
then $I(\sigma_0)$ is irreducible for any $\sigma_0 \in \Pi_{\phi_0}$.
Namely, to show the irreducibility of $I(\sigma_0)$, 
it suffices to prove that $L(s, \phi, \Ad)$ is regular at $s=1$.
\par

We decompose
\[
\phi = \phi_{1}|\cdot|_{\F_v}^{s_{1}} \oplus \dots \oplus \phi_{r}|\cdot|_{\F_v}^{s_{r}}
\oplus \phi_{0} \oplus 
\phi_{r}^\vee|\cdot|_{\F_v}^{-s_{r}} \oplus \dots \oplus \phi_{1}^\vee|\cdot|_{\F_v}^{-s_{1}},
\]
where
\begin{itemize}
\item
$\phi_i$ is an irreducible tempered representation of $\WD_{\F_v}$ of dimension $d_i$;
\item
$\phi_0 \in \Phi_\temp(\Oo(V_{2n_0})(\F_v))$;
\item
$d_1+\dots+d_r+n_0=n$ and $s_1 \geq \dots \geq s_r > 0$.
\end{itemize}
Then $L(s, \phi, \Ad)$ is equal to
\begin{align*}
&L(s, \phi_0, \Ad)
\left(\prod_{i=1}^{r}L(s, \phi_i, \Ad_{\GL})
L(s+s_i, \phi_0 \otimes \phi_i) L(s-s_i, \phi_0 \otimes \phi_i^\vee)
L(s+2s_i, \phi_i, \wedge^2) L(s-2s_i, \phi_i^\vee, \wedge^2) \right)\\
&\quad\times\left(\prod_{1 \leq i<j \leq r} 
L(s+s_i+s_j, \phi_i \otimes \phi_j)
L(s+s_i-s_j, \phi_i \otimes \phi_j^\vee)
L(s-s_i+s_j, \phi_i^\vee \otimes \phi_j)
L(s-s_i-s_j, \phi_i^\vee \otimes \phi_j^\vee)
\right), 
\end{align*}
where $\Ad_{\GL}$ is the adjoint representation of $\GL_{m}(\C)$ on $\Lie(\GL_{m}(\C))$ 
for a suitable $m$.
\par

Since $\phi$ corresponds to an irreducible representation of $\GL_{2n}(\F_v)$, 
which is a local constituent of $\Sigma = \boxplus_{i=1}^l \Sigma_i$, 
we have $|s_i|<1/2$ for any $i$.
See e.g., \cite[(2.5) Corollary]{JS} and \cite[Appendix]{RS}.
Note that for tempered representations, 
all $L$-functions appeared in the above equation are
regular for $\re(s) > 0$.
Hence we conclude that $L(s, \phi, \Ad)$ is regular at $s=1$.
\par

Since $I(\sigma_0)$ is a standard module, 
the last assertion of Conjecture \ref{std} follows from the Langlands classification.
\end{proof}
\par

Go back to the general situation.
Let $\Sigma \in \Psi_2(\Oo(V_{2n}))$ be a global $A$-parameter,
and $\psi = \Sigma_v$, $\psi_0$ and $I(\sigma_0)$ for $\sigma_0 \in \Pi_{\psi_0}$ be as above.
We define the local $A$-packet $\Pi_{\psi}$ associated to $\psi = \Sigma_v$, 
which is a multiset on $\Irr(\Oo(V_{2n})(\F_v))$, by
\[
\Pi_\psi = \bigsqcup_{\sigma_0 \in \Pi_{\psi_0}}\{\sigma\ |\ 
\text{$\sigma$ is an irreducible constituent of $I(\sigma_0)$}
\}.
\]
Namely, $\Pi_\psi$ is the disjoint union of
the multisets of the Jordan--H\"{o}lder series of $I(\sigma_0)$.
Similarly, we can define the local $A$-packet $\Pi_{\psi}^0$ associated to $\psi = \Sigma_v$, 
which is a multiset on $\Irr(\SO(V_{2n})(\F_v))/\sim_\epsilon$.
Since $\Sc_\psi = \Sc_{\psi_0}$ and $\Sc_\psi^+ = \Sc_{\psi_0}^+$, 
we can define maps
\[
\iota_c \colon \Pi_\psi \rightarrow \widehat{\Sc_\psi}
\quad\text{and}\quad
\iota_c \colon \Pi_\psi^0 \rightarrow \widehat{\Sc_\psi}^+
\]
by
\[
\iota_c(\sigma) \coloneqq \iota_c(\sigma_0)
\quad\text{and}\quad
\iota_c([\sigma]) \coloneqq \iota_c([\sigma_0])
\]
if $\sigma$ is an irreducible constituent of $I(\sigma_0)$.
\par

%\subsection{Global $A$-packets}
\subsection{Global $A$-packets}
Let $\HH(\Oo(V_{2n}))=\otimes_v'\HH(\Oo(V_{2n})(\F_v))$ 
(\resp $\HH(\SO(V_{2n}))=\otimes_v'\HH(\SO(V_{2n})(\F_v))$)
be the global Hecke algebra on $\Oo(V_{2n})(\A)$ (\resp $\SO(V_{2n})(\A)$) with respect to 
the maximal compact subgroup $K=\prod_vK_v$ (\resp $K_0 = \prod_vK_{0,v}$) 
fixed in \S \ref{notation}.
Namely, $\HH(\Oo(V_{2n})(\F_v))$ (\resp $\HH(\SO(V_{2n})(\F_v))$)
is the algebra of smooth, left and right $K_v$-finite (\resp $K_{0,v}$-finite) functions 
of compact support on $\Oo(V_{2n})(\F_v)$ (\resp $\SO(V_{2n})(\F_v)$).
We denote by $\HH^\epsilon(\SO(V_{2n})(\F_v))$ the subspace of functions in 
$\HH(\SO(V_{2n})(\F_v))$ which are invariant under $\epsilon_v$.
We put $\HH^\epsilon(\SO(V_{2n})) = \otimes_v' \HH^\epsilon(\SO(V_{2n})(\F_v))$.
We say that two admissible representations of $\SO(V_{2n})(\A)$ of the form
\[
\sigma_0=\otimes'_v \sigma_{0,v}
\quad\text{and}\quad
\sigma'_0=\otimes'_v \sigma_{0,v}'
\]
are $\epsilon$-equivalent if $\sigma_{0,v} \sim_{\epsilon_v} \sigma_{0,v}'$ for each $v$.
The $\epsilon$-equivalence class of $\sigma_0$ is denoted by $[\sigma_0]=\otimes'_v[\sigma_{0,v}]$.
For $f \in \HH^\epsilon(\SO(V_{2n}))$ and any $\sigma_0$ as above, 
the operator $\sigma_0(f)$ depends only on the class $[\sigma_0]$.
\par

Let $\Sigma \in \Psi_2(\Oo(V_{2n}))$ be a global $A$-parameter.
We attach a global $A$-packets 
\[
\Pi_\Sigma = \{\sigma = \otimes'_v\sigma_{v} \ |\ 
\sigma_{v} \in \Pi_{\Sigma_v},\ 
\iota_c(\sigma_v)=\1 \text{ for almost all $v$}
\}
\]
of equivalence classes of irreducible representations of $\Oo(V_{2n})(\A)$, 
and a global $A$-packets 
\[
\Pi_\Sigma^0 = \{[\sigma_0]=[\otimes'_v\sigma_{0,v}] \ |\ 
[\sigma_{0,v}] \in \Pi_{\Sigma_v}^0,\ 
\iota_c([\sigma_{0,v}])=\1 \text{ for almost all $v$}
\}
\]
of $\epsilon$-equivalence classes of irreducible representations of $\SO(V_{2n})(\A)$.
Note that an element $\sigma \in \Pi_\Sigma$ and a representative $\sigma_0$ of 
$[\sigma_0] \in \Pi_\Sigma^0$ are not necessarily unitary.
For $\sigma \in \Pi_\Sigma$ and $[\sigma_0] \in \Pi_\Sigma^0$, 
the operators $\sigma(f)$ and $\sigma_0(f_0)$
are well-defined for $f \in \HH(\Oo(V_{2n}))$ and $f_0 \in \HH^\epsilon(\SO(V_{2n}))$, 
respectively.
\par

For $\sigma = \otimes_v'\sigma_v \in \Pi_\Sigma$ and
$[\sigma_0]=[\otimes'_v\sigma_{0,v}] \in \Pi_\Sigma^0$, 
we define characters $\iota_c(\sigma)$ of $\prod_v\Sc_{\Sigma_v}$ and
$\iota_c([\sigma_0])$ of $\prod_v\Sc_{\Sigma_v}^+$ by
\[
\iota_c(\sigma)=\prod_v \iota_c(\sigma_v)
\quad\text{and}\quad
\iota_c([\sigma_0])=\prod_v \iota_c([\sigma_{0,v}]), 
\]
respectively.

%\subsection{Arthur's multiplicity formula for $\SO(V_{2n})$}
\subsection{Arthur's multiplicity formula for $\SO(V_{2n})$}
We say that a function $\varphi \colon \Oo(V_{2n})(\A) \rightarrow \C$ is 
an automorphic form on $\Oo(V_{2n})(\A)$ if 
$\varphi$ satisfies the following conditions:
\begin{itemize}
\item
$\varphi $ is left $\Oo(V_{2n})(\F)$-invariant;
\item
$\varphi $ is smooth and moderate growth;
\item
$\varphi $ is right $K$-finite, where 
$K=\prod_vK_v$ is the maximal compact subgroup of $\Oo(V_{2n})(\A)$ fixed in \S \ref{notation};
\item
$\varphi $ is $\z$-finite, where $\z$ is the center of the universal enveloping algebra of
$\Lie(\Oo(V_{2n})(\F_\infty)) \otimes_\R \C$.
\end{itemize}
We define automorphic forms on $\SO(V_{2n})(\A)$ similarly.
More precisely, see \cite[\S 4.2]{BJ}.
Let $\AA(\Oo(V_{2n}))$ be the space of automorphic forms on $\Oo(V_{2n})(\A)$.
We denote by $\AA_2(\Oo(V_{2n}))$ the subspace of $\AA(\Oo(V_{2n}))$ consisting of 
square-integrable automorphic forms on $\Oo(V_{2n})(\A)$.
Similarly, we define $\AA(\SO(V_{2n}))$ and $\AA_2(\SO(V_{2n}))$.
We call $\AA_2(\Oo(V_{2n}))$ (\resp $\AA_2(\SO(V_{2n}))$)
the automorphic discrete spectrum of $\Oo(V_{2n})$ (\resp $\SO(V_{2n})$).
\par

The Hecke algebra $\HH(\Oo(V_{2n}))$ (\resp $\HH(\SO(V_{2n}))$) acts on 
$\AA(\Oo(V_{2n}))$ (\resp $\AA(\SO(V_{2n}))$) by 
\begin{align*}
(f \cdot \varphi)(g) &= \int_{\Oo(V_{2n})(\A)}\varphi(gx)f(x)dx\\
(\text{\resp} \quad
(f_0 \cdot \varphi_0)(g_0) &= \int_{\SO(V_{2n})(\A)}\varphi_0(g_0x_0)f_0(x_0)dx_0
)
\end{align*}
for $f \in \HH(\Oo(V_{2n}))$ and $\varphi \in \AA(\Oo(V_{2n}))$
(\resp $f_0 \in \HH(\SO(V_{2n}))$ and $\varphi_0 \in \AA(\SO(V_{2n}))$).
This action preserves $\AA_2(\Oo(V_{2n}))$ (\resp $\AA_2(\SO(V_{2n}))$).
\par

Arthur's multiplicity formula for $\SO(V_{2n})$ is formulated as follows:
%Arthur's multiplicity formula
\begin{thm}[Arthur's multiplicity formula ({\cite[Theorem 1.5.2]{Ar}})]\label{AMSO}
Let $V_{2n}$ be be the orthogonal space over a number field $\F$
associated to $(d,c)$ for some $c,d \in \F^\times$.
Then for each $\Sigma \in \Psi_2(\SO(V_{2n}))/\sim_\epsilon$, there exists a character
\[
\ep_\Sigma \colon \Sc_\Sigma \rightarrow \{\pm1\}
\]
defined explicitly in terms of symplectic $\ep$-factors such that
\[
\AA_2(\SO(V_{2n}))
\cong \bigoplus_{\Sigma \in \Psi_{2}(\SO(V_{2n}))/\sim_\epsilon} m_{\Sigma} 
\bigoplus_{[\sigma_0] \in \Pi_{\Sigma}^0(\ep_\Sigma)}[\sigma_0]
\]
as $\HH^\epsilon(\SO(V_{2n}))$-modules.
Here for $\Sigma = \boxplus_{i=1}^{l} \Sigma_i[d_i] \in \Psi_2(\SO(V_{2n}))/\sim_\epsilon$
with $\Sigma_i \in \AA_\cusp(\GL_{m_i})$, we put
\[
m_\Sigma = \left\{
\begin{aligned}
&1	\iif \Sigma \in \Psi_2^\epsilon(\SO(V_{2n}))/\sim_\epsilon, 
\text{i.e., $m_id_i$ is odd for some $i$},\\
&2	\other, 
\end{aligned}
\right.
\]
and we put
\[
\Pi_{\Sigma}^0(\ep_\Sigma)=\{
[\sigma_0] \in \Pi_\Sigma^0\ |\ \iota_c([\sigma_0]) \circ \Delta = \ep_\Sigma | \Sc_\Sigma^+
\}.
\]
Moreover, if $\Sigma \in \Psi_{2,\temp}(\SO(V_{2n}))$, then 
$\ep_\Sigma$ is the trivial representation of $\Sc_\Sigma$.
\end{thm}
For the definition of $\ep_\Sigma$, see the remark after \cite[Theorem 1.5.2]{Ar}.

\begin{rem}
In fact, Arthur described a spectral decomposition of 
$L^2_\disc(\SO(V_{2n})(\F) \bs \SO(V_{2n})(\A))$.
However it is well understood (by Harish-Chadra, Langlands etc) that
$\AA_2(\SO(V_{2n}))$ is a dense subspace of $L^2_\disc(\SO(V_{2n})(\F) \bs \SO(V_{2n})(\A))$.
So we shall work with $\AA_2(\SO(V_{2n}))$ in this paper.
\end{rem}

\begin{rem}
Ta\"{i}bi \cite{Ta} prove the multiplicity formula for $\SO(V_{2n})$ when
$\SO(V_{2n})(\F_\infty)$ is compact and 
$\SO(V_{2n})(\F_v)$ is quasi-split at all finite places $v$ of $\F$.
\end{rem}

%\subsection{Arthur's multiplicity formula for $\SO(V_{2n})(\A) \cdot \Oo(V_{2n})(\F)$}
\subsection{Arthur's multiplicity formula for $\SO(V_{2n})(\A) \cdot \Oo(V_{2n})(\F)$}
\label{secAMSOO}
Theorem \ref{AMSO} follows from a more precise result. 
In this subsection, we recall this result.
\par

Let $\epsilon \in \Oo(V_{2n})(\F)$ be as in \S \ref{notation}.
We may consider $\HH(\SO(V_{2n})) \rtimes \pair{\epsilon}$, where
$\epsilon$ acts on $\HH(\SO(V_{2n}))$ by 
\[
(\epsilon f \epsilon^{-1})(x) \coloneqq f(\epsilon^{-1} x \epsilon)
\]
for $f \in \HH(\SO(V_{2n}))$.
We say that $(\sigma, \VV_\sigma)$ is an $(\HH(\SO(V_{2n})), \epsilon)$-module if
\begin{itemize}
\item
$(\sigma, \VV_\sigma)$ is an $\HH(\SO(V_{2n}))$-module;
\item
there is an automorphism $\sigma(\epsilon)$ on $\VV_\sigma$ such that 
$\sigma(\epsilon)^2 = \1_{\VV_\sigma}$;
\item
$\sigma(\epsilon f \epsilon ^{-1}) \circ \sigma(\epsilon) = \sigma(\epsilon) \circ \sigma(f)$
for any $f \in \HH(\SO(V_{2n}))$.
\end{itemize}
\par

We define an action of $\epsilon$ on $\AA_2(\SO(V_{2n}))$ by
\[
(\epsilon \cdot \varphi)(h_0) = \varphi(\epsilon^{-1} h_0 \epsilon).
\]
It makes $\AA_2(\SO(V_{2n}))$ an $(\HH(\SO(V_{2n})), \epsilon)$-module.
\par

For an $\HH(\SO(V_{2n}))$-module $(\sigma, \VV_\sigma)$, we define 
an $\HH(\SO(V_{2n}))$-module $(\sigma^\epsilon, \VV_{\sigma^\epsilon})$ by 
$\VV_\sigma = \VV_{\sigma^\epsilon}$ and
\[
\sigma^\epsilon(f)v = \sigma(\epsilon^{-1}f\epsilon)v
\]
for $f \in \HH(\SO(V_{2n}))$ and $v \in \VV_\sigma$.

\begin{lem}
If $\VV_0 \subset \AA_2(\SO(V_{2n}))$
is an irreducible $\HH(\SO(V_{2n}))$-summand isomorphic to $\sigma$, 
then the subspace
\[
\VV_0^\epsilon = \{\epsilon \cdot \varphi \ |\ \varphi \in \VV_0\}
\]
is an irreducible $\HH(\SO(V_{2n}))$-summand isomorphic to $\sigma^\epsilon$.
\end{lem}
\begin{proof}
Obvious.
\end{proof}

Now, we recall a result of Arthur \cite{Ar} which states a decomposition of $\AA_2(\SO(V_{2n}))$
as an $(\HH(\SO(V_{2n})), \epsilon)$-module by using global $A$-parameters.
Let $\Sigma = \boxplus_{i=1}^l \Sigma_i[d_i] \in \Psi_2(\SO(V_{2n}))/\sim_\epsilon$
and $[\sigma_0] \in \Pi_\Sigma^0$. 
We take a representative $\sigma_0$ which occurs in $\AA_2(\SO(V_{2n}))$, 
and we denote by $\VV_0$ a subspace of $\AA_2(\SO(V_{2n}))$ which realizes $\sigma_0$.
We distinguish 3 cases as follows:
\begin{enumerate}
\item[(A)]
Suppose that $\Sigma \in \Psi_2^\epsilon(\SO(V_{2n}))/\sim_\epsilon$.
In this case, we have $m_\Sigma = 1$. 
Hence $\VV_0$ is stable under the action of $\epsilon$,
and so that $\sigma_0^\epsilon \cong \sigma_0$.
The space $\VV_0$ realizes a distinguished extension of $\sigma_0$ to 
an $(\HH(\SO(V_{2n})), \epsilon)$-module.
\item[(B)]
Suppose that $\Sigma \not\in \Psi_2^\epsilon(\SO(V_{2n}))/\sim_\epsilon$ and 
$\sigma_0^\epsilon \not\cong \sigma_0$. 
In this case, $m_\Sigma = 2$ and $\VV_0^\epsilon \not \cong \VV_0$.
This shows that both $\sigma_0$ and $\sigma_0^\epsilon$ occur in $\AA_2(\SO(V_{2n}))$.
This explains why $m_\Sigma = 2$.
\item[(C)]
Suppose that $\Sigma \not\in \Psi_2^\epsilon(\SO(V_{2n}))/\sim_\epsilon$ and 
$\sigma_0^\epsilon \cong \sigma_0$. 
Then there are exactly two extensions $\sigma_1$ and $\sigma_2$ of $\sigma_0$ to
$(\HH(\SO(V_{2n})), \epsilon)$-modules.
Moreover, \cite[Theorem 4.2.2 (a)]{Ar} implies that
both $\sigma_1$ and $\sigma_2$ occur in $\AA_2(\SO(V_{2n}))$.
This explains why $m_\Sigma = 2$.
\end{enumerate}
\par

In the case (A), there are exactly two extensions $\sigma_1$ and $\sigma_2$ of $\sigma_0$ to
$(\HH(\SO(V_{2n})), \epsilon)$-modules.
The above argument shows that exactly one of $\sigma_1$ or $\sigma_2$
occurs in $\AA_2(\SO(V_{2n}))$.
The following theorem determines which extension occurs.
%AMSOO
\begin{thm}[{\cite[Theorem 4.2.2]{Ar}}]\label{AMSOO}
Let $\Sigma \in \Psi_2^\epsilon(\SO(V_{2n}))/\sim_\epsilon$ and $[\sigma_0] \in \Pi_\Sigma^0$.
Assume that an $\HH(\SO(V_{2n}))$-module $\sigma_0 = \otimes'_v \sigma_{0,v}$ 
occurs in $\AA_2(\SO(V_{2n}))$, so that 
$\iota_c([\sigma_0]) \circ \Delta = \ep_\Sigma|\Sc_\Sigma^+$
by Arthur's multiplicity formula.
For each place $v$ of $\F$, take an extension $\sigma_v$ of $\sigma_{0,v}$ to 
an $(\HH(\SO(V_{2n})(\F_v)), \epsilon_v)$-module
such that $\iota_c(\sigma_v)=\1$ for almost all $v$.
Put $\sigma = \otimes'_v\sigma_v$.
Let $\iota_c(\sigma)$ be the character of $\prod_v\Sc_{\Sigma_v}$ defined by
\[
\iota_c(\sigma)=\prod_v \iota_c(\sigma_v).
\]
Then as an $(\HH(\SO(V_{2n})), \epsilon)$-module, 
$\sigma$ occurs in $\AA_2(\SO(V_{2n}))$ if and only if 
$\iota_c(\sigma) \circ \Delta = \ep_\Sigma$.
\end{thm}

%\section{Comparison $\AA_2(\Oo(V_{2n}))$ with $\AA_2(\SO(V_{2n}))$}
%\section{Comparison $\AA_2(\Oo(V_{2n}))$ with $\AA_2(\SO(V_{2n}))$}
\section{Arthur's multiplicity formula for $\Oo(V_{2n})$}
In this section, we prove Arthur's multiplicity formula for $\Oo(V_{2n})$, 
which is the third main theorem in this paper.

%\subsection{Statements}
\subsection{Statements}
The global main theorem is Arthur's multiplicity formula for $\Oo(V_{2n})$, 
which is formulated as follows:
%main3
\begin{thm}[Arthur's multiplicity formula for $\Oo(V_{2n})$]\label{main3}
Assume Conjectures \ref{mult} and \ref{std}.
Let $\Sigma \in \Psi_2(\Oo(V_{2n}))$ and 
\[
\ep_\Sigma \colon \Sc_\Sigma \rightarrow \{\pm1\}
\]
be the character of $\Sc_\Sigma$ as in Theorem \ref{AMSO}, which
is trivial if $\Sigma \in \Psi_{2,\temp}(\Oo(V_{2n}))$.
Then we have a decomposition
\[
\AA_2(\Oo(V_{2n})) \cong \bigoplus_{\Sigma \in \Psi_2(\Oo(V_{2n}))}
\bigoplus_{\sigma \in \Pi_\Sigma(\ep_\Sigma)} \sigma
\]
as $\HH(\Oo(V_{2n}))$-modules.
Here we put
\[
\Pi_\Sigma(\ep_\Sigma) = \{\sigma \in \Pi_\Sigma\ |\ \iota_c(\sigma) \circ \Delta = \ep_\Sigma\}.
\]
\end{thm}

Also, we will show the following:
\begin{prop}\label{mult1}
Assume Conjectures \ref{mult} and \ref{std}.
Then for an irreducible $\HH(\Oo(V_{2n}))$-module $\sigma$, we have
\[
\dim_\C\Hom_{\HH(\Oo(V_{2n}))}(\sigma, \AA_2(\Oo(V_{2n}))) \leq 1.
\]
In other words, 
$\AA_2(\Oo(V_{2n}))$ is multiplicity-free as an $\HH(\Oo(V_{2n}))$-module. 
\end{prop}

\begin{rem}\label{temp}
The arguments in the proofs of Theorem \ref{main3} and Proposition \ref{mult1} 
work when we restrict to $\Sigma \in \Psi_{2,\temp}(\Oo(V_{2n}))$.
In this case, Proposition \ref{heier}, 
which shows Conjecture \ref{std} for $\Sigma \in \Psi_{2,\temp}(\Oo(V_{2n}))$, 
also implies that
the local $A$-packet $\Pi_{\Sigma_v}$ associated to the localization $\Sigma_v$
becomes a local $L$-packet, which is multiplicity-free, i.e., 
which satisfies Conjecture \ref{mult}.
Hence without assuming any conjectures, 
Theorem \ref{main3} and Proposition \ref{mult1} hold for $\Sigma \in \Psi_{2,\temp}(\Oo(V_{2n}))$. 
In particular, the tempered part of 
the automorphic discrete spectrum of $\Oo(V_{2n})$
\[
\AA_{2,\temp}(\Oo(V_{2n})) \coloneqq \bigoplus_{\Sigma \in \Psi_{2,\temp}(\Oo(V_{2n}))}
\bigoplus_{\sigma \in \Pi_\Sigma(\ep_\Sigma)} \sigma
\]
is multiplicity-free as an $\HH(\Oo(V_{2n}))$-module unconditionally. 
\end{rem}

The rest of this section is devoted to the proofs of Theorem \ref{main3} and Proposition \ref{mult1}.

%\subsection{Restriction of automorphic forms}
\subsection{Restriction of automorphic forms}
Now we compare $\AA_2(\Oo(V_{2n}))$ with $\AA_2(\SO(V_{2n}))$.
To do this, we consider the restriction map
\[
\Res \colon \AA(\Oo(V_{2n})) \rightarrow \AA(\SO(V_{2n})),\ 
\varphi  \mapsto \varphi |\SO(V_{2n})(\A).
\]
\begin{lem}
We have $\Res(\AA_2(\Oo(V_{2n}))) \subset \AA_2(\SO(V_{2n}))$.
\end{lem}
\begin{proof}
Let $\varphi \in \AA_2(\Oo(V_{2n}))$.
Then 
\[
\int_{\Oo(V_{2n})(\F) \bs \Oo(V_{2n})(\A)} |\varphi(g)|^2dg
= \int_{\mu_2(\F) \bs \mu_2(\A)}\left(
\int_{\SO(V_{2n})(\F) \bs \SO(V_{2n})(\A)}|\varphi(h \epsilon_t)|^2dh
\right)dt
\]
is finite.
By Fubini's theorem, we see that for almost everywhere $t \in \mu_2(\F)\bs\mu_2(\A)$, 
the function
$h \mapsto |\varphi(h \epsilon_t)|^2$ is integrable on $\SO(V_{2n})(\F) \bs \SO(V_{2n})(\A)$.
\par

Since $\varphi(g)$ is right $K$-finite, 
there exists a finite set $S$ of finite places of $\F$ containing
all infinite places of $\F$ such that
$\varphi$ is right $\epsilon_t$-invariant for any $t \in \prod_{v \not\in S}\mu_2(\F_v)$.
Since $\mu_2(\F_\infty)$ is finite, we see that
$\prod_{v \not\in S}\mu_2(\F_v)$ is open in $\mu_2(\A)$
(not only in $\mu_2(\A_{\fin})$).
This implies that 
\[
\int_{\SO(V_{2n})(\F) \bs \SO(V_{2n})(\A)}|\varphi(h \epsilon_t)|^2dh < \infty
\]
for some (hence any) $t \in \prod_{v \not\in S}\mu_2(\F_v)$.
In particular, $\Res(\varphi)$ is square-integrable.
\end{proof}

\begin{prop}\label{surj}
We have $\Res(\AA_2(\Oo(V_{2n}))) = \AA_2(\SO(V_{2n}))$.
\end{prop}
\begin{proof}
Let $\varphi_0 \in \AA_2(\SO(V_{2n}))$.
Since $\varphi_0$ is $K_0$-finite, 
there exists a compact open subgroup $K_1$ of $K_0 \cap \SO(V_{2n})(\A_{\fin})$ such that
$\varphi_0$ is right $K_1$-invariant.
We may assume that $K_1$ is of the form $K_1=\prod_{v<\infty}K_{1,v}$
for some compact open subgroup $K_{1,v}$ of $K_{0,v}$ such that 
$\epsilon_v^{-1}K_{1,v}\epsilon_v = K_{1,v}$ for any $v < \infty$.
Moreover, we can find a finite set $S$ of places of $\F$ containing
all infinite places of $\F$ such that $K_{1,v} = K_{0,v}$ for any $v \not \in S$.
We fix a complete system $B$ of representative of 
\[
\mu_2(\F) \bs \left(\prod_{v \in S}\mu_2(\F_v)\right).
\] 
We may assume that $B$ contains the identity element $1 \in \prod_{v \in S}\mu_2(\F_v)$.
\par

We regard $\varphi_0$ as a function on $\Oo(V_{2n})(\F) \cdot \SO(V_{2n})(\A)$, 
which is left $\Oo(V_{2n})(\F)$-invariant.
For $t \in \mu_2(\A)$, we define a function
$\varphi_t \colon \Oo(V_{2n})(\F) \cdot \SO(V_{2n})(\A) \rightarrow \C$ by
\[
\varphi_t(h) = \left\{
\begin{aligned}
&\varphi_0(h)				\iif (t_v)_{v \in S} \in B,\\
&\varphi_0(h \epsilon)		\iif (t_v)_{v \in S} \not\in B
\end{aligned}
\right.
\]
for $h \in \Oo(V_{2n})(\F) \cdot \SO(V_{2n})(\A)$.
Then we see that
\[
\varphi_{ta}(h) = \varphi_t(h)
,\quad
\varphi_{-t}(h) = \varphi_t(h \epsilon)
\]
for $t \in \mu_2(\A)$ and $a \in \prod_{v \not \in S}\mu_2(\F_v)$ since $\epsilon^2 = \1_{V_{2n}}$.
Moreover, $\varphi_t$ is right $K_1$-invariant for any $t \in \mu_2(\A)$.
\par

Now we define a function $\varphi \colon \Oo(V_{2n})(\A) \rightarrow \C$ by 
\[
\varphi(g) = \varphi_{\det(g)}(g\epsilon_{\det(g)}^{-1})
\]
for $g \in \Oo(V_{2n})(\A)$.
Then we have $\Res(\varphi) = \varphi_0$.
We show that $\varphi \in \AA_2(\Oo(V_{2n}))$.
\par

Let $\gamma \in \Oo(V_{2n})(\F)$. 
If $\det(\gamma) = 1$, we have
\[
\varphi(\gamma g) = \varphi_{\det(g)}(\gamma g \epsilon_{\det(g)}^{-1})
=\varphi_{\det(g)}(g \epsilon_{\det(g)}^{-1}) =\varphi(g).
\]
If $\det(\gamma) = -1$, we have
\[
\varphi(\gamma g) = \varphi_{-\det(g)}(\gamma g \epsilon_{-\det(g)}^{-1})
=\varphi_{\det(g)}(g \epsilon_{-\det(g)}^{-1}\epsilon) =\varphi(g).
\]
Hence $\varphi$ is left $\Oo(V_{2n})(\F)$-invariant.
It is easy to see that 
$\varphi$ is right $(K_1 \cdot \prod_{v \not \in S}K_v)$-invariant and 
is a $C^\infty$-function on $\Oo(V_{2n})(\F_\infty)$.
Hence $\varphi$ is a smooth function on $\Oo(V_{2n})(\A)$.
\par

We denote the space spanned by $k\cdot \varphi$ for $k \in K$ 
(\resp for $k \in K_0$) by $K\varphi$ (\resp $K_0 \varphi$).
Since any $\varphi' \in K\varphi$ is right $\prod_{v \not \in S}K_v$-invariant, 
the finiteness of $\dim (K\varphi)$ is equivalent to the one of $\dim (K_0\varphi)$.
So we shall prove that $\dim (K_0\varphi)<\infty$.
Let 
\[
\varphi' = \sum_{i=1}^r c_i (k_i \cdot \varphi) \in K_0\varphi
\]
with $c_i \in \C$ and $k_i \in K_0$.
Then for $a \in \prod_{v \in S}\mu_2(\F_v)$ and $x \in \SO(V_{2n})(\A)$, we have
\[
(\epsilon_a \cdot \varphi')(x) = \varphi'(x\epsilon_a)
= \sum_{i=1}^r c_i \varphi(x \epsilon_a k_i)
= \sum_{i=1}^r c_i \varphi(x (\epsilon_a k_i \epsilon_a^{-1}) \epsilon_a).
\]
Since $\epsilon_a k_i \epsilon_a^{-1} \in K_0$, we have
\[
(\epsilon_a \cdot \varphi')|\SO(V_{2n})(\A) \in K_0((\epsilon_a \cdot \varphi)|\SO(V_{2n})(\A)).
\]
Hence we may consider the map
\[
\Phi \colon K_{0}\varphi \rightarrow \bigoplus_{a} 
K_{0}((\epsilon_a \cdot \varphi)|\SO(V_{2n})(\A)),\ 
\varphi' \mapsto 
\oplus_a((\epsilon_a \cdot \varphi')|\SO(V_{2n})(\A))_a,
\]
where $a$ runs over $\prod_{v \in S}\mu_2(\F_v)$.
Since any $\varphi' \in K_0\varphi$ is right $\prod_{v \not \in S}K_v$-invariant and
the map
\[
\prod_{v \in S}\mu_2(\F_v) \rightarrow \SO(V_{2n})(\A) \bs \Oo(V_{2n})(\A)/\prod_{v \not \in S}K_v,\ 
a \mapsto \epsilon_a
\]
is bijective, 
we see that $\Phi$ is injective.
Since $(\epsilon_a \cdot \varphi')|\SO(V_{2n})(\A) \in \AA(\SO(V_{2n}))$, 
it is $K_0$-finite.
Hence we have $\dim (K_0\varphi) < \infty$, and so that 
we get the $K$-finiteness of $\varphi$.
Similarly, we obtain the $\z$-finiteness of $\varphi$.
Note that $\Lie(\Oo(V_{2n})(\F_\infty)) = \Lie(\SO(V_{2n})(\F_\infty))$.
\par

Now we show that $\norm{\varphi}_{L^2(\Oo(V_{2n}))}<\infty$, 
where $\norm{\cdot}_{L^2(\Oo(V_{2n}))}$ (\resp $\norm{\cdot}_{L^2(\SO(V_{2n}))}$)
is the $L^2$-norm on $\Oo(V_{2n})(\F) \bs \Oo(V_{2n})(\A)$
(\resp $\SO(V_{2n})(\F) \bs \SO(V_{2n})(\A)$).
Let $\Omega$ be the characteristic function on 
$\prod_{v \not\in S}\mu_2(\F_v) \times B \subset \mu_2(\A)$.
Then we have
\begin{align*}
\norm{\varphi}_{L^2(\Oo(V_{2n}))}^2
&=\int_{\Oo(V_{2n})(\F) \bs \Oo(V_{2n})(\A)}|\varphi(g)|^2dg\\
&=\int_{\mu_2(\F)\bs \mu_2(\A)}\left(
\int_{\SO(V_{2n})(\F) \bs \SO(V_{2n})(\A)}|\varphi(h \epsilon_t)|^2dh
\right)dt\\
&=\int_{\mu_2(\A)}\Omega(t) \cdot \left(
\int_{\SO(V_{2n})(\F) \bs \SO(V_{2n})(\A)}|\varphi(h \epsilon_t)|^2dh
\right)dt\\
&=\int_{\mu_2(\A)}\Omega(t) \cdot \left(
\int_{\SO(V_{2n})(\F) \bs \SO(V_{2n})(\A)}|\varphi_0(h)|^2dh
\right)dt\\
&=\vol(\prod_{v \not\in S}\mu_2(\F_v) \times B) \cdot \norm{\varphi_0}_{L^2(\SO(V_{2n}))}^2\\
&=2^{-1}\norm{\varphi_0}_{L^2(\SO(V_{2n}))}^2 < \infty.
\end{align*}
\par

We have shown that 
$\varphi$ is a smooth, $K$-finite, $\z$-finite, square-integrable function on
$\Oo(V_{2n})(\F) \bs \Oo(V_{2n})(\A)$.
Such functions are of moderate growth (see \cite[\S 4.3]{BJ}).
Therefore we conclude that $\varphi \in \AA_2(\Oo(V_{2n}))$.
This completes the proof.
\end{proof}

%\subsection{Near equivalence classes}
\subsection{Near equivalence classes}
Let $\sigma=\otimes_v'\sigma_v$ and $\sigma'=\otimes'_v\sigma_v'$
be two admissible representations of $\HH(\Oo(V_{2n}))$.
We say that $\sigma$ and $\sigma'$ are nearly equivalent if
$\sigma_v \cong \sigma'_v$ for almost all $v$.
In this case, we write $\sigma \sim_{\text{ne}} \sigma'$.
Similarly, let $[\sigma_0]=\otimes_v'[\sigma_{0,v}]$ and $\sigma_0'=\otimes'_v[\sigma'_{0,v}]$
be two equivalence classes of admissible representations of $\HH(\SO(V_{2n}))$.
We say that $[\sigma_0]$ and $[\sigma_0']$ are $\epsilon$-nearly equivalent if 
$[\sigma_{0,v}] = [\sigma'_{0,v}]$, i.e., 
$\sigma'_{0,v}\cong \sigma_{0,v}$ or $\sigma_{0,v}^\epsilon$ for almost all $v$.
In this case, we write $[\sigma_0] \sim_{\text{ne}} [\sigma_0']$.
\par

By a near equivalence class in $\AA_2(\Oo(V_{2n}))$, 
we mean a maximal $\HH(\Oo(V_{2n}))$-submodule $\VV$ of $\AA_2(\Oo(V_{2n}))$
such that
\begin{itemize}
\item
all irreducible constituents of $\VV$ are nearly equivalent each other;
\item
any irreducible $\HH(\Oo(V_{2n}))$-submodule of $\AA_2(\Oo(V_{2n}))$
which is orthogonal to $\VV$
is not nearly equivalent to the constituents of $\VV$.
\end{itemize}
We define a $\epsilon$-near equivalence class in $\AA_2(\SO(V_{2n}))$ similarly.
\par

We relate $\epsilon$-near equivalence classes in $\AA_2(\SO(V_{2n}))$
with elements in $\Psi_2(\SO(V_{2n}))/\sim_\epsilon$.
\begin{prop}\label{nearSO}
The there exists a canonical bijection
\[
\{\text{$\epsilon$-near equivalence classes in $\AA_2(\SO(V_{2n}))$}\}
\longleftrightarrow
\Psi_2(\SO(V_{2n}))/\sim_\epsilon.
\]
\end{prop}
\begin{proof}
Suppose that $[\sigma_0]$ and $[\sigma_0']$ appear in 
$\AA_2(\SO(V_{2n}))$.
Let $\Sigma$ and $\Sigma' \in \Psi_2(\SO(V_{2n}))/\sim_\epsilon$ be $A$-parameters 
associated to $[\sigma_0]$ and $[\sigma'_0]$, 
i.e., $[\sigma_0] \in \Pi_\Sigma^0$ and $[\sigma_0'] \in \Pi_{\Sigma'}^0$, respectively.
We claim that $[\sigma_0] \sim_{\text{ne}} [\sigma'_0]$ if and only if $\Sigma = \Sigma'$.
Suppose that $[\sigma_0],[\sigma_0'] \in \Pi_\Sigma^0$. 
Note that both of $\sigma_{0,v}$ and $\sigma'_{0,v}$ are unramified for almost all $v$.
By \cite[\S 4.4 Proposition]{M}, for such $v$,
we have
$[\sigma_{0,v}] = [\sigma'_{0,v}]$. 
Hence $[\sigma_0] \sim_{\text{ne}} [\sigma'_0]$.
\par

Conversely, suppose that $[\sigma_0] \sim_{\text{ne}} [\sigma'_0]$.
For each place $v$ of $\F$,
we decompose
\[
\Sigma_v = (\phi_{1,v} \boxtimes S_{d_1}) \oplus \dots \oplus (\phi_{l,v} \boxtimes S_{d_l}), 
\quad
\Sigma'_v = (\phi'_{1,v} \boxtimes S_{d'_1}) \oplus \dots \oplus (\phi'_{l,v} \boxtimes S_{d'_{l'}}), 
\]
where $\phi_{i,v}$ and $\phi'_{j,v}$ are representations of $\WD_{\F_v} = W_{\F_v} \times \SL_2(\C)$.
Note that for almost all $v$,  $\phi_{i,v}$ and $\phi'_{j,v}$ are trivial on $\SL_2(\C)$.
Hence by \cite[\S 4.2 Corollaire]{M}, we have $\Sigma_v \cong \Sigma'_v$ for almost all $v$.
As in \cite[\S 1.3]{Ar}, $\Sigma$ and $\Sigma'$ define 
isobaric sums of representations $\phi_\Sigma$ and $\phi_{\Sigma'}$ which
belong to the discrete spectrum of $\GL(2n)$.
It follows by the generalized strong multiplicity one theorem of Jacquet--Shalika 
([JS2, (4.4) Theorem]) that $\phi_\Sigma \cong \phi_{\Sigma'}$.
Since the map $\Sigma \mapsto \phi_\Sigma$ is injective, 
we conclude that $\Sigma = \Sigma'$.
\end{proof}

\begin{cor}\label{disjoint}
Let $\Sigma, \Sigma' \in \Psi_2(\SO(V_{2n}))/\sim_\epsilon$.
If $\Sigma  \not= \Sigma'$, then $\Pi_\Sigma^0 \cap \Pi_{\Sigma'}^0 = \emptyset$.
\end{cor}
\par

This corollary together with
Arthur's multiplicity formula (Theorem \ref{AMSO}) and the argument in \S \ref{secAMSOO} 
gives the multiplicity in $\AA_2(\SO(V_{2n}))$.

\begin{cor}\label{free}
Assume Conjectures \ref{mult} and \ref{std}.
\begin{enumerate}
\item
For $[\sigma_0] \in \Pi_\Sigma^0$, we have
\[
\dim_\C\Hom_{\HH^\epsilon(\SO(V_{2n}))}([\sigma_0], \AA_2(\SO(V_{2n})))
=\left\{
\begin{aligned}
&m_\Sigma	\iif \iota_c([\sigma_0]) = \ep_\Sigma,\\
&0	\other.
\end{aligned}
\right.
\]
\item
For any irreducible $(\HH(\SO(V_{2n})), \epsilon)$-module $\sigma$, we have
\[
\dim_\C\Hom_{(\HH(\SO(V_{2n})), \epsilon)}(\sigma, \AA_2(\SO(V_{2n}))) \leq 1.
\]
In other words, 
$\AA_2(\SO(V_{2n}))$ is multiplicity-free as an $(\HH(\SO(V_{2n})), \epsilon)$-module. 
\end{enumerate}
In particular, the same properties unconditionally hold for  the tempered part of 
the automorphic discrete spectrum of $\SO(V_{2n})$
\[
\AA_{2,\temp}(\SO(V_{2n})) \coloneqq
\bigoplus_{\Sigma \in \Psi_{2,\temp}(\SO(V_{2n}))/\sim_\epsilon}
\bigoplus_{[\sigma_0] \in \Pi^0_\Sigma(\ep_\Sigma)} m_\Sigma [\sigma_0].
\]
\end{cor}
\begin{proof}
By Arthur's multiplicity formula (Theorem \ref{AMSO}), we have
\[
\AA_2(\SO(V_{2n})) \cong 
\bigoplus_{\Sigma \in \Psi_2(\SO(V_{2n}))/\sim_\epsilon}
\bigoplus_{[\sigma_0] \in \Pi^0_\Sigma(\ep_\Sigma)} m_\Sigma [\sigma_0]
\]
as $\HH^\epsilon(\SO(V_{2n}))$-modules.
Since $\Pi_\Sigma^0 \cap \Pi_{\Sigma'}^0 = \emptyset$ for $\Sigma \not= \Sigma'$ 
by Corollary \ref{disjoint},
and $\Pi_\Sigma^0$ is a multiplicity-free set by Conjectures \ref{mult} and \ref{std}, 
we have
\[
\dim_\C\Hom_{\HH^\epsilon(\SO(V_{2n}))}([\sigma_0], \AA_2(\SO(V_{2n})))
=\left\{
\begin{aligned}
&m_\Sigma	\iif \iota_c([\sigma_0]) = \ep_\Sigma,\\
&0	\other.
\end{aligned}
\right.
\]
This is (1).
The proof of (2) is similar.
\end{proof}
\par

Recall that by Proposition \ref{surj}, there exists a surjective linear map
\[
\Res \colon \AA_2(\Oo(V_{2n})) \twoheadrightarrow \AA_2(\SO(V_{2n})).
\]
It is easy to see that $\Res$ is an $(\HH(\SO(V_{2n})), \epsilon)$-homomorphism.

\begin{prop}\label{near}
The map $\Res$ induces a bijection
\[
\begin{CD}
\{ \text{near equivalence classes in } \AA_2(\Oo(V_{2n}))\}\\
@VVrV\\
\{ \text{$\epsilon$-near equivalence classes in } \AA_2(\SO(V_{2n}))\}.
\end{CD}
\]
\end{prop}
\begin{proof}
Let $\Phi$ be a near equivalence class in $\AA_2(\Oo(V_{2n}))$. 
Then $\Res(\Phi)$ will a priori meet several $\epsilon$-near equivalence classes. 
Suppose that for an irreducible $\HH(\Oo(V_{2n}))$-module $\sigma$ belonging to $\Phi$, 
$\Res(\sigma)$ contains an irreducible $\HH(\SO(V_{2n}))$-module $\sigma_0$
and $[\sigma_0]$ belongs to 
an $\epsilon$-near equivalence class $\Phi_0$. 
Then we claim that
\[
\Res(\Phi) \subset \Phi_0.
\]
Indeed, suppose that $\sigma' \in \Phi$, $\Res(\sigma') \supset \sigma_0'$, and 
$[\sigma'_0]$ belongs to an $\epsilon$-near equivalence class $\Phi'_0$. 
Then for almost all $v$, we have
$\sigma'_v \cong \sigma_{v}$ so that 
$\sigma'_{0,v} \cong \sigma_{0,v}$ or $\sigma_{0,v}^\epsilon$.
This means that $[\sigma_0] \sim_{\text{ne}} [\sigma'_0]$ so that $\Phi_0 = \Phi'_0$.
Therefore if we define $r(\Phi)=\Phi_0$, then $r$ is well-defined.
\par

For the injectivity, if $r(\Phi) = r(\Phi')$, 
then for $\sigma \in \Phi$ and $\sigma' \in \Phi'$, one has
$\sigma'_v \cong \sigma_v$ or $\sigma'_v \otimes \det$ for almost all $v$.
Since $\sigma_v$ and $\sigma'_v$ are unramified with respect to $K_v$ for almost all $v$, 
we must have $\sigma'_v \cong \sigma_v$ for almost all $v$
(see Lemma \ref{unram}).
Hence we have $\Phi = \Phi'$.
\par

For the surjectivity, we decompose
\[
\AA_2(\Oo(V_{2n})) \cong \bigoplus_\lam \sigma_\lam
\]
into a direct sum of irreducible $\HH(\Oo(V_{2n}))$-modules.
Then we may decompose 
\[
\sigma_\lam \cong \bigoplus_\kappa \sigma_{\lam, \kappa}
\]
into a direct sum of irreducible $\HH(\SO(V_{2n}))$-modules.
Since $\Res \colon \AA_2(\Oo(V_{2n})) \twoheadrightarrow \AA_2(\SO(V_{2n}))$
is a surjective $\HH(\SO(V_{2n}))$-homomorphism, 
any irreducible $\HH(\SO(V_{2n}))$-submodule $\sigma_0$ of $\AA_2(\SO(V_{2n}))$
is isomorphic to some $\sigma_{\lam,\kappa}$ via $\Res$.
Then we have $\Res(\sigma_\lam) \supset \sigma_0$.
This shows that
If $\Phi_0$ (\resp $\Phi_\lam$) is the $\epsilon$-near equivalence class of $\sigma_0$
(\resp the near equivalence class of $\sigma_\lam$), 
then $r(\Phi_\lam) = \Phi_0$.
\end{proof}

%\subsection{Proof of Theorem \ref{main3}}
\subsection{Proof of Theorem \ref{main3}}
In this subsection, we will complete the proof of Theorem \ref{main3} and 
show Proposition \ref{mult1}.
\par

Recall that
$\Psi_2(\Oo(V_{2n})) = \Psi_2(\SO(V_{2n}))/\sim_\epsilon$.
By Propositions \ref{nearSO} and \ref{near}, 
we obtain a canonical bijection
\[
\{\text{near equivalence classes in $\AA_2(\Oo(V_{2n}))$}\}
\longleftrightarrow
\Psi_2(\Oo(V_{2n})).
\]
In other words, we obtain a decomposition
\[
\AA_2(\Oo(V_{2n})) = \bigoplus_{\Sigma \in \Psi_2(\Oo(V_{2n}))}\AA_{2,\Sigma}, 
\]
where $\AA_{2,\Sigma}$ is the direct sum over the near equivalence class 
corresponding to $\Sigma$.
Moreover, we have
\[
\AA_{2,\Sigma} = \bigoplus_{\sigma \in \Pi_\Sigma} m(\sigma) \sigma
\]
for some $m(\sigma) \in \Z_{\geq0}$.
We have to show that
\[
m(\sigma) = \left\{
\begin{aligned}
&1	\iif \iota_c(\sigma) \circ \Delta = \ep_\Sigma,\\
&0	\other.
\end{aligned}
\right.
\]
\par

Consider the restriction map 
$\Res \colon \AA_2(\Oo(V_{2n})) \rightarrow \AA_2(\SO(V_{2n}))$.
This is an $(\HH(\SO(V_{2n})), \epsilon)$-equivariant homomorphism.

\begin{lem}\label{resirr}
Assume Conjectures \ref{mult} and \ref{std}.
Let $\sigma$ be an irreducible $\HH(\Oo(V_{2n}))$-submodule of $\AA_2(\Oo(V_{2n}))$.
Then $\Res(\sigma)$ is nonzero and
irreducible as an $(\HH(\SO(V_{2n})), \epsilon)$-module.
\end{lem}
\begin{proof}
It is clear that $\Res(\sigma)$ is nonzero.
Decompose $\sigma \cong \otimes_v \sigma_v$.
If $\sigma_v \otimes \det \not\cong \sigma_v$ for any $v$, 
then $\sigma$ is irreducible as an $(\HH(\SO(V_{2n})), \epsilon)$-module.
Hence so is $\Res(\sigma)$.
\par

We may assume that $\sigma_v \otimes \det \cong \sigma_v$ for some $v$.
This is in the case (B) as in \S \ref{secAMSOO}.
We decompose
\[
\sigma \cong \bigoplus_{\lam \in \Lam}\sigma_\lam
\]
into a direct sum of irreducible $(\HH(\SO(V_{2n})), \epsilon)$-modules.
Here, $\Lam$ is an index set.
Then $\Res(\sigma) \cong \oplus_{\lam \in \Lam_0} \sigma_\lam$
for some non-empty subset $\Lam_0$ of $\Lam$.
As an $\HH^\epsilon(\SO(V_{2n}))$-module, each $\sigma_\lam$ is a direct sum of 
two copies of an irreducible $\HH^\epsilon(\SO(V_{2n}))$-module $[\sigma_0]$.
Hence $\Res(\sigma) \cong [\sigma_0]^{\oplus2|\Lam_0|}$ 
as $\HH^\epsilon(\SO(V_{2n}))$-modules.
By Corollary \ref{free} (which we have shown using Conjectures \ref{mult} and \ref{std}), 
we have
\[
\dim_\C \Hom_{\HH^\epsilon(\SO(V_{2n}))}([\sigma_0], \AA_2(\SO(V_{2n}))) \leq 2.
\]
This implies that $|\Lam_0| = 1$.
Hence $\Res(\sigma)$ is irreducible as an $(\HH(\SO(V_{2n})), \epsilon)$-module.
\end{proof}

By Arthur's multiplicity formulas for $\SO(V_{2n})$ and $\SO(V_{2n})(\A) \cdot \Oo(V_{2n})(\F)$
(Theorems \ref{AMSO} and \ref{AMSOO}), we see that
if $\sigma \in \Pi_\Sigma$ occurs in $\AA_2(\Oo(V_{2n}))$, 
then $\Res(\sigma) \not= 0$ so that $\iota_c(\sigma)\circ \Delta = \ep_\Sigma$. 
In other words, if $\sigma \in \Pi_\Sigma$ satisfies $\iota_c(\sigma)\circ \Delta \not= \ep_\Sigma$, 
then $\sigma$ does not occur in $\AA_2(\Oo(V_{2n}))$, i.e., $m(\sigma)=0$.
\par

Now we consider the case when $\iota_c(\sigma)\circ \Delta = \ep_\Sigma$.
\begin{prop}\label{m>0}
Assume Conjectures \ref{mult} and \ref{std}.
Let $\sigma = \otimes_v\sigma_v \in \Pi_\Sigma$ 
such that $\iota_c(\sigma)\circ \Delta = \ep_\Sigma$.
Then there exists an $\HH(\Oo(V_{2n}))$-subspace $\AA_\sigma$ of $\AA_2(\Oo(V_{2n}))$
on which $\HH(\Oo(V_{2n}))$ acts by $\sigma$.
\end{prop}
\begin{proof}
Let $[\sigma_0] \in \Pi_\Sigma^0$ be an element satisfying 
$\sigma_{0,v} \subset \sigma_v|\SO(V_{2n})(\F_v)$ for each $v$.
By Arthur's multiplicity formula (Theorem \ref{AMSO}), 
we see that $[\sigma_0]$ occurs in $\AA_2(\SO(V_{2n}))$ 
as an $\HH^\epsilon(\SO(V_{2n}))$-module.
We may assume that $\sigma_0$ occurs in $\AA_2(\SO(V_{2n}))$
as an $\HH(\SO(V_{2n}))$-module.
Since $\Res \colon \AA_2(\Oo(V_{2n})) \rightarrow \AA_2(\SO(V_{2n}))$ is surjective 
(Proposition \ref{surj}),
and $\AA_2(\Oo(V_{2n}))$ is a direct sum of irreducible $\HH(\Oo(V_{2n}))$-modules,
we can find an irreducible $\HH(\Oo(V_{2n}))$-module $\sigma' = \otimes_v \sigma'_v$ 
such that $\sigma'$ occurs in $\AA_2(\Oo(V_{2n}))$ and $\Res(\sigma') \supset \sigma_0$.
By (the proof of) Proposition \ref{near}, we see that 
$\sigma' \in \Pi_\Sigma$ and 
\[
\sigma' \cong \sigma \otimes {\det}_S
\]
for some finite set $S$ of places of $\F$.
Here, ${\det}_S$ is the determinant for places in $S$ and trivial outside $S$.
We consider 3 cases (A), (B) and (C) as in \S \ref{secAMSOO} separately.
\par

We consider the case (A).
Suppose that $\Sigma \in \Psi_2^\epsilon(\Oo(V_{2n}))$.
Then $[\Sc_\Sigma : \Sc_\Sigma^+]=2$ and $[\Sc_{\Sigma,v} : \Sc_{\Sigma,v}^+]=2$
for any place $v$ of $\F$.
By Conjectures \ref{mult} and \ref{std}, we see that 
\[
\iota_c(\sigma') = \iota_{c}(\sigma) \otimes \left(\prod_{v \in S}\eta_{0,v}\right), 
\]
where $\eta_{0,v}$ is the non-trivial character of $\Sc_{\Sigma,v}/\Sc_{\Sigma,v}^+$.
Hence we have
\[
\iota_c(\sigma') \circ \Delta = (\iota_c(\sigma)\circ \Delta) \cdot \eta_0^{|S|}
= \epsilon_\Sigma \cdot \eta_0^{|S|}, 
\]
where $\eta_0$ is the non-trivial character of $\Sc_\Sigma/\Sc_\Sigma^+$.
Since $\Res(\sigma')$ occurs in $\AA_2(\SO(V_{2n}))$ as 
an $(\HH(\SO(V_{2n})), \epsilon)$-module, by Theorem \ref{AMSOO}, 
the number of $S$ must be even.
Hence if $\varphi' \in \AA_2(\Oo(V_{2n}))$, then
\[
\varphi(g) \coloneqq {\det}_S(g) \cdot \varphi'(g) \in \AA_2(\Oo(V_{2n})).
\]
This implies that $\sigma$ also occurs in $\AA_2(\Oo(V_{2n}))$.
\par

We consider the case (B).
Suppose that $\Sigma \not\in \Psi_2^\epsilon(\Oo(V_{2n}))$
and $\sigma_{0}^\epsilon \not\cong \sigma_{0}$.
Then $\sigma_{v} \otimes {\det}_v \cong \sigma_v$ for some $v$.
Hence we can take $S$ such that $|S|$ is even, and we see that 
$\sigma$ also occurs in $\AA_2(\Oo(V_{2n}))$.
\par

We consider the case (C).
Suppose that $\Sigma \not\in \Psi_2^\epsilon(\Oo(V_{2n}))$
and $\sigma_{0}^\epsilon \cong \sigma_{0}$.
Then there are two extension of $\sigma_0$ to 
irreducible $(\HH(\SO(V_{2n})), \epsilon)$-modules, 
and both of them occur in $\AA_2(\SO(V_{2n}))$.
Note that 
$\Res(\sigma') \cong \Res(\sigma' \otimes {\det}_S)$ as 
$(\HH(\SO(V_{2n})), \epsilon)$-modules
if and only if $|S|$ is even.
This implies that 
$\sigma \cong \sigma' \otimes {\det}_S$ occurs in $\AA_2(\Oo(V_{2n}))$ for any $S$.
This completes the proof.
\end{proof}

Under assuming Conjectures \ref{mult} and \ref{std}, 
by a similar argument to Corollary \ref{free}, we have
\[
m(\sigma) = 
\dim_\C\Hom_{\HH(\Oo(V_{2n}))}(\sigma, \AA_2(\Oo(V_{2n}))).
\]
Hence Proposition \ref{m>0} says that if $\iota_c(\sigma) \circ \Delta = \ep_\Sigma$, 
then $m(\sigma)>0$.
To prove $m(\sigma) = 1$ for such $\sigma$, it suffices to show Proposition \ref{mult1}.

\begin{proof}[Proof of Proposition \ref{mult1}]
Let $\sigma$ be an irreducible $\HH(\Oo(V_{2n}))$-module.
We have to show that 
\[
\dim_\C\Hom_{\HH(\Oo(V_{2n}))}(\sigma, \AA_2(\Oo(V_{2n}))) \leq 1.
\]
First we assume that $\sigma$ remains irreducible as an $(\HH(\SO(V_{2n})), \epsilon)$-module.
Then by Lemma \ref{resirr}, the map
\[
\Hom_{\HH(\Oo(V_{2n}))}(\sigma, \AA_2(\Oo(V_{2n})))
\rightarrow
\Hom_{(\HH(\SO(V_{2n})), \epsilon)}(\sigma, \AA_2(\SO(V_{2n}))), 
\ f \mapsto \Res \circ f
\]
is injective.
Hence by Corollary \ref{free}, we have
\[
\dim_\C\Hom_{\HH(\Oo(V_{2n}))}(\sigma, \AA_2(\Oo(V_{2n})))
\leq
\dim_\C\Hom_{(\HH(\SO(V_{2n})), \epsilon)}(\sigma, \AA_2(\SO(V_{2n})))
\leq 1.
\]
\par

Now suppose for the sake of contradiction that
$\sigma$ is an irreducible $\HH(\Oo(V_{2n}))$-module such that
\[
\dim_\C\Hom_{\HH(\Oo(V_{2n}))}(\sigma, \AA_2(\Oo(V_{2n}))) \geq 2.
\]
We know that as an $(\HH(\SO(V_{2n})), \epsilon)$-module,  
$\sigma$ is a multiplicity-free sum 
\[
\sigma = \bigoplus_{\lam} \sigma_{\lam}
\]
of irreducible $(\HH(\SO(V_{2n})), \epsilon)$-modules $\sigma_\lam$.
For a fixed nonzero $f_1$ in the above $\Hom$ space, 
we have shown that $\Res(f_1(\sigma))$ is irreducible (Lemma \ref{resirr}), so say 
$\Res(f_1(\sigma))  = \sigma_{\lambda_1}$.
Consider the natural map
\[
\Hom_{\HH(\Oo(V_{2n}))}(\sigma, \AA_2(\Oo(V_{2n}))) \rightarrow 
\Hom_{(\HH(\SO(V_{2n})), \epsilon)} (\sigma_{\lambda_1}, \AA_2(\SO(V_{2n})))
\]
given by 
\[
f \mapsto \Res \circ f \circ \iota_1,
\]
where $\iota_1 \colon \sigma_{\lambda_1} \hookrightarrow \sigma$ is a fixed inclusion.
We have shown that the right hand side has dimension one (Corollary \ref{free}), 
so this map has nonzero kernel, i.e., there is $f_2$ in the left hand side such that 
$0 \not= \Res(f_2(\sigma))  = \sigma_{\lambda_2}$, with $\lambda_1 \ne \lambda_2$.
This shows that $\AA_2(\SO(V_{2n}))$ contains $\sigma_{\lambda_1} \oplus \sigma_{\lambda_2}$ 
as an $(\HH(\SO(V_{2n})), \epsilon)$-module, 
which implies that $\AA_2(\SO(V_{2n}))$ contains $4 [\sigma_0]$ 
as an $\HH^{\epsilon}(\SO(V_{2n}))$-module.
This contradicts Corollary \ref{free} (1).
This completes the proof.
\end{proof}

By Propositions \ref{mult1} and \ref{m>0}, 
we see that if $\iota_c(\sigma) \circ \Delta = \ep_\Sigma$, then $m(\sigma)>0$.
This completes the proof of Theorem \ref{main3}.

%\begin{thebibliography}
%\begin{thebibliography}


\begin{thebibliography}{ABCD,99}

\bibitem[AGRS]{AGRS}
{A. Aizenbud, D. Gourevitch, S. Rallis and G. Schiffmann}, 
{\em Multiplicity one theorems},
{\it Ann. of Math}. (2) {\bf 172} (2010), no.~2, 1407--1434.

\bibitem[AMR]{AMR}
{N. Arancibia, C. Moeglin and D. Renard},
{\em Paquets d'Arthur des groupes classiques et unitaires}, 
arXiv:1507.01432v1.

\bibitem[Ar]{Ar}
{J. Arthur},  
{\em The endoscopic classification of representations. Orthogonal and symplectic groups}, 
{\it American Mathematical Society Colloquium Publications}, {\bf 61}.

\bibitem[At1]{At}
{H. Atobe},
{\em The local theta correspondence and 
the local Gan--Gross--Prasad conjecture for the symplectic-metaplectic case},
arXiv:1502.03528v1.

\bibitem[At2]{At2}
{H. Atobe},
{\em On the uniqueness of generic representations in an $L$-packet},
arXiv:1511.08897v1.


\bibitem[AG]{AG}
{H. Atobe and W. T. Gan}, 
{\em Local Theta correspondence of Tempered Representations and Langlands parameters}, 
preprint.

\bibitem[BJ]{BJ}
{A. Borel and H. Jacquet}, 
{\em Automorphic forms and automorphic representations}, 
{\it Proc. Sympos. Pure Math}. {\bf 33} 
Part 1, pp. 189--207, Amer. Math. Soc., Providence, R.I., 1979.

\bibitem[CL1]{CL1}
{P.-H. Chaudouard and G. Laumon}, 
{\em Le lemme fondamental pond{\'e}r{\'e}. $\mathrm{I}$. Constructions g{\'e}om{\'e}triques},
{\it Compos. Math}. {\bf146} (2010), no.~6, 1416--1506.

\bibitem[CL2]{CL2}
{P.-H. Chaudouard and G. Laumon}, 
{\em Le lemme fondamental pond{\'e}r{\'e}. $\mathrm{II}$. {\'E}nonc{\'e}s cohomologiques},
{\it Ann. of Math}. (2) {\bf176} (2012), no.~3, 1647--1781. 

\bibitem[CS]{CS}
{W. Casselman and J. Shalika}, 
{\em The unramified principal series of $p$-adic groups. $\mathrm{II}$. The Whittaker function}, 
{\it Compos. Math}. {\bf 41} (1980), 207--231.

\bibitem[GGP]{GGP}
{W. T. Gan, B. H. Gross and D. Prasad}, 
{\em Symplectic local root numbers, central critical $L$-values, and restriction problems in the representation theory of classical groups. }
{Sur les conjectures de Gross et Prasasd. $\mathrm{I}$.}
{\it Ast\'{e}risque.} No.~{\bf 346} (2012), 1--109.

\bibitem[GI1]{GI1}
{W. T. Gan and A. Ichino}, 
{\em Formal degrees and local theta correspondence}. 
{\it Invent. Math}. {\bf195} (2014), no. 3, 509--672. 

\bibitem[GI2]{GI2}
{W. T. Gan and A. Ichino}, 
{\em The Gross--Prasad conjecture and local theta correspondence},
arXiv:1409.6824v2.

\bibitem[GP]{GP}
{B. H. Gross and D. Prasad},
{\em On the decomposition of a representation of $\SO_n$ when restricted to $\SO_{n-1}$},
{\it Canad. J. Math}. {\bf 44} (1992), no.~5, 974--1002.

\bibitem[GS]{GS}
{W. T. Gan and G. Savin}, 
{\em Representations of metaplectic groups I: epsilon dichotomy and local Langlands correspondence},
{\it Compos. Math.} {\bf 148} (2012), 1655--1694.

\bibitem[GT1]{GT1}
{W. T. Gan and S. Takeda},
{\em On the Howe duality conjecture in classical theta correspondence},
arXiv:1405.2626v3.

\bibitem[GT2]{GT2}
{W. T. Gan and S. Takeda},
{\em A proof of the Howe duality conjecture},
arXiv:1407.1995v4.

\bibitem[H]{H}
{V. Heiermann}
{\em A note on standard modules and Vogan $L$-packets}
arXiv:1504.04524v1.

\bibitem[JS]{JS}
{H. Jacquet and J. A. Shalika}, 
{\em On Euler products and the classification of automorphic representations. $\mathrm{I}$}, 
{\it Amer. J. Math}. {\bf 103} (1981), 499--558.

\bibitem[Ka1]{Ka}
{T. Kaletha}, 
{\em Genericity and contragredience in the local Langlands correspondence}, 
{\it Algebra Number Theory} {\bf7} (2013), no.~10, 2447--2474.

\bibitem[Ka2]{Ka2}
{T. Kaletha}, 
{\em Rigid inner forms of real and $p$-adic groups}, 
arXiv:1304.3292v5.

\bibitem[KMSW]{KMSW}
{T. Kaletha, A. Minguez, S. W. Shin and P.-J. White},
{\em Endoscopic classification of representations: Inner forms of unitary groups},
arXiv:1409.3731v3.

\bibitem[Ku]{Ku}
{S. S. Kudla}, 
{\em On the local theta correspondence}, 
{\it Invent. Math}. {\bf 83} (1986), 229--255.

\bibitem[LR]{LR}
{E. M. Lapid and S. Rallis}, 
{\em On the local factors of representations of classical groups}, 
Automorphic representations, $L$-functions and applications: progress and prospects, 
{\it Ohio State Univ. Math. Res. Inst. Publ}. {\bf11}, de Gruyter, Berlin, 2005, pp. 309--359.

\bibitem[M1]{M}
{C.M\oe glin}, 
{\em Comparaison des param\`etres de Langlands et des exposants \`a l'int\'erieur d'un paquet d'Arthur},
{\it J. Lie Theory} {\bf19} (2009), no.~4, 797--840.

\bibitem[M2]{M2}
{C.M\oe glin}, 
{\em Multiplicit\'{e} $1$ dans les paquets d'Arthur aux places $p$-adiques}, 
{\it On certain L-functions, Clay Math.~Proc}., vol.~{\bf 13}, Amer.~Math.~Soc., 
Providence, RI, 2011, pp. 333--374.

\bibitem[M3]{M3}
{C.M\oe glin}, 
{\em Image des op{\'e}rateurs d'entrelacements normalis{\'e}s eta p{\^o}les des s{\'e}ries d'Eisenstein
}, 
{\it Adv. Math}. {\bf228} (2011), no.~2, 1068--1134.

\bibitem[MVW]{MVW}
{C.M\oe glin, M.-F. Vigneras and J.-L. Waldspurger}, 
{\em Correspondances de Howe sur un corps $p$-adique}, 
Lecture Notes in Mathematics, {\bf1291}, 
{\it Springer-Verlag, Berlin}. 1987.

\bibitem[MW]{MW}
{C.M\oe glin and J.-L.~Waldspurger}, 
{\em La conjecture locale de Gross--Prasad pour les groupes sp\'eciaux orthogonaux: 
le cas g\'en\'eral},
Sur les conjectures de Gross et Prasad. $\mathrm{II}$. 
{\it Ast\'erisque}. No.~{\bf 347} (2012), 167--216.

\bibitem[MW, $\mathrm{VI}$]{Stab6}
{C.M\oe glin and J.-L. Waldspurger}, 
{\em Stabilisation de la formule des traces tordue $\mathrm{VI}$: 
la partie g\'eom\'etrique de cette formule},
arXiv:1406.2257.

\bibitem[MW, $\mathrm{X}$]{Stab10}
{C.M\oe glin and J.-L. Waldspurger}, 
{\em Stabilisation de la formule des traces tordue $\mathrm{X}$: 
stabilisation spectrale},
arXiv:1412.2981.

\bibitem[P1]{P1}
{D. Prasad},
{\em Trilinear forms for representations of $\GL(2)$ and local $\epsilon$-factors}, 
{\it Compositio Math}. {\bf75} (1990), no.~1, 1--46.

\bibitem[P2]{P}
{D. Prasad},
{\em On the local Howe duality correspondence},
{\it Internat. Math. Res. Notices} (1993), 279--287.

\bibitem[P3]{P3}
{D. Prasad},
{\em Some applications of seesaw duality to branching laws},  
{\it Math. Ann}. {\bf 304} (1996), no.~1, 1--20. 

\bibitem[P4]{P4}
{D. Prasad},
{\em Relating invariant linear form and local epsilon factors via global methods. With an appendix by Hiroshi Saito},  
{\it Duke Math. J}. {\bf138} (2007), no.~2, 233--261.

\bibitem[RS]{RS}
{Z. Rudnick and P. Sarnak}, 
{\em Zeros of principal L-functions and random matrix theory}, 
{\it Duke Math. J}. {\bf 81} (1996), 269--322.

\bibitem[S1]{S1}
{F. Shahidi},
{\em On certain L-functions}, 
{\it Amer. J. Math}. {\bf103} (1981), no.~2, 297--355.

\bibitem[S2]{S2}
{F. Shahidi},
{\em A proof of Langlands' conjecture on Plancherel measures; Complementary series for $p$-adic groups}, 
{\it Annals of Math}. {\bf132} (1990), 273--330.

\bibitem[Ta]{Ta}
{O. Ta\"{i}bi}, 
{\em Arthur's multiplicity formula for certain inner forms of special orthogonal and symplectic groups},
arXiv:1510.08395.

\bibitem[T]{T}
{J. Tate}, 
{\em Number theoretic background}. 
Automorphic forms, representations and $L$-functions.
{\it Proc.~Sympos.~Pure Math., Oregon State Univ., Corvallis, Ore.,} 1979, Part 2, pp. 3--26. 

\bibitem[V]{V}
{G. A. Vogan}, 
{\em The local Langlands conjecture},
{\it Representation theory of groups and algebras}, 305--379, 
{\it Contemp. Math}. {\bf 145}, Amer. Math. Soc. Providence, RI, 1993.

\bibitem[W1]{W1}
{J.-L. Waldspurger}, 
{\em Demonstration d'une conjecture de dualite de Howe dans le cas $p$-adique, $p\not= 2$}, 
{\it in Festschrift in honor of I. I. Piatetski-Shapiro on the occasion of his sixtieth birthday}, 
part $\mathrm{I}$ (Ramat Aviv, 1989), Israel Mathematical Conference Proceedings, vol. 2 (Weizmann, Jerusalem, 1990), 267--324.

\bibitem[W2]{W2}
{J.-L. Waldspurger},
{\em Une formule int\'egrale reli\'ee \`a la conjecture locale de Gross--Prasad},
{\it Compos. Math}. {\bf 146} (2010), no.~5, 1180--1290. 

\bibitem[W3]{W3}
{J.-L. Waldspurger}, 
{\em Une formule int\'egrale reli\'ee \`a la conjecture locale de Gross--Prasad, 
2e partie: extension aux repr\'esentations temp\'er\'ees},
{\it Ast\'erisque}. No.~{\bf 346} (2012), 171--312.

\bibitem[W4]{AGRS2}
{J.-L. Waldspurger}, 
{\em Une variante d'un r\'{e}sultat de Aizenbud, Gourevitch, Rallis et Schiffmann}, 
{\it Ast\'erisque}. No.~{\bf 346} (2012), 313--318.

\bibitem[W5]{W4}
{J.-L. Waldspurger}, 
{\em Calcul d'une valeur d'un facteur $\ep$ par une formule int\'egrale},
{\it Ast\'erisque}. No.~{\bf 347} (2012), 1--102.

\bibitem[W6]{W5}
{J.-L. Waldspurger}, 
{\em La conjecture locale de Gross--Prasad pour 
les repr\'esentations temp\'er\'ees 
des groupes sp\'eciaux orthogonaux}, 
{\it Ast\'erisque}. No.~{\bf 347} (2012), 103--165.

\bibitem[W, $\mathrm{I}$]{Stab1}
{J.-L. Waldspurger}, 
{\em Stabilisation de la formule des traces tordue $\mathrm{I}$: 
endoscopie tordue sur un corps local}, 
arXiv:1401.4569.

\bibitem[W, $\mathrm{II}$]{Stab2}
{J.-L. Waldspurger}, 
{\em Stabilisation de la formule des traces tordue $\mathrm{II}$: 
int\'egrales orbitales et endoscopie sur un corps local non-archim\'edien; 
d\'efinitions et \'enonc\'es des r\'esultats},
arXiv:1401.7127.

\bibitem[W, $\mathrm{III}$]{Stab3}
{J.-L. Waldspurger}, 
{\em Stabilisation de la formule des traces tordue $\mathrm{III}$: 
int\'egrales orbitales et endoscopie sur un corps local non-archim\'edien; 
r\'eductions et preuves},
arXiv:1402.2753.

\bibitem[W, $\mathrm{IV}$]{Stab4}
{J.-L. Waldspurger}, 
{\em Stabilisation de la formule des traces tordue $\mathrm{IV}$: 
transfert spectral archim\'edien},
arXiv:1403.1454.

\bibitem[W, $\mathrm{V}$]{Stab5}
{J.-L. Waldspurger}, 
{\em Stabilisation de la formule des traces tordue $\mathrm{V}$: 
int\'egrales orbitales et endoscopie sur le corps r\'eel},
arXiv:1404.2402.

\bibitem[W, $\mathrm{VII}$]{Stab7}
{J.-L. Waldspurger}, 
{\em Stabilisation de la formule des traces tordue $\mathrm{VII}$: 
descente globale},
arXiv:1409.0960.

\bibitem[W, $\mathrm{VIII}$]{Stab8}
{J.-L. Waldspurger}, 
{\em Stabilisation de la formule des traces tordue $\mathrm{VIII}$: 
l'application $\epsilon_{\tilde{M}}$ 
sur un corps de base local non-archim\'edien},
arXiv:1410.1124.

\bibitem[W, $\mathrm{IX}$]{Stab9}
{J.-L. Waldspurger}, 
{\em Stabilisation de la formule des traces tordue $\mathrm{IX}$: 
propri\'et\'es des int\'egrales orbitales pond\'er\'ees $\omega$-\'equivariantes sur le corps r\'eel},
arXiv:1412.2565.

\bibitem[X]{X}
{B. Xu}, 
{\em On M{\oe}glin's parametrization of Arthur packets for $p$-adic quasisplit $\Sp(N)$ and $\SO(N)$},
arXiv:1507.08024v1.

\end{thebibliography}
\end{document}